\documentclass[11pt]{article}
\usepackage{multicol}
\usepackage{newlfont}
\usepackage{amsmath,amssymb}

\def\rit{\mathbb{R}}
\def\zit{\mathbb{Z}}
\def\nit{\mathbb{N}}
\def\ait{\mathbb{A}}

\def\ppit{\mathbb{P}}
\def\qit{\mathbb{Q}}
\def\cit{\mathbb{C}}
\def\fit{\mathbb{F}}

\def\gm{\mathfrak{m}}

\def\lcm{l.c.m}

\newcommand{\pf}{{\em Proof.~}}
\newcommand{\qed}{\hfill~~\mbox{$\Box$}}

\newenvironment{proof}{\smallskip \noindent \pf}{\qed \bigskip}

\newtheorem{theorem}{Theorem}[subsection]
\newtheorem{proposition}[theorem]{Proposition}
\newtheorem{definition}[theorem]{Definition}
\newtheorem{lemma}[theorem]{Lemma}
\newtheorem{corollary}[theorem]{Corollary}
\newtheorem{notation}[theorem]{Notation}
\newtheorem{remark}[theorem]{Remark}
\newtheorem{example}[theorem]{Example}

\begin{document}

\title{\bf Quantum differential systems and some applications to mirror symmetry }
\author{\sc Antoine Douai \thanks{Partially supported by the grants ANR-08-BLAN-0317-01 and ANR-13-IS01-0001-01 of the Agence nationale de la recherche.}\\
Laboratoire J.A Dieudonn\'e, UMR CNRS 7351, \\
Universit\'e de Nice, Parc Valrose, F-06108 Nice Cedex 2, France \\
Email address: Antoine.DOUAI@unice.fr}

\maketitle


\begin{abstract}  We study mirror symmetry (A-side {\em vs} B-side) in the framework of quantum differential systems.
We focuse on the logarithmic non-resonant case, which describes the geometric situation and for which quantum differential systems are produced on the B-side 
by avatars of rescalings of regular tame functions.
We show that quantum differentials systems provide a good framework in order to 
generalize the construction of the rational structure given in \cite{KKP} for the complex projective space. 
As an application, we compute the rational structure obtained in this way on the orbifold cohomology of weighted projective spaces and 
on the flat sections of the Gauss-Manin connection associated with 
their Landau-Ginzburg models (suitable Laurent polynomials).
As an example, and in order to complete the panorama, we also calculate, in the setting of quantum differential systems, a mirror partner of the Hirzebruch surface $\fit_{2}$. 
\end{abstract}

\tableofcontents

\section{Introduction}

This paper deals with {\em quantum differential systems}, namely trivial bundles equipped with a flat meromorphic connection with prescribed poles
 together with a flat nondegenerate bilinear form. We will focuse mainly on their relation with mirror symmetry.

 Such systems already appear, more or less explicitely and under various names, in the work of a lot of people, essentially motivated by the construction 
of Frobenius manifolds, see {\em f.i} \cite{Dub1}, \cite{Dub2}, \cite{Her}, \cite{Mal2}, \cite{Manin}, \cite{SK}, \cite{Sab1}, \cite{DoSa1}, \cite{R}... 
They first arose in singularity theory ($B$-side, local version) thirty years ago in the work of K. Saito
about the primitive forms \cite{SK}, and took the form that we will use in the work of B. Malgrange \cite{Mal1}. 
In connection with the construction of Frobenius manifolds (for which another important ingredient are the {\em primitive forms}),
 a global version of these objects has been discussed in \cite{DoSa1} where it is explained how a regular tame function on an affine manifold yields a 
quantum differential system, naturally produced by solutions of the Birkhoff problem for its Brieskorn lattice (the construction is outlined in the Appendix). 
The tameness condition is required for finitness reasons.

As it follows from Dubrovin's formalism \cite{Dub0}, \cite{Dub1}, \cite{Dub2} (see also \cite{CK} and the references therein), 
quantum differential systems also appear in quantum cohomology theory ($A$-side), giving an extension of the Dubrovin connexion as an absolute flat connection 
and encoding the quantum product and 
its basic properties, taking into account a supplementary homogeneity condition. 

It is thus natural to investigate mirror symmetry through quantum differential systems 
(the step before Frobenius manifolds): two models will be mirror partners if their associated quantum differential systems are isomorphic, 
as bundles with connections. 
Notice that in this setting, 
Givental's quantum differential operators are interpreted as minimal polynomials of suitable (primitive) sections. 

Some motivations are in order:
\begin{itemize}
\item First, and the aim of this paper is to emphasize this point, such systems can be computed on the $B$-side, without any references 
to correlators (and hence to the quantum product), which are rather complicated objects. In this way, mirror symmetry can be useful in order to understand 
more clearly (and sometimes predicts) what happens on the A-side; in turn the A-side produces B-models that are interesting on their own.
A step in this direction can be found in \cite{DoMa},
which gives a counterpart of the computations carried in \cite{Coa} for the small quantum (orbifold) cohomology of weighted projective spaces.
A connected class of examples is given by rescalings of regular tame functions which, despite its trivial appearance, give
a quite good picture of the situation (see section \ref{sec:rescaling}). 
We also discuss the case of the Hirzebruch surface $\fit_{2}$ in section \ref{sec:hirzebruch} where 
Givental's mirror map \cite{Giv1} 
appears naturally as a function in flat coordinates where 
flatness has to be understood with respect to a flat residual connection, 
naturally produced by the quantum differential systems involved. This flat connection is a central object for our purpose 
because, in mirror symmetry, flat coordinates are: on the A-side, coordinates are indeed flat. More generally, these techniques could be used for instance 
in order to study hypersurfaces or complete intersections 
in weighted projective spaces, \cite{Giv}, \cite{HV}, \cite{VP}. 
\item Also, a quantum differential system is a very flexible object: for instance, and as emphasized in different papers \cite{Dub1}, \cite{Mal1}, \cite {HeMa},
it can be universally unfolded in some cases and we can modify accordingly its base space, which can be the affine space, a torus (algebraic setting)
 or a punctual germ (analytic setting). In other words, a whole quantum differential system can be, in some cases, 
determined by a restricted set of data, and this observation is very useful in order to simplify the computations on the B-side, see \cite{D1}, \cite{D2}. 
\item It fits very well with ``quantizations'' ({\em f.i} small quantum cohomology)
and it is a good setting in order to study ``large radius limits'', using the classical techniques in differential equation theory. Notice that these limits 
(these are of course quantum differential systems on a point) always produce meromorphic connections with regular singularities.
This is explained in section \ref{nonresonnant} and in section \ref{ex:CohQuant}.
\item Last, it is a natural framework in order to generalize the construction of the rational structure on the $A$-side given in \cite{KKP} for $\ppit^{n}$. 
\end{itemize}

Nevertheless, it should be noticed that, on the $B$-side, a given tame regular function can produce several quantum differential systems which can 
be difficult to compare. While the situation is clear on the $A$-side (the cohomology gives naturally a flat basis and flat coordinates),
 we have to fix, on the $B$-side, some choices: the general principle is that a canonical quantum differential system is built from 
the canonical solution of the Birkhoff problem given 
by M. Saito's method, for which a substantial tool is Hodge theory (see \cite{D2}, \cite[Appendix B]{DoSa1}, \cite{DoSa2} and Appendix). 
Anyway, the general place of this geometric solution in mirror symmetry has to be further explored.\\

Let us now discuss more precisely these motivations. An aspect of mirror symmetry is the following: given a projective manifold $X$, one can computes its Gromov-Witten invariants, or more generally 
its correlators, with the help of Picard-Fuchs equations associated with some mirror partner. 
This is classically used to express these correlators in terms of combinatorial data: this is {\em f.i} what gives Givental's ``I=J'' mirror theorem
(see \cite{CK} for an overview). We explain how this can be achieved using quantum differential systems, in particular  what should be
the correlators of a quantum differential system.
Another and connected goal is to define {\em the} 
$J$-function of a general quantum differential system. Again, we have to fix some choices:
we are led to define {\em canonical} fundamental solutions of the Dubrovin connection of a quantum differential system and this is done using Dubrovin's 
conformal and symmetric solutions. 
The situation is particularly nice when the quantum differential systems are {\em logarithmic} and {\em non-resonant}, see section 
\ref{nonresonnant}: prototypes of such systems 
are given by small quantum cohomology, thanks to the divisor axiom, see also \cite{R}, \cite{RS}. 
In this case, the canonical solutions have an explicit description: they are uniquely determined by a matrix of holomorphic functions, 
satisfying an initial condition. This matrix  can be computed algebrically, using a recursion relation (relation (\ref{eq:DiffxHd})).
We will call its coefficients the  {\em correlators} of the given quantum differential system because, in the case of the small quantum cohomology, 
they provide the usual correlators, see section \ref{sec:correlateurs}. 
In order to compute the correlators of a projective variety, it is thus enough to identify the canonical fundamental solutions of the mirror 
quantum differential system and this is reduced, on the $B$-side, to computations of algebra. This is emphasized in section \ref{sec:hirzebruch}.

Quantum differential systems provide also a good framework in order to generalize
the construction of the rational structure on the cohomology of the complex projective space $\ppit^{n}$ and their Landau-Ginzburg models given in \cite{KKP}. 
The strategy in {\em loc. cit.} is the following: the rational structure is first constructed on the $B$-side 
on the flat sections of the Gauss-Manin connection associated (after quantization) with a suitable regular function, the Landau-Ginzburg model.
At the beginning, this rational structure is provided by the Lefschetz thimbles and then transferred to the flat sections using oscillating integrals.  
In order to get first a precise formula for this rational structure on the $B$-side, one needs an explicit description of these flat sections and this is done 
using the quantum differential system produced 
by the Landau-Ginzburg model.
The rational structure that we get on the $B$-side is then shifted, taking the classical limit, on the
cohomology ($A$-side) using a mirror theorem which identifies the standard cohomology basis with suitable explicit differential forms. Our main purpose is to 
extend this method:
in order to do so, we first define quantum differential systems, their classical limits and their conformal Dubrovin's solutions (see section 3). 
Conformality is used here in order to get a precise description of the flat sections: 
this is discussed in section \ref{sec:RationalStructure} (see proposition \ref{prop:classescar}).   
For our geometric setting, it is enough to consider logarithmic quantum differential systems (see section \ref{nonresonnant}): 
their main properties are given by theorem \ref{cor:casvariete} and corollary \ref{coro:sechor}. It turns out that only {\em flat} 
(in the sense of definition \ref{eq:FlatQDS}) logarithmic quantum differential systems are only relevant. 

As an application of our method, we give in section \ref{subsec:RatWPS} a description of the rational structure obtained in this way on the orbifold 
cohomology of weighted projective spaces (see corollary \ref{coro:descriptionPsiconst}) and their Landau-Ginzburg models (see theorem \ref{theo:Rat}). More precisely,
we first give a closed formula on the $B$-side, for which the mirror quantum 
differential system is identified in \cite{DoMa}. This formula involves
various numbers (depending on the combinatorics), produced by the computation of some relevant oscillating integrals whose integral kernel depend on the choice 
of suitable bases of differential forms. 
This rational structure on the $B$-side is, 
after \cite[Theorem 4.10]{Sab4}, 
an ingredient of a (variation of a pure, rational) non-commutative Hodge structure in the sense of \cite[Definition 2.7]{KKP}, 
related with the ''$\qit$-structure axiom`` (the link 
between Hodge theory and Lefschetz thimbles is a quite old and long story: see for instance \cite[Chapitre III, paragraphes 12 et 14]{AVG} and the references 
therein and also, closer from our concern, \cite{DoSa1}, \cite[section 6]{DoSa2} and \cite{Sab2}).
In order to reach the $A$-side, we then use the explicit description of the mirror partner of the standard orbifold cohomology basis given 
in \cite[Theorem 5.1.1]{DoMa}. Notice that the construction of such structures on the $A$-side
is also considered in \cite{Ir} for toric orbifolds using a completely different approach (in particular we will not make use of equivariant perturbations 
and localization arguments in this paper): up to a ramification (due to the fact that we have to consider flat bases with respect to a residual connection), 
we get at the end Iritani's formula \cite[Theorem 4.11]{Ir} for weighted projective spaces. A striking fact is that a part of the constants in the formula 
for the rational structure that we get on the $B$-side miraculously disappear when we apply the mirror theorem and that, in the end, 
we get a very simple formula for 
the rational structure on the $A$-side, see corollary \ref{coro:descriptionPsiconst}.

Last, and in order to complete the panorama, we compute a mirror partner of the Hirzebruch surface $\fit_{2}$ 
(the classical non-Fano example) using quantum differential systems. This example is very interseting because it produces some new, but also in some sense 
intermediate (between the ones produced by projective space and the ones produced by weighted projective spaces; see for instance section \ref{sec:LogFrob} 
where the construction of a logarithmic Frobenius manifold is also discussed), phenomena. We show how our method allow to recover well-known results, 
see {\em f.i} \cite[Section 11.2]{CK} and \cite[Example 5.4]{Guest}.
In particular, it is readily seen that the change of variables considered there in order to get the ``correct'' quantum product (in other words, the mirror map) is naturally 
given by flat coordinates.
We also verify that the quantum differential system associated with the mirror partner of the projective space $\ppit (1,1,2)$ is 
obtained as a classical limit of the one associated with the mirror of $\fit_{2}$, as it has been first checked in \cite{CIT}.\\

By way of conclusion, let us emphasize the fact that the point of view developped here is in essence not so far from Givental's theory of mirror symmetry and  
``quantum differential equations'' \cite{Giv0}, \cite{Giv1} (roughly speaking, we consider
matrices rather than their characteristic polynomials)  but the technics in order to get a mirror theorem are somewhat different:
our main objective was to show how solutions of the Birkhoff problem for the Brieskorn lattice 
of a regular tame function as defined in \cite{DoSa1} should be naturally exploited in order to understand better (small) quantum cohomology. 
We where motivated by the lecture of \cite{KKP}, \cite{T} and \cite{Ir} about rational structures.\\

This paper is organized as follows: we define quantum differential systems in \ref{sec:SDQ}, and discuss their relationship with mirror symmetry. 
In sections \ref{sec:Dub}, \ref{sec:SolFonda} and \ref{sec:solfondacan}, we define the canonical fundamental solutions and the canonical $J$-functions 
of a quantum differentials system with the help of Dubrovin's conformal and symmetric solutions \cite{Dub2}.
The case of the non-resonant, logarithmic systems is handled in section \ref{nonresonnant} 
and we give some examples in section \ref{ex:CohQuant}. 
We apply the results obtained there in order to describe a rational structure on the orbifold cohomology of weighted projective spaces 
and to compute correlators in sections \ref{sec:RationalStructure} and \ref{sec:correlateurs}.
We explain in section \ref{sec:hirzebruch} how to get an explicit mirror quantum differential system to the small quantum cohomology of 
the Hirzebruch surface. Last, we recall in the Appendix how to construct a quantum differential system from a regular tame function.

Several results (including the discussion about rational structures) presented in these notes
are now published in \cite{D3}.

\section{Quantum differential systems}

\label{sec:SDQ}

We introduce here our main object, the {\em quantum differential systems}
\footnote{They are also called  ``tr.TLEP-structures'' in the work of C. Hertling \cite{HeMa}}. The basic definitions and properties are for instance compiled 
in C. Sabbah's book \cite{Sab1}, using B. Malgrange's setting \cite{Mal1}, \cite{Mal2}. We first list some of them.

\subsection{Definitions}

Let $M$ be a complex analytic manifold, equipped with coordinates $\underline{x}=(x_{0},\cdots ,x_{r})$. We will denote by $U_{0}$ ({\em resp.} $U_{\infty}$) the chart of $\ppit^{1}$ centered at $0$ ({\em resp.} $\infty$) and by $\theta$ ({\em resp.} $\tau:=\theta^{-1}$) the coordinate on $U_{0}$ ({\em resp.} $U_{\infty}$). Let $\pi$ be the projection $\pi :\ppit^{1}\times M\rightarrow M$.

\begin{definition}\label{def:SDQ}
A {\em quantum differential system} on $M$ is a tuple         
$${\cal Q}=(M, {\cal G}, \nabla ,S, d)$$
where
\begin{itemize}
\item $d$ is an integer,
\item ${\cal G}$ is a trivial\footnote{In ``the fibers'' of the projection $\pi$: ${\cal G}\simeq\pi^{*}\pi_{*}{\cal G}$} bundle on $\ppit^{1}\times M$,
\item $\nabla$ is a flat meromorphic connexion on ${\cal G}$, with poles of Poincar\'e rank less or equal to $2$ along $\{0\}\times M$, logarithmic along $\{\infty \}\times M$,
\item  $S$ is a $\nabla$-flat, non-degenerate bilinear form $S:{\cal O}({\cal G})\times j^{*}{\cal O}({\cal G})\rightarrow \theta^{d}{\cal O}_{\ppit^{1}\times M}$, where
$$j:\ppit^{1}\times M\rightarrow :\ppit^{1}\times M$$
is defined by $j(\theta ,\underline{x})=(-\theta ,\underline{x})$.
\end{itemize}
\end{definition}

\noindent In what follows, we will denote by\\

\noindent $\bullet$ $\mu$ the rank of the bundle ${\cal G}$,\\

\noindent $\bullet$ $G_{0}$ the restriction of ${\cal G}$ at $U_{0}$: this is a free ${\cal O}_{M}[\theta]$-module of rank $\mu$,\\

\noindent $\bullet$ $G_{\infty}$ the restriction of ${\cal G}$ at $U_{\infty}$: this is a free ${\cal O}_{M}[\tau]$-module of rank $\mu$.\\

Let ${\cal Q}=(M, {\cal G}, \nabla ,S, d)$ be a quantum differential system, $i_{\{\theta =0\}}$ ({\em resp.} $i_{\{\theta =\infty\}}$) be the zero 
({\em resp.} infinity) section of $M$ in $\ppit^{1}\times M$ and $E:=\pi_{*}{\cal G}$.

\begin{proposition}\label{keyproposition}
(1) We have the isomorphisms $i^{*}_{\{\theta =\infty\}}{\cal G}\simeq E\simeq i^{*}_{\{\theta =0\}}{\cal G}$.\\
(2) The connection $\nabla$ takes the form
$$\nabla =\bigtriangledown +\frac{\Phi}{\theta} +(\frac{V_{0}}{\theta}+V_{\infty})\frac{d\theta}{\theta}$$
where
\begin{itemize}
\item $\bigtriangledown$  is a connection on $E$,
\item $\Phi$ is a Higgs field, that is an ${\cal O}_{M}$-linear map $\Phi: E\rightarrow  E\otimes\Omega_{M}^{1}$, such that $\Phi\wedge\Phi=0$,
\item $V_{0}$ and $V_{\infty}$ are two ${\cal O}_{M}$-linear endomorphisms of $E$,
\end{itemize}
these objects satisfying
\begin{equation}\label{eq:integrabilite}
\bigtriangledown^{2}=0,\  \Phi\wedge\Phi =0,\  \bigtriangledown \Phi =0,\ \bigtriangledown V_{\infty}=0,\  [V_{0},\Phi ]=0\  \bigtriangledown (V_{0})+\Phi =[\Phi, V_{\infty}]
\end{equation}
In particular, the connection $\bigtriangledown$ is flat.
\end{proposition}
\begin{proof} Standard, see {\em f.i} \cite{Sab1}, but we outline it, due to its importance for what follows: the isomorphisms expected in (1) follow from the triviality of the bundle ${\cal G}$ (restriction of sections). 
The assumptions on the order of the poles show that, in a basis $\omega =(\omega_{0},\cdots ,\omega_{\mu -1})$ of $E$, the matrix of $\nabla$ is 
\begin{equation}\label{eq:baseglobale}
(\frac{A_{0}(\underline{x})}{\theta}+A_{\infty}(\underline{x}))\frac{d\theta}{\theta}+C(\underline{x})+\frac{D (\underline{x})}{\theta}
\end{equation}
where $\underline{x}=(x_{0},\cdots ,x_{\mu})\in M$,
$$C(\underline{x})=\sum_{i=1}^{r}C^{(i)}(\underline{x})dx_{i}\ \mbox{et}\ D(\underline{x})=\sum_{i=1}^{r}D^{(i)}(\underline{x})dx_{i}.$$
The connection $\bigtriangledown$ is first defined on $i^{*}_{\{\theta =\infty\}}{\cal G}$ as the restriction at $\tau =0$ of a flat connection. Its matrix in the basis $\omega$ is $C(\underline{x})$. The ${\cal O}_{M}$-linear homomorphism $\Phi$ is first defined on $i^{*}_{\{\theta =0\}}{\cal G}$. Its matrix in the basis $\omega$ is $D(\underline{x})$. Relations (\ref{eq:integrabilite}) follow from the flatness of $\nabla$. This shows (2).
\end{proof}

\noindent Notice that the flat residual connection $\bigtriangledown$ is not defined if we forget the ``logarithmic'' assumption on the poles at infinity.

\begin{definition}\label{def:eqcar}
We will call equation (\ref{eq:baseglobale}) {\em characteristic equation} of the quantum differential system ${\cal Q}$.
\end{definition}

\begin{remark}\label{rem:metglob} (1) $S$ induces bilinear forms (also denoted by $S$)
\begin{equation}\label{formeE}
S:E\times j^{*}E\rightarrow {\cal O}_{M}\theta^{d},
\end{equation}
\begin{equation}\label{formeG0}
S:G_{0}\times j^{*}G_{0}\rightarrow {\cal O}_{M}[\theta ]\theta^{d}
\end{equation}
\noindent and
\begin{equation}\label{formeGinfty}
S:G_{\infty}\times j^{*}G_{\infty}\rightarrow {\cal O}_{M}[\tau ]\tau^{-d}.
\end{equation}
\noindent If $\eta$ and $\nu$ are two global sections of ${\cal G}$, that is if $\eta ,\nu\in E$, 
will write  
\begin{equation}\label{formeg}
S(\eta ,\nu )=g(\eta ,\nu )\theta^{d}\in {\cal O}_{M}\theta^{d}
\end{equation}
where $g$ is a non-degenerate ${\cal O}_{M}$-bilinear form on $E$ and we have, because $S$ is flat, 
\begin{equation}\label{eq:sym}
V_{\infty}+V_{\infty}^{*}=dI,\ V_{0}^{*}=V_{0}\ \mbox{and}\ \Phi^{*}=\Phi
\end{equation}
\noindent where $^{*}$ denotes the adjoint with respect to $g$.\\
\noindent (2) If the basis $\omega$ of $E$ is $\bigtriangledown$-flat, that is if $C(\underline{x})\equiv 0$ in equation (\ref{eq:baseglobale}), we have $S(\omega_{i},\omega_{j})\in \cit\theta^{d}$
for all $i$ and for all $j$, because $S$ is $\nabla$-flat.
If moreover $S(\omega_{i},\omega_{j})=0$ except for a unique index $\overline{i}$, we will also put  $\omega^{i}:=\frac{1}{g(\omega_{i},\omega_{\overline{i}})} \omega_{\overline{i}}$ and
we will call $\omega^{i}$ the {\em dual} of $\omega_{i}$. We will also say that $\omega$ is {\em adapted} to $S$.
\end{remark}

\begin{remark}\label{rem.MiscSDQ} 
(1) If ${\cal Q}$ is a quantum differential system on a point, the connection $\nabla$ takes the form
$$\nabla =(\frac{V_{0}}{\theta}+V_{\infty})\frac{d\theta}{\theta}$$
where $V_{0}$ and $V_{\infty}$ are endomorphisms of the finite dimensional $\cit$-vector space $E$, and $g$ is a bilinear, symmetric and non-degenerate form on $E$ such that
\begin{center}
$V_{0}^{*}=V_{0}$, $V_{\infty}+V_{\infty}^{*}=dI$ 
\end{center}
where, as above, $^{*}$ denotes the adjoint with respect to $g$.\\
(2) We will also consider quantum differential systems on the affine space (resp. the torus...). In this algebraic setting, the objects $\bigtriangledown$ , $V_{0}$, $V_{\infty}$, $\Phi$ and $g$ are modified accordingly (replace ${\cal O}_{M}$-linear by $\cit [\underline{x}]$-linear, resp. $\cit [\underline{x}, \underline{x}^{-1}]$-linear...). In this situation, $E$ is a free $\cit [\underline{x}]$-module (resp. $\cit [\underline{x}, \underline{x}^{-1}]$-module...), see \cite{D1}, \cite{D2}.
\end{remark}

In some cases, we can refine the previous definitions and define the {\em logarithmic} quantum differential systems (see \cite{R}):

\begin{definition}\label{def:LogSDQ}
Let $D$ be a divisor in $M$. We will say that the quantum differential system ${\cal Q}$ has logarithmic poles along $D$ if moreover
$$\Phi: E\rightarrow  E\otimes\Omega_{M}^{1}(\log D)$$ 
and
$$\bigtriangledown: E\rightarrow  E\otimes\Omega_{M}^{1}(\log D)$$
where $\Omega_{M}^{1}(\log D)$ denotes the module of the differential forms with logarithmic poles along $D$.
\end{definition}

\noindent The following example shows that quantum cohomology produces naturally logarithmic quantum differential systems:

\begin{example} \label{exemplebasiqueA} (A-side)

\noindent One associates (canonically) a quantum differential system to the (small) quantum cohomology of a smooth projective variety $X$, with cohomology only in even degree, as follows
(see for instance \cite{CK} and the references therein; we assume here that the quantum product $\circ$ is everywhere convergent\footnote{Replace $M$ by $U$ in what follows if the quantum product converges only on $U\subset M$}):
the trivial bundle is the one whose fibers are $H^*(X,\cit )$, that is
$$\pi :\ppit^1\times M\times H^*(X,\cit )\rightarrow \ppit^1\times M$$
where $M=H^*(X,\cit )$.
Let $\{\phi_{k}\}_{k=0}^{\mu -1}$ be a (homogeneous) basis of $H^{*}(X)$ and $\{t_{k}\}_{k=0}^{\mu -1}$ be a dual coordinate system on $H^{*}(X, \cit )$. The connection $\nabla^{A}$ is defined by   
$$\nabla^{A}_{\partial_{t_{k}}}=\partial_{t_{k}}+\frac{1}{\theta}\phi_{k}\circ_{\xi}\ \mbox{and}\ \nabla^{A}_{\theta\partial_{\theta}}=\theta\partial_{\theta}+\frac{1}{\theta}E\circ_{\xi}+\mu$$
where $\circ$ denotes the quantum product, parametrized by $\xi\in M$, $E$ is the ``Euler vector field``
$$E:=c_{1}(TX)+\sum_{k=0}^{\mu -1}(1-\frac{1}{2}\deg \phi_{k})t^{k}\phi_{k}$$
and
$$\mu (\phi_{k}):=(\frac{1}{2}\deg \phi_{k}-\frac{n}{2})\phi_{k}.$$
\noindent Notice that, by definition, the sections $\phi_{k}$ are $\bigtriangledown$-flat. Flatness of $\nabla^{A}$ follows from the associativity and the commutativity of the quantum product. The
metric $S$ is built with the help of the form $(a,b)=\int_X a\cup b$,
where $a$ and $b$ cohomology classes. 
In the case of small quantum cohomology, we have
$$\nabla^{A}_{q_{k}\partial_{q_{k}}}=q_{k}\partial_{q_{k}}+\frac{1}{\theta}\phi_{k}\circ_{\xi}$$
where $q_{k}=e^{t_{k}}$ for $k=1,\cdots , r$ and $r=\dim_{\cit}H^2 (X,\cit )$ and $t_{k}\in H^2 (X,\cit )/2i\pi H^2 (X,\zit )$, thanks to the divisor axiom. 
Analogous construction for orbifolds (see f.i \cite{Ir}, \cite{DoMa}). 
\end{example}

\subsection{Quantum differential systems associated with regular functions}
\label{sec:QDSregular}
One can attach a quantum differential system to any regular tame function, see \cite{DoSa1}, \cite{DoSa2}, \cite{D1}, \cite{D2}: an overview of the 
construction is given in the Appendix. Let us emphasize the following facts:

\begin{enumerate}
\item a solution of the Birkhoff problem for the Brieskorn lattice of a regular tame function (see step 2 in the Appendix) produces a quantum 
differential system.  
\item Two different solutions of the Birkhoff problem yield {\em a priori} two different bundles (which can be difficult to compare) and, 
even if the maps $\Phi$ ({\em resp.} $V_{0}$) associated with two different quantum differential systems are conjugated, 
the endomorphisms $V_{\infty}$ ({\em resp.} the connection $\bigtriangledown$) will not in general. 
We thus have to pay attention to this crucial problem on the B-side: what is the/a  ``good'' choice? 
\item It turns out that the solutions of the Birkhoff problem are in one-to-one correspondence with the opposite filtrations, stable under the
 action of the monodromy, to the Hodge filtration defined on the nearby cycles, see \cite[Appendix B]{DoSa1}. In \cite{DoSa1}, 
canonical solutions of the Birkhoff problem (hence canonical quantum differential systems) are defined (we require in addition that $V_{\infty}$ is 
semi-simple, its eigenvalues running through the spectrum at infinity of the function, as defined in \cite{Sab3}): 
the opposite filtrations alluded to are constructed with the help of Deligne's $I^{pq}$, following an idea of M. Saito \cite{Sai}.   
\end{enumerate}
 
\noindent All these phenomena are also discussed in \cite{DoSa2}, \cite{D1} and \cite{D2}.

\begin{example}\label{exemplebasiqueB}  ($B$-side, after \cite{DoMa} and \cite{DoSa2})

\noindent Let $(w_{1},\cdots ,w_{n})$ be strictly positive integers,
$U=(\cit^{*})^{n}$ et $M^{B}=\cit^*$. We define $F:U\times M^{B}\rightarrow\cit$ by
$$F(u_{1},\cdots ,u_{n},x)=u_{1}+\cdots +u_{n}+\frac{x}{u_{1}^{w_{1}}\cdots u_{n}^{w_{n}}}.$$
A distinguished solution $\omega =(\omega_{0},\cdots ,\omega_{\mu -1})$ of the Birkhoff problem for the Brieskorn lattice $G_{0}$ of $F$ 
(see Appendix), which is a free $\cit [x, x^{-1},\theta ]$-module of rank $\mu :=1+w_1 +\cdots + w_n$, 
is described in \cite[Section 4]{DoMa} (see also section \ref{subsubsec:Bside} below). 
It yields an extension of this lattice as a trivial bundle  ${\cal G}^{B}$ on $\ppit^{1}\times M^{B}$, 
equipped with a connection $\nabla^{B}$ with the required poles: the matrix of the Gauss-Manin connection in the basis $\omega$ takes the form
\begin{equation}\label{eq:carorbB}
(\frac{A_{0}(x)}{\theta}+A_{\infty})\frac{d\theta}{\theta}+(R -\frac{A_{0}(x)}{\mu\theta})\frac{dx}{x}
\end{equation}
where\footnote{We put $w^{w}=w_{1}^{w_{1}}\cdots w_{n}^{w_{n}}$.}
$$A_{0}(x)=\mu \left ( \begin{array}{cccccc}
0  & 0  & 0 & 0  & 0 & \frac{x}{w^{w}}\\
1  & 0  & 0 & 0  & 0 & 0\\
0  & 1  & 0 & 0  & 0 & 0\\
0  & 0  & 1 & 0  & 0 & 0\\
0  & 0  & 0 & .. & 0 & 0\\
0  & 0  & 0 & 0  & 1 & 0
\end{array}
\right )$$
\noindent $R=diag (c_{0},\cdots ,c_{\mu -1})$, $c_{i}\in [0,1[$ (see formula (\ref{eq:ck}) in section \ref{subsubsec:Bside} for a definition of the $c_{i}$'s) 
and $A_{\infty}=diag(\alpha_0 ,\cdots , \alpha_{\mu -1} )$. The rational numbers $\alpha_{i}$ are defined by 
$\alpha_{i}=i-\mu c_{i}$ and run through the spectrum at infinity, as defined in \cite{Sab3} (see also \cite{DoSa2}), of the function $f:=F(\bullet , 1)$. 

Define
$$S^{B}(\omega_{k},\omega_{\ell})=\left\{ \begin{array}{ll}
\frac{1}{w_1 \cdots w_n} \theta^{n} &  \mbox{if $0\leq k\leq n$ and $k+\ell =n$,}\\
\frac{x}{w^{w}}\frac{1}{w_1 \cdots w_n} \theta^{n}  & \mbox{if  $n+1\leq k\leq \mu -1 $ and $k+\ell =\mu +n$,}\\
0 & \mbox{otherwise}
\end{array}
\right .$$
These formulas are extended to ${\cal G}^{B}$ and yield a bilinear form satisfying the properties of definition \ref{def:SDQ} with $d=n$.

Summarizing (\cite[Theorem 4.4.1]{DoMa}) the tuple 
$${\cal Q}^{B}=(M^{B},{\cal G}^{B}, \nabla^{B}, S, n)$$
is a quantum differential system on $M^{B}=\cit^{*}$. The restriction of this quantum differential system at $x=1$ is a quantum differential system on a point
which is precisely, after \cite{DoSa2}, the canonical quantum differential system associated with the function $F(\bullet ,1)$ by the construction above.
\end{example}

\begin{remark}\label{rem:FonctionsMirroir} 
Keeping mirror symmetry in mind (see section \ref{subsec:Mir} below), some generalizations of the previous example are expected. For instance:\\
(1) the ones considered in
\cite{Ir} (see also \cite{RS}), where the function $F$ is the function 
\begin{equation}\label{eq:gen}
u_{0}+\cdots +u_{n}
\end{equation}
defined on
$$U=\{(u_{0},\cdots ,u_{n})\in\cit^{n+1}|\ u_{0}^{w_{0}^{1}}\cdots u_{n}^{w_{n}^{1}}=x_{1},\ \cdots , u_{0}^{w_{0}^{r}}\cdots u_{n}^{w_{n}^{r}}=x_{r}\}$$
where $(x_{1},\cdots ,x_{r})\in (\cit^{*})^{r}$ and $(w_{0}^{i},\cdots , w_{n}^{i})\in\zit^{n}$ for $1\leq i\leq r$;\\
(2) the Hori-Vafa models (see \cite{HV}, \cite{Giv}, \cite{VP}...) where the function $F$ is the function 
\begin{equation}\label{eq:genIC}
u_{0}+\cdots +u_{n}
\end{equation}
defined on 
$$U=\{(u_{0},\cdots ,u_{n})\in\cit^{n+1}|\ u_{0}^{w_{0}}\cdots u_{n}^{w_{n}}=x,\ \sum_{j\in I_{i}}u_{j}=1\}.$$
Here $x\in \cit^{*}$, $(w_{0},\cdots , w_{n})\in (\nit^{*})^{n}$ and $I_{i}$, $i=1,\cdots ,k$, are $k$ non-intersecting subsets of $\{0,\cdots ,n\}$. 
The coordinate $u_{j}$ is assumed to have the degree $w_{j}$ and $\sum_{j\in I_{i}}u_{j}$ the degree $l_{i}:=\sum_{j\in I_{i}}w_{j}$. These functions are expected to be $B$-models for smooth hypersurfaces in projective spaces.
Whatever happens, it could be 
interesting to study the quantum differential system associated with $F$, see 
\cite{D4}, \cite{GS}.
\end{remark}

\noindent Another slightly different source of examples is given by  the {\em rescalings} of a tame regular function: we will come back to this in detail 
in section \ref{sec:rescaling}. In order to complete the panorama, see also section \ref{sec:hirzebruch} for the mirror of the Hirzebruch surface.

\subsection{Reconstruction theorem and (pre-)primitive forms}

An important point in the theory of quantum differential systems is that one can define unfoldings and even {\em universal} unfoldings of such objects, 
see \cite[Definition 2.3]{HeMa}.
In some cases, a finite set of initial data allows to construct a universal unfolding of a given quantum differential system (see \cite[Theorem 2.5]{HeMa} 
and the references to B. Malgrange and B. Dubrovin therein): this is the starting point in \cite{D1} and \cite{D2} in order to construct {\em canonical} 
quantum differential systems and canonical Frobenius manifolds associated with Laurent polynomials.

\begin{definition}
\label{def:periodmap}
Let ${\cal Q}=(M, {\cal G}, \nabla ,S, d)$ be a quantum differential system on $M$, $\omega$ be a global section of ${\cal G}$. 
The {\em period map} associated with $\omega$ is the map
$$\varphi_{\omega}:\Theta_{M}\rightarrow i^{*}_{\{\theta =0\}}{\cal G}$$ 
defined by $\varphi_{\omega}(\xi )=-\Phi_{\xi}(\omega)$ where $\Theta_{M}$ denotes the sheaf of vector fields on $M$.
\end{definition}

Let $M=(\cit^{r+1},0)$ and $\gm$ be the maximal ideal of ${\cal O}_{M}$. The index $^{o}$ will denote the operation ''modulo $\gm$''. 
Hertling and Manin's theorem is the following:

\begin{theorem}[Theorem 2.5 in \cite{HeMa}]
\label{theo:reconstruction}
Let ${\cal Q}=(M, {\cal G}, \nabla ,S, d)$ be a quantum differential system on $M$. 
Assume that there exists a $\triangledown$-flat section $\omega\in i^{*}_{\{\theta =0\}}{\cal G}$ such that
\begin{enumerate}
\item (GC) $\omega^{o}$ and its images under iteration of the maps $R_{0}^{o}$ and $\Phi_{\xi}^{o}$, for all $\xi\in\Theta_{M}^{o}$, 
generate $i^{*}_{\{\theta =0\}}{\cal G}^{o}$,
\item (IC) the period map $\varphi_{\omega^{o}}:\Theta_{M}^{o}\rightarrow i^{*}_{\{\theta =0\}}{\cal G}^{o}$ defined is injective.
\end{enumerate}
\noindent Then the quantum differential system ${\cal Q}$ has a universal deformation.
\end{theorem}

\begin{definition}\label{def:preprimcan}

\begin{enumerate}
\item A section $\omega$ satisfying the conditions (GC) and (IC) is called {\em pre-primitive}. 
If moreover the period map $\varphi_{\omega}$ is an isomorphism we will say that $\omega$ is {\em primitive}.
\item We will say that a pre-primitive section $\omega$ is {\em canonical} if it generates the eigenspace associated 
with the smallest eigenvalue of $V_{\infty}$.
\end{enumerate}
\end{definition}

\begin{remark} \label{remarkreg} 
(1) Condition (IC) is empty
if $M=\{point\}$. Assume moreover that $R_{0}$ is regular, i.e its characteristic polynomial is equal to its minimal polynomial: 
there exists $\omega$ such that 
$$\omega ,R_{0}(\omega ),\cdots , (R_{0})^{\mu -1}(\omega )$$
is a basis of $i^{*}_{\{\theta =0\}}{\cal G}$ over $\cit$ and $\omega$ is thus pre-primitive.\\
(2) Condition (GC) is satisfied if $\omega$ and its derivative $\theta \nabla_{X_{1}}\theta \nabla_{X_{2}}\cdots \theta \nabla_{X_{\ell}}\omega$ 
generate $i^{*}_{\{\theta =0\}}{\cal G}$.
\end{remark}

 Theorem \ref{theo:reconstruction} is originally stated for a punctual germ $M$. More convenient (for our purpose) global versions of this result 
can be found in \cite{D2}: for instance, if $M=\ait^{r+1}$, the period map attached to $\omega$ is now a $\cit [\underline{x}]$-linear map, defined on the 
Weyl algebra $\ait^{r}(\cit )=\cit [\underline{x}]<\partial_{\underline{x}}>$, by $\varphi_{\omega}(\xi)=-\Phi_{\xi}(\omega )$ 
(analogous construction if $M=(\cit^{*})^{r}$ is a torus).

\begin{example}
\label{exemple:PrePrimCan}
In the situation of example \ref{exemplebasiqueB}, the period map associated with the section $\omega_{0}$ 
is 
$$\varphi_{\omega_{0}}:\cit [x,x^{-1}]<\partial_{x}>\rightarrow G_{0}/\theta G_{0}$$
where $\varphi_{\omega_{0}}(x\partial_{x})=\omega_{1}$. The section $\omega_{0}$ is pre-primitive and 
canonical in the sense of definition \ref{def:preprimcan}.
We also define a connection $\bigtriangledown^{\omega_{0}}$ on $\cit [x,x^{-1}]<\partial_{x}>$ by
$$\bigtriangledown^{\omega_{0}}_{x\partial_{x}}(x\partial_{x})=\bigtriangledown_{x\partial_{x}}(\varphi_{\omega_{0}}(x\partial_{x}))=
\bigtriangledown_{x\partial_{x}}(\omega_{1})$$
The vector field $x\partial_{x}$ is $\bigtriangledown^{\omega_{0}}$-flat because 
$\bigtriangledown_{x\partial_{x}}(\omega_{1})=c_{1}\omega_{1}$
and $c_{1}=0$ for $n\geq 1$ (see section \ref{subsubsec:Bside} below): 
the coordinate $t$ defined by $x=e^{t}$ is thus flat. 
See also section \ref{subsub:FlatCoordinates} for another concrete computation of flat coordinates.
\end{example}

\subsection{A motivation: mirror symmetry via quantum differential systems}

\label{subsec:Mir}

\subsubsection{Mirror symmetry and quantum differential systems}
Another (connected) important point is that one can compare two quantum differential 
systems:

\begin{definition}\label{def:IsoSDQ}
The quantum differential systems 
$$(M^{A}, H^{A}, \nabla^{A}, S^{A} , n^{A})\ \mbox{and}\  (M^{B}, H^{B}, \nabla^{B}, S^{B} , n^{B})$$ 
are {\em isomorphic} if there exists an isomorphism 
$(id ,\nu): \ppit^1\times M^{A}\rightarrow \ppit^1\times M^{B}$ and an isomorphism of vector bundles 
$\gamma : H^{A}\rightarrow (id ,\nu )^*H^{B}$ compatible with the connections and the metrics.
\end{definition}

\begin{definition}\label{def:Mir}
Two models are mirror partners if their quantum differential systems are isomorphic.
\end{definition}

\noindent Notice that the map $\nu$ measures ``flatness'': 
if $(M^{A}, H^{A}, \nabla^{A}, S^{A} , n^{A})$ is the quantum differential system associated with the small quantum cohomology 
by example  \ref{exemplebasiqueA}, it is defined by flat coordinates on $M^{B}$, because the ones used on $M^{A}$ are so: 
see {\em f.i} theorem \ref{theo:miroir} and section \ref{sec:MirF2} below.

\subsubsection{Example: mirror symmetry for weighted projective spaces}
\label{sec:MirSymWPS}

A nice class of examples, which bring to light some unexpected phenomena (this is discussed with some details in \cite{DoMa}), 
is given by weighted projective spaces for which we have the following result:
let $p=c_1({\cal O}(1))\in H^2_{orb}(\ppit (w), \cit )$ and
$$(p^{\circ_{tp}})^i=p\circ_{tp}\cdots \circ_{tp}p,$$
where the quantum product $\circ$ is counted $i-1$ times. We keep the notations of example \ref{exemplebasiqueB}.

\begin{theorem}[Theorem 5.1.1 in \cite{DoMa}]
\label{theo:miroir}
The quantum differential system associated with the weighted projective spaces\footnote{The case $w_{1}=\cdots =\omega_{n}=1$ 
has been first considered in \cite{Bar1}}
$\ppit (1,w_{1},\cdots ,w_{n})$ by example \ref{exemplebasiqueA} is isomorphic to the one associated with the 
fonction $F$ by example \ref{exemplebasiqueB}. We have $M^{A}=M^{B}=\cit^{*}$ and
the isomorphism $\gamma$ (resp. $\nu$) sends $(p^{\circ_{tp}})^i$ onto $\omega_{i}$ (resp. is the identity).
\end{theorem}

\noindent Notice that the coordinate $x$ is flat (see example \ref{exemple:PrePrimCan} above) and this explains why the map $\nu$ is equal to the identity.
Nevertheless, $(p^{\circ_{tp}})^i$ is {\em not} a flat section of the residual connection $\bigtriangledown^{A}$. 

\begin{remark} (1) As a consequence of the theorem, the matrix of the small quantum multiplication $p\circ_{tp}$ in the basis $((p^{\circ_{tp}})^i)$ is equal to
$A_0^F(e^t)$, which is the matrix of multiplication\footnote{$G_0^F/\theta G_0^F$ is naturally equipped with a structure of ring, {\em via}  $\omega_{0}$, a (pre-)primitive section.} by $\omega_0$ on $G_0^F/\theta G_0^F$ in the basis induced by $\omega$.
One gets in this way a correspondence between the products which allows to see the quantum product as a simple computation in algebra.\\ 
(2) The rational number 
$\alpha_{i}$ is equal to half of the orbifold degree of $(p^{\circ_{tp}})^i$: the latter are thus in correspondence with the 
``spectrum at infinity'' of the fonction $F(\bullet ,1)$.
\end{remark}

More generally, the function in remark \ref{rem:FonctionsMirroir} (1) should give the mirror (in the sense of definition \ref{def:Mir}) of toric orbifolds (see \cite{Ir}, \cite{RS})   
and the one in remark \ref{rem:FonctionsMirroir} (2) should give, after \cite{HV}, \cite{Giv}, \cite{VP}, 
the mirror of complete intersections in the weighted projective space $\ppit (w_{0},\cdots ,w_{n})$. The computation of our mirror partner of the 
Hirzebruch surface $\fit_{2}$ is done in section \ref{sec:hirzebruch} (see theorem \ref{theo:miroirF2}).

\section{The Dubrovin connection and the quantum product of a quantum differential system}
\label{sec:Dub}
Let ${\cal Q}=(M, {\cal G}, \nabla , S, n)$ be a quantum differential system, $\omega =(\omega_{0},\cdots ,\omega_{\mu -1})$ 
be an ${\cal O}_{M}$-basis\footnote{The algebraic version is straightforward.} of $E=\pi_{*}{\cal G}$. 
As above, we will denote by $\underline{x}=(x_{0},\cdots ,x_{r})$ the coordinates on $M$.

\subsection{Dubrovin connection of a quantum differential system}\label{sec:DubConnection}

An important object for our purpose is the {\em Dubrovin connection} of the quantum differential system ${\cal Q}$, 
which encodes the ``quantum product'', as defined in section \ref{ProduitQuantique} below. 
We rewrite here what is known in a slightly different settings, see \cite{Dub1} and \cite[section 8.4]{CK}.

\begin{definition}\label{def:connexionDubrovin}
The connection $\nabla^{d}$ defined by
\begin{equation}
\nabla^{d}:=\bigtriangledown +\frac{\Phi}{\theta},
\end{equation}
is called the {\em Dubrovin connexion} of the quantum differential system ${\cal Q}$.
\end{definition}

\begin{proposition}\label{prop:dubint}
The connection $\nabla^{d}$ is flat.
\end{proposition}
\begin{proof}
Flatness of $\nabla^{d}$ is equivalent to $\bigtriangledown^{2}=0$, $\Phi\wedge\Phi =0$ and $\bigtriangledown \Phi =0$. This is precisely what gives proposition \ref{keyproposition}.
\end{proof}

\noindent Let us emphasize once again that the flatness of $\bigtriangledown$, and hence the flatness of $\nabla^{d}$, is a characteristic property of quantum differential systems which is lost if we drop the assumption on the poles at infinity of the connection $\nabla$.\\

\begin{remark}
A quantum differential system on $M$ produces naturally a variation of semi-infinite Hodge structures on $M$, in the sense of Barannikov \cite{Bar2}. 
Recall that $G_{0}$ denotes the restriction of ${\cal G}$ at $U_{0}$: it is a free ${\cal O}_{M}[\theta ]$-module. 
We define an increasing filtration $F_{\bullet}$ of $G:=G_{0}[\theta^{-1}]$ by ${\cal O}_{M}[\theta]$-submodules, putting $F_{p}G:=\theta^{-p}G_{0}$:
we thus have $\nabla^{d}(F_{p})\subset F_{p+1}$ (Griffith's transversality condition).
Let us now define the ``$\theta$-connection''
$$\nabla^{\theta}:G_{0}\rightarrow \Omega^{1}_{M}\otimes G_{0}$$
by $\nabla^{\theta}:=\theta\nabla^{d}$. By definition it satisfies
$$\nabla^{\theta}_{X}(f(\underline{x}, \theta )\eta )=(\theta Xf(\underline{x},\theta ))\eta +f(\underline{x},\theta )\nabla^{\theta}_{X}\eta ,$$
$$[\nabla^{\theta}_{X} , \nabla^{\theta}_{Y}]=\theta \nabla^{\theta}_{[X,Y]}$$
and
$$\theta X S(\eta ,\nu )=S(\nabla^{\theta}_{X}\eta ,\nu )-S(\eta ,\nabla^{\theta}_{X}\nu )$$
where $\eta$ and $\nu$ ({\em resp.} $f(\underline{x},\theta)$) are sections of $G_{0}$ ({\em resp.} ${\cal O}_{M}[\theta]$) and $X$ and $Y$  are vector fields on $M$. Summarizing,
the tuple $VSHS_{\cal Q}:=(M, G_{0}, \nabla^{\theta}_{M}, S)$ is a variation of semi-infinite Hodge structures on $M$.    
\end{remark}

\subsection{Quantum product of a quantum differential system}
\label{ProduitQuantique}

In what follows, we will assume that $\Phi_{\partial_{x_{0}}}$ is the identity\footnote{On the $A$-side, that is in the setting of example \ref{exemplebasiqueA}, the coordinate $x_{0}$ is associated with the first cohomology goup $H^{0}$ while on the $B$-side $x_{0}$ is the constant term in the unfoldings of functions}. 

\begin{definition}
We define an ${\cal O}_{M}$-bilinear map $*$ on $E\simeq G_{0}/ \theta G_{0}$ by
$$\omega_{i}*\omega_{j} := [\Phi_{\partial_{x_{i}}}(\omega_{j})]$$
for all $0\leq i\leq r$ and $0\leq j\leq \mu -1$, where $[\ ]$ denotes the class in $E$.
\end{definition}

\noindent This map defines a product but it doesn't need to have any associativity and/or commutativity property and/or an identity. This is however sometimes the case: 

\begin{proposition}\label{prop:prodquant}
Assume that the section $\omega_{0}$ is such that $\omega_{i}=\Phi_{\partial_{x_{i}}}(\omega_{0})$ for all $0\leq i\leq r$. Then
\begin{enumerate}
\item $\omega_{i}*\omega_{0}=\omega_{i}$ for all $0\leq i\leq r$,
\item $\omega_{i}*\omega_{j}=\omega_{j}*\omega_{i}$ for all $0\leq i,j\leq r$,
\item $(\omega_{i}*\omega_{j})*\omega_{k}=\omega_{i}*(\omega_{j}*\omega_{k})$ for all $0\leq i,j,k\leq r$, under the assumption that $\omega_{i}*\omega_{j}\in \sum_{k=0}^{r}{\cal O}_{M}\omega_{k}$\footnote{This happens for instance if $r=\mu -1$. If not, the left hand is not well defined.},
\item $g(\omega_{i}*\omega_{j},\omega_{k})=g(\omega_{j},\omega_{i}*\omega_{k})$ for all $0\leq i\leq r$ and $0\leq j,k\leq \mu -1$
\end{enumerate}
where $g$ is the bilinear form defined by formula (\ref{formeg}).
\end{proposition}
\begin{proof} The first point is clear, thanks to the assumption on $\omega_{0}$. The second one follows
from the formula $\Phi\wedge\Phi =0$ (see proposition \ref{keyproposition}) from which we get $\Phi_{\partial_{x_{i}}}\Phi_{\partial_{x_{j}}}=\Phi_{\partial_{x_j}}\Phi_{\partial_{x_i}}$. For the third one, one writes 
$$\omega_{i}*(\omega_{j}*\omega_{k})=\Phi_{\partial_{x_{i}}}(\omega_{j}*\omega_{k})=
\Phi_{\partial_{x_{i}}}(\Phi_{\partial_{x_{j}}}(\Phi_{\partial_{x_{k}}}(\omega_{0})))$$
and, using the moreover the second assertion,  
$$(\omega_{i}* \omega_{j})*\omega_{k}=\omega_{k}* (\omega_{i}* \omega_{j})=\Phi_{\partial_{x_{k}}}(\Phi_{\partial_{x_{i}}}(\Phi_{\partial_{x_{j}}}(\omega_{0})))$$
We get the desired formula because $\Phi_{\partial_{x_{l}}}$ and $\Phi_{\partial_{x_{r}}}$ commute.
Last, we have
$$g(\omega_{i}*\omega_{j},\omega_{k})=g(\Phi_{i}(\omega_{j}),\omega_{k})=g(\omega_{j}, \Phi_{i}(\omega_{k}))
=g(\omega_{j}, \omega_{i}*\omega_{k}),$$
where the second equality follows from $\Phi^{*}=\Phi$ (see remark \ref{rem:metglob}).
\end{proof}

\noindent Notice that a section $\omega_0$ as in proposition \ref{prop:prodquant}
defines an injective period map, see definition \ref{def:preprimcan}. 
On the $B$-side it happens that there exists such sections, for instance if the quantum differential 
system is associated with a subdiagram deformation of a convenient and non-degenerate Laurent polynomial, see \cite{D2}. See also 
example \ref{exemple:PrePrimCan} and section \ref{sec:hirzebruch}.

\begin{definition}
In the situation of proposition \ref{prop:prodquant}, we will say that the map $*$ is the {\em quantum product} and that $E$ is the  {\em quantum algebra} of the quantum differential system ${\cal Q}$.
\end{definition}

\begin{remark}
(1) The variation of semi-infinite Hodge structure $VSHS_{\cal Q}$, and hence the quantum differential system ${\cal Q}$,
yields a ``quantization'' on the $\theta$-axis
of the quantum algebra $E$ (on the B-side, $E$ is a Jacobian ring): $G_{0}$ is a ${\cal O}_{M}[\theta]$-free module and
$E\simeq i^{*}_{\{\theta =0\}}G_{0}$.\\
(2) 
Assume that $\omega_{0}$ is as in proposition \ref{prop:prodquant}.
Assume moreover that we have an isomorphism of ${\cal O}_{M}$-modules $\delta : E\rightarrow H$ and let $\eta_{i}= \delta (\omega_{i})$. 
This isomorphism yields a product
$\circ$ on $H$ by
\begin{equation}\label{ProdQuant}
\eta_{k}\circ\eta_{\ell}:=\delta (\delta^{-1}(\eta_{k})*\delta^{-1}(\eta_{\ell})),
\end{equation}
a connection $\bigtriangledown^{\delta}$ on $H$ by 
\begin{equation}\label{ConQuant}
\bigtriangledown^{\delta}(\eta_{j})=\delta\bigtriangledown (\delta^{-1}(\eta_{j}))
\end{equation}
and a ``metric'' $g^{\delta}$ on $H$ by 
\begin{equation}\label{MetQuant}
g^{\delta}(\eta_{i},\eta_{j})=g(\delta^{-1}(\eta_{i}),\delta^{-1}(\eta_{j})),
\end{equation}
these objects being extended on $H$ by linearity.
By definition, the product $\circ$ inherits all the properties of $*$ and
the connection $\nabla^{\delta}$ defined on $H$ by 
$$\nabla^{\delta}_{\partial_{x_{i}}}(\sum_{k}a_{k}(\underline{x})\eta_{k})=
\bigtriangledown^{\delta}_{\partial_{x_{i}}}(\sum_{k}a_{k}(\underline{x})\eta_{k}) +\tau\sum_{k}a_{k}(\underline{x})\eta_{i}\circ\eta_{k}$$
is flat. If $H=\Theta_{M}$ and if the period map $\varphi_{\omega_{0}}$ (see theorem \ref{theo:reconstruction}) is an isomorphism, 
we get in this way a Frobenius manifold. 
\end{remark}

\section{Fundamental solutions of a quantum differential system}
\label{sec:SolFonda}

Let ${\cal Q}=(M, {\cal G}, \nabla , S, n)$ be a quantum differential system $\dim_{\cit}M=r+1$. We have
$$\nabla =\bigtriangledown +\frac{\Phi}{\theta} +(\frac{V_{0}}{\theta}+V_{\infty})\frac{d\theta}{\theta}$$
by proposition \ref{keyproposition}. Until the end of this paper, we will assume that $V_{\infty}$ is {\em semi-simple}
as it will be the case in our favorite situations ($A$-side and $B$-side). 
In what follows, $\omega =(\omega_{0},\cdots ,\omega_{\mu -1})$ will denote a basis of $\pi_{*}{\cal G}$ over ${\cal O}_{M}$, fixed once for all. In the basis $1\otimes \omega$ de ${\cal G}$, the matrix of $\nabla$ is thus
\begin{equation}
(\frac{A_{0}(\underline{x})}{\theta}+A_{\infty}(\underline{x})\frac{d\theta}{\theta}+C(\underline{x})+\frac{D (\underline{x})}{\theta}
\end{equation}
where $\underline{x}=(x_{0},\cdots ,x_{r})\in M$. Recall that $\tau :=\theta^{-1}$.

\subsection{The fundamental solutions of the Dubrovin connection}
\label{subsec:Dub}
Recall the flat Dubrovin connection $\nabla^{d}$ of definition \ref{def:connexionDubrovin}. It has flat sections. Let us precise this point. Again, 
the results of this section are in essence classical, see {\em f.i} \cite{Dub2}, \cite{CK}, and we detail them in our situation 
mainly to establish notations and to set the different objects that we will use later.\\

\begin{lemma} There exists a (non necessarily unique) formal power series (in $\tau$)
$$Q(\underline{x},\tau )=I+\sum_{i\geq 1}Q_{i}(\underline{x})\tau^{i}$$
where $I\in {\cal M}_{\mu\times\mu}(\cit)$ is the identity matrix and $Q_{i}(\underline{x})\in {\cal M}_{\mu\times\mu}({\cal O}(M))$ such that
\begin{equation}\label{eq:Qtau}
\nabla^{d}(\omega Q)=(\bigtriangledown \omega)Q
\end{equation}
\end{lemma}
\begin{proof}
We have to show that there exists a matrix $Q$ such that
$$d_M Q(\underline{x},\tau )=-\tau D(\underline{x})Q(\underline{x},\tau )$$
\noindent where $d_M$ denotes the differential on $M$. The independant term of $\tau$ in this equality shows that $d_M Q_{0}(\underline{x})=0$ thus $Q_{0}(\underline{x})$ is constant: we choose it equal to the identity $I$.
The term of degree $r\geq  1$ in $\tau$ yields 
$$d_M Q_{r}(\underline{x})=-D(\underline{x})Q_{r-1}(\underline{x}).$$
\noindent For $r=1$, this equation has a solution because $d_M D(\underline{x})=0$ (this is what gives equation $\bigtriangledown\Phi =0$) and $d_M Q_{0}(\underline{x})=0$. It has also a solution for $r\geq 2$ because
$d_M (D(\underline{x})Q_{r-1}(\underline{x}))=0$, which is equivalent to
$$\partial_{x_{j}}(D^{(i)}(\underline{x}) Q_{r-1}(\underline{x}) )=\partial_{x_{i}}(D^{(j)}(\underline{x}) Q_{r-1}(\underline{x}))$$
for all $i$ and for all $j$. These equalities are shown by induction, using moreover the fact that $\Phi\wedge\Phi =0$.
\end{proof}

\begin{corollary}\label{prop:fonda}
\begin{enumerate}
\item There exists a matrix 
$$P(\underline{x},\tau )=P_{0}(\underline{x})+\sum_{i\geq 1}P_{i}(\underline{x})\tau^{i}\in {\cal M}_{\mu\times\mu}({\cal O}(\widetilde{M})[[\tau ]])$$
such that $\nabla^{d}((\omega_{0},\cdots ,\omega_{\mu -1})P)=0$.
\item After the base change of matrix $P$, the matrix of $\nabla$ takes the form $R(\tau )\frac{d\tau}{\tau}$ 
where 
$$R(\tau )=\sum_{k\geq 0}R_{k}\tau^{k}$$ is a formal power series in $\tau$ and the matrices $R_{k}$ are constant\footnote{In particular, $R_{0}=-P_0^{-1}A_{\infty}P_0$.}.
\end{enumerate}
\end{corollary}
\begin{proof} (1) Let $P_{0}(x)$ such that $\bigtriangledown ((\omega_{0},\cdots ,\omega_{\mu -1}) P_{0}(x))=0$. 
Apply the previous lemma to the $\bigtriangledown$-flat 
basis $(\omega_{0},\cdots ,\omega_{\mu -1})P_{0}(x)$ (if the basis $\omega$ is $\bigtriangledown$-flat from the beginning, 
$P_{0}(x)=I$ and $P=Q$). 
(2) Follows from the flatness of $\nabla$.
\end{proof}

\begin{definition} The matrix 
$P:=P(x,\tau )$ is called {\em fundamental solution} of the Dubrovin 
connection $\nabla^{d}$.
\end{definition}

\begin{remark}\label{rem:compsolfond} Let $P(\underline{x},\tau )$ be a fundamental solution.
Then $\tilde{P}(\underline{x},\tau )$ is a fundamental solution if and only if there exists an invertible matrix
$$\alpha (\tau )=\alpha_{0}+\sum_{k\geq 1}\alpha_{k}\tau^{k}$$
\noindent where $\alpha_i\in {\cal M}_{\mu\times\mu}(\cit)$, such that $\tilde{P}(\underline{x},\tau )=P(\underline{x},\tau )\alpha (\tau )$, as it follows 
from the base change formula for a connection.
\end{remark}

\subsection{The $J$-functions of a quantum differential system} 
We define here the $J$-functions of a general quantum differential system. In the situation of example \ref{exemplebasiqueA}, 
our definition agrees with the usual object considered in classical mirror symmetry.\\

Let $P$ be a fundamental solution and $(e_{0},\cdots ,e_{\mu -1}):=(\omega_{0},\cdots ,\omega_{\mu -1})P$. We
define 
$${\cal H}_{P}=\{ (e_{0},\cdots ,e_{\mu -1})
\left ( \begin{array}{c}
y_{0}\\
y_{1}\\
\cdots \\
y_{\mu -1}
\end{array}
\right )\in \Gamma(U_{\infty}\times M, {\cal G})|\ 
P\left ( \begin{array}{c}
y_{0}\\
y_{1}\\
\cdots \\
y_{\mu -1}
\end{array}
\right )\in ({\cal O}_{M})^{\mu}\}$$

\noindent If $\eta =\sum_{i=0}^{\mu -1}s_{i}\omega_{i}\in\Gamma (\ppit^{1}\times M,{\cal G})$, we will write
\begin{equation}\label{formula:J}
J_{{\cal Q}}^{P, \eta}:= P^{-1}\eta =(e_{0},\cdots ,e_{\mu -1})
\left ( \begin{array}{c}
y_{0}\\
y_{1}\\
\cdots \\
y_{\mu -1}
\end{array}
\right )\ \mbox{where}\ P
\left ( \begin{array}{c}
y_{0}\\
y_{1}\\
\cdots \\
y_{\mu -1}
\end{array}
\right )
=
\left ( \begin{array}{c}
s_{0}\\
s_{1}\\
\cdots \\
s_{\mu -1}
\end{array}
\right ).
\end{equation}
\noindent By definition, the function $J_{{\cal Q}}^{P, \eta}$ is thus the section $\eta$ expressed in the frame 
$(e_{0},\cdots ,e_{\mu -1}).$ \\

\begin{lemma}\label{lemma:ConnHP} 
\begin{enumerate}
 \item 
Let $\widetilde{\nabla}^{d}$ be the connection induced by $\nabla^{d}$ on ${\cal H}_{P}$. Then
\begin{equation}\label{eq:ConnHP}
\widetilde{\nabla}^{d}_{X}(y_{0}e_{0}+\cdots +y_{\mu -1}e_{\mu -1})=X(y_{0})e_{0}+\cdots +X(y_{\mu -1})e_{\mu -1}
\end{equation}
for any vector field $X$ on $M$.
\item We have 
\begin{equation}\label{eq:XJ}
\widetilde{\nabla}^{d}_{X}(J^{P,\eta}_{{\cal Q}} )=P^{-1}\nabla^{d}_{X}\eta
\end{equation}
for any vector field $X$ on $M$ and $\eta\in\Gamma (\ppit^{1}\times M, {\cal G})$. 
\end{enumerate}
\end{lemma}
\begin{proof}
1. Let $X$ be a vector field on $M$. We have
$$\widetilde{\nabla}^{d}_{X}(
(e_{0},\cdots ,e_{\mu -1})\left ( \begin{array}{c}
y_{0}\\
y_{1}\\
\cdots \\
y_{\mu -1}
\end{array}
\right ))
=
\nabla^{d}_{X}(
(\omega_{0},\cdots ,\omega_{\mu -1})P
\left ( \begin{array}{c}
y_{0}\\
y_{1}\\
\cdots \\
y_{\mu -1}
\end{array}
\right ))
$$
$$=
(\omega_{0},\cdots ,\omega_{\mu -1})P
\left ( \begin{array}{c}
X(y_{0})\\
X(y_{1})\\
\cdots \\
X(y_{\mu -1})
\end{array}
\right )
=
(e_{0},\cdots ,e_{\mu -1})
\left ( \begin{array}{c}
X(y_{0})\\
X(y_{1})\\
\cdots \\
X(y_{\mu -1})
\end{array}
\right )$$
where the second equality follows from the fact that $P$ is a fundamental solution.\\
2. We now have, keeping the previous notations,
$$\left ( \begin{array}{c}
X(y_{0})\\
X(y_{1})\\
\cdots \\
X(y_{\mu -1})
\end{array}
\right )
=P^{-1}i_{X}\Omega
\left ( \begin{array}{c}
s_{0}\\
s_{1}\\
\cdots \\
s_{\mu -1}
\end{array}
\right )$$
where $\Omega$ is the matrix of $\nabla^{d}$ in the basis $\omega$. 
\end{proof}

\noindent As suggested by formula (\ref{eq:ConnHP}), we will write $X(J^{P,\eta}_{{\cal Q}} )$ 
instead of $\widetilde{\nabla}^{d}_{X}(J^{P,\eta}_{{\cal Q}} )$. Notice that $YX(J^{P,\eta}_{{\cal Q}})=P^{-1}\nabla^{d}_{Y}\nabla^{d}_{X}\eta$ {\em etc}...\\

Keeping quantum product in mind, see section \ref{ProduitQuantique}, it is natural to consider the section
$J_{{\cal Q}}^{P, \omega_{0}}$ where $\omega_{0}$ satisfies the condition of proposition \ref{prop:prodquant}. 
If it happens to be the case, we have $\theta\partial_{j}J_{{\cal Q}}^{P, \omega_{0}}=P^{-1}\omega_{j}+\theta P^{-1}\bigtriangledown_{\partial_{j}}\omega_{0}$. 
In particular,
$$ \theta\partial_{j}J_{{\cal Q}}^{P, \omega_{0}}=P^{-1}\omega_{j}$$
if $\bigtriangledown \omega_{0}=0$.

\begin{definition}\label{def:fonctionJ}
Let $P$ be a fundamental solution of the Dubrovin connection and assume that the section $\omega_{0}$ is as in proposition \ref{prop:prodquant}.
We will call 
$J_{{\cal Q}}^{P, \omega_{0}}$ a
{\em $J$-function} of the quantum differential system ${\cal Q}$.
\end{definition}


 The following result explains the link with the product $*$ defined in section \ref{ProduitQuantique}.
We will write, for $H$ a polynomial function of $2r+3$ variables,  $H(\theta\partial_{x}, \underline{x} ,\theta )$ instead of  $H(\theta\partial_{x_{0}},\cdots , \theta\partial_{x_{r}}, x_{0},\cdots ,x_{r} ,\theta )$ and, for simplicity, $J$ instead of $J_{{\cal Q}}^{P, \omega_{0}}$.

\begin{proposition}\label{HJHcirc} Let $J$ be a $J$-function of the quantum differential system ${\cal Q}$.
\begin{enumerate}
\item We have $H(* , \underline{x}, 0)=0$ if $H(\theta\partial_{x}, \underline{x}, \theta ) J=0$.
\item Assume moreover that the basis $\omega$ is $\bigtriangledown$-flat. We have (operators of order $2$),
$$\theta\partial_{x_{j}}\theta\partial_{x_{i}}J=\sum_{k}a^{k}_{ji}(\underline{x})\theta\partial_{x_{k}}J$$
if and only if
$\omega_{j}*\omega_{i}=\sum_{k}a^{k}_{ji}(\underline{x})\omega_{k}$.
\end{enumerate}
\end{proposition}
\begin{proof} 
1. Let us observe that $H(\theta\partial_{x}, \underline{x} ,\theta )J=0$ if and only if 
$H(\theta\nabla^{d}_{\partial_{x}} , \underline{x}, \theta )\omega_{0} =0$,
as it follows from formula (\ref{eq:XJ}).
Now, using the definition of the quantum product,
$$\theta\nabla^{d}_{\partial_{i_{k}}}\cdots \theta\nabla^{d}_{\partial_{i_{1}}}\omega_{0} =\omega_{i_{k}}*\cdots *\omega_{i_{1}}+\theta A (\theta )$$
for a suitable formal power series $A$.
For 2., we have, thanks to the flatness of $\omega$,
$\theta^{2}\partial_{x_{j}}\partial_{x_{i}}J=P^{-1}\omega_{j}* \omega_{i}$
and we use again $\theta\partial_{x_{k}}J=P^{-1}\omega_{k}$.
\end{proof}

Let us emphasize the fact that the function $J^{P, \omega_{0}}_{{\cal Q}}$ depends on the pre-primitive section and on the fundamental solution $P$. 
The aim of the next section is to define {\em canonical} $J$-functions. In the situation of example \ref{exemplebasiqueA}, 
we will see that Givental's $J$ function is such a canonical function, obtained by taking $\omega_{0}=1$ together with a canonical fundamental solution $P$.
In this case, we have
$$J^{P, \omega_{0}}_{{\cal Q}}=e^{\tau Q\ln x}(\omega_{0}+O(\tau))$$
for a suitable constant matrix $Q$. See corollary \ref{coro:DLJ}. We will see also that, for suitable fundamental solutions $P$, 
$J^{P, \omega_{0}}_{{\cal Q}}$ can be expressed with the help of the bilinear form of the quantum differential system
${\cal Q}$, see proposition \ref{prop:fonctionJg}.

\section{Canonical fundamental solutions of a quantum differential system}

\label{sec:solfondacan}

The aim of this section is to define {\em canonical} fundamental solutions, and hence canonical $J$-functions (see section \ref{sec:Jcan}): this is done with the help of Dubrovin's symmetric and conformal fundamental solutions, see \cite[Lecture 2]{Dub2}. We keep the situation of the beginning of section \ref{sec:SolFonda}.

\subsection{A class of convergent fundamental solutions: conformal solutions (after \cite{Dub2})}

We define here some convergent (in $\tau$) fundamental solutions. 
First, one can precise corollary \ref{prop:fonda} with the help of the following well-known lemma:

\begin{lemma}\label{lemma:levelt} There exist a fundamental solution $\tilde{P}(\underline{x},\tau )$ such that the matrix of the connection is,
after the base change of matrix $\tilde{P}(\underline{x},\tau )$, 
\begin{equation}\label{matconconf}
(\tilde{R}_{0}+\tilde{R}_{1}\tau +\cdots +\tilde{R}_{\delta}\tau^{\delta})\frac{d\tau }{\tau}
\end{equation}
\noindent the matrices $\tilde{R}_{k}$ satisfying $[\tilde{R}_{0},\tilde{R}_{k}]=-k\tilde{R}_{k}$ for all $k\geq 1$
and $\tilde{R}_{0}=R_{0}$ where $R_{0}$ is defined in corollary \ref{prop:fonda}.
\end{lemma}
\begin{proof} It is essentially the one giving the Levelt normal form 
(see for instance \cite[Exercice II.2.20]{Sab1}, \cite[Lemma 2.5]{Dub2} and the references therein). 
\end{proof}

\begin{remark}
\label{rem:nonreso}
One has $\delta\leq max |\lambda_i -\lambda_j |$,
where the maximum is taken over the differences of the eigenvalues of $R_{0}$. 
In particular, $\tilde{R}_{k}=0$ for all $k\geq 1$ if $R_{0}$ is non-resonant. On the $B$-side (with the previous notations, $R_{0}=-A_{\infty}$), 
we thus have $\delta\leq \dim_{\cit}U$ 
if the quantum differential system is associated with a regular tame function $f:U\rightarrow\cit$, see Appendix, because the 
eigenvalues of $A_{\infty}$ run through the spectrum of $f$ at infinity. 
\end{remark}

\begin{definition}
A fundamental solution which has the properties of lemma \ref{lemma:levelt} is called {\em conformal}.
\end{definition}

\noindent The following two results motivate the definition of conformal solutions:

\begin{proposition} A conformal fundamental solution is convergent (in $\tau$).
\end{proposition}
\begin{proof} Let $\tilde{P}(\underline{x},\tau )$ be such a solution. then it satisfies
\begin{equation}\label{eq:PFondConf}
\tau\partial_{\tau}\tilde{P}(\underline{x},\tau )=\tilde{P}(\underline{x},\tau )\tilde{R}(\tau )+(\tau A_{0}(\underline{x})+A_{\infty})\tilde{P}(\underline{x},\tau )
\end{equation}
The matrices $\tilde{R}(\tau )$ and $\tau A_{0}(\underline{x})+A_{\infty}$ are convergent, and we can conclude using a classical 
argument of regular singularity (see {\em f.i} \cite[Proposition II.2.18]{Sab1}).
\end{proof}

\begin{proposition}\label{proposition:sectionshorizontales} Let $\tilde{P}(\underline{x},\tau )$ be a conformal fundamental solution  and 
$\Lambda$ be the diagonal matrix whose eigenvalues are the integer parts of those of $\tilde{R}_{0}$. Assume moreover that $\tilde{R}_{0}$ is block diagonal,
each block corresponding to an eigenvalue.
\begin{enumerate}
\item After the base change, meromorphic in $\tau$, of matrix $\tilde{P}(\underline{x},\tau )\tau^{-\Lambda}$, the matrix of $\nabla$ takes the form
$$(\tilde{R}_{0}-\Lambda +\tilde{R}_{1}+\cdots +\tilde{R}_{\delta})\frac{d\tau}{\tau}.$$
\item A basis of flat sections is
$(\omega_{0},\cdots ,\omega_{\mu -1})\tilde{P}(\underline{x},\tau )\tau^{-\tilde{R}_{0}}\tau^{- (\tilde{R}_{1}+\cdots +\tilde{R}_{\delta})}$.
\end{enumerate}
\end{proposition}
\begin{proof} 1.  Indeed, $[\Lambda ,\tilde{R}_{k}]=-k\tilde{R}_{k}$ if $[\tilde{R}_{0},\tilde{R}_{k}]=-k\tilde{R}_{k}$ and
2. follows because $\tilde{R}_{0}-\Lambda$ commutes with $\tilde{R}_{1}+\cdots +\tilde{R}_{\delta}$ and $\Lambda$.
\end{proof}

\begin{remark} \label{rem:conforme}
We can compare conformal fundamental solutions, as in remark \ref{rem:compsolfond}.
We will say that the matrix 
$\alpha (\tau )$ is {\em homogeneous} if
$$\alpha (\tau )=\sum_{\ell\geq 0}\alpha_{\ell}^{(-\ell )}\tau^{\ell}$$
where $\alpha_{\ell}^{(-\ell )}\in \ker (\mbox{Ad} R_{0}+\ell I)$ for all $\ell\in\zit$.
Let $P$ and $\tilde{P}$ be two fundamental solutions and assume that $P$ is conformal. Then $\tilde{P}$ is conformal if and only if there exists  a homogeneous matrix
$\alpha (\tau )$ such that $\tilde{P}=P\alpha (\tau )$, see also \cite{Dub2}.
\end{remark}

\subsection{Symmetric solutions (after \cite{Dub2})}
\label{SolSym}

Let $P$ be a fundamental solution and
$$e=(e_{0},\cdots , e_{\mu -1}):=(\omega_{0},\cdots ,\omega_{\mu -1})P$$
Let us analyze the behaviour of $e$ with respect to $S$. By the proof of corollary \ref{prop:fonda} (1), we may assume that the basis $\omega =(\omega_{0},\cdots ,\omega_{\mu -1})$ is $\triangledown$-flat. 

\begin{lemma}\label{lemme:Sconst}
We have $S(e_i ,e_j )\in\cit [\tau ]\tau^{-n}$ for all $i$ and for all $j$.
\end{lemma}
\begin{proof}
$S(e_i ,e_j )$ depends only on $\tau$, because $P$ is a fundamental solution and because $S$ is $\nabla$-flat, and the result follows from formula (\ref{formeGinfty}).
\end{proof}

\noindent The best that we can expect is   
$S(e_{i},e_{j})\in\cit\tau^{-n}$ and this happens for instance if $S(e_{i},e_{j})=S(\omega_{i},\omega_{j})$, 
because $S(\omega_{i},\omega_{j})\in\cit\tau^{-n}$ by remark \ref{rem:metglob}. 

\begin{definition} Let $P$ be a fundamental solution. We will say that it is  {\em symmetric} if
$$P^{*}(\underline{x},-\tau )P(\underline{x}, \tau )=I$$
where $P^{*}(\underline{x},-\tau )= \eta^{-1}P(\underline{x},-\tau )^{T}\eta$,
$\eta$ being the Gram matrix of $S$ in the basis $\omega$  and $T$ denoting the transpose matrix.
\end{definition}

We will consider the following situation in section \ref{nonresonnant} (see theorem \ref{cor:casvariete}):

\begin{proposition}\label{ex:solSymfondPCQO}
Let us assume that $M=\cit^{*}$ and consider the fundamental solution $P(x_{1},\tau )=H(x_{1},\tau )e^{\tau D\ln x_{1}}$ where
$H$ is a matrix of holomorphic functions on  $\cit\times\cit$ such that $H(0,\tau )=I$ and $D$ is a constant matrix. Then $P$ is symmetric. 
\end{proposition}
\begin{proof}
Recall that $S(\omega_{i},\omega_{j})\in\cit\tau^{-n}$ for all $i$ and for all $j$.
We have
$$S(e_{i}, e_{j})=S(H(0,\tau )e^{\tau D\ln x}\omega_{i}, H(0,\tau )e^{\tau D\ln x}\omega_{j})+o(1)$$
$$=S(e^{\tau D\ln x}\omega_{i}, e^{\tau D\ln x}\omega_{j})+o(1)= S(e^{-\tau D^{*}\ln x}e^{\tau D\ln x}\omega_{i}, \omega_{j})+o(1)$$
as $x\rightarrow 0$. By lemma \ref{lemme:Sconst}, we must have $e^{\tau D^{*}\ln x}=e^{\tau D\ln x}$ 
and $S(e_{i}, e_{j})=S(\omega_{i}, \omega_{j})$: $P$ is thus symmetric.
\end{proof}

\noindent More generally, but we won't directly use this result, one can show, as in \cite[Lemma 2.5]{Dub2}, 
that there exists conformal and symmetric fundamental solutions.

\begin{remark}  
It follows from remark \ref{rem:conforme} that a conformal, symmetric fundamental solution is unique up to right multiplication by homogeneous and symmetric (i.e satisfying $\alpha (-\tau )^{*}\alpha (\tau )=I$) sections $\alpha (\tau )$.
\end{remark}

Here are two consequences of the symmetry:  

\begin{proposition} Let $P$ be a conformal, symmetric fundamental solution. Then we have 
$\tilde{R}^{*}_{k}=(-1)^{k+1}\tilde{R}_{k}$
pour $k\geq 1$ and $R_{0}^{*}=-nId-R_{0}$ in formula (\ref{matconconf}).
\end{proposition}
\begin{proof}
Transpose equality (\ref{eq:PFondConf}) taking into account equations (\ref{eq:sym}).
\end{proof}

\noindent Last, if if $P$ is symmetric and if the basis $\omega$ is orthogonal with respect to $S$, the $J$-functions of definition \ref{def:fonctionJ} are expressed as follows (the dual $\omega^{i}$ of $\omega_{i}$ is defined in remark \ref{rem:metglob}):

\begin{proposition}\label{prop:fonctionJg}
Let us assume that the fundamental solution $P$ is symmetric. Then, the functions $J^{P, \omega_{0}}$ is defined by the formula
$$J^{P, \omega_{0}}_{{\cal Q}}=\sum_{j=0}^{\mu -1}g(P(\omega_{j}), \omega_{0} )\omega^{j}$$
where $g$ is given by formula (\ref{formeg}) and we have 
$$\partial_{x_{i}}J^{P, \omega_{0}}_{{\cal Q}}=\sum_{j=0}^{\mu -1}\partial_{x_{i}}(g(P(\omega_{j}), \omega_{0} ))\omega^{j}$$
for all $i=1,\cdots ,r$.
\end{proposition}
\begin{proof}
The first equality follows from symmetry. For the second one, we have, using respectively the definition, the symmetry and the flatness of $g$, 
together with the fact that $P$ is a fundamental solution,
$$g(\omega_{j}, \partial_{x_{i}}J^{P, \omega_{0}}_{{\cal Q}})= 
g(\omega_{j}, P^{-1}\nabla^{d}_{\partial_{x_{i}}}\omega_{0})=g(P\omega_{j}, \nabla^{d}_{\partial_{x_{i}}}\omega_{0})=
\partial_{x_{i}}(g(P(\omega_{j}), \omega_{0} ))$$
and this gives the expected result.
\end{proof}

\subsection{Canonical fundamental solutions and canonical $J$-functions}
\label{sec:Jcan}

\begin{definition} We will say that the fundamental solution $P$ is {\em canonical} if it is conformal, symmetric and if moreover 
\begin{equation}\label{eq:Leveltcan}
R(\tau )=R_{0}+R_{1}\tau
\end{equation}
$R(\tau )\frac{d\tau}{\tau}$ denoting the matrix of $\nabla$ after the base change of matrix $P$.
\end{definition}

\begin{remark}
\label{rem:sectionshorizontales}
Let $P$ be a canonical fundamental solution. It follows from proposition \ref{proposition:sectionshorizontales} that a basis of flat sections takes the form $(\omega_{0},\cdots ,\omega_{\mu -1})\Psi$ where
$$\Psi (\underline{x},\tau )=P(\underline{x},\tau )\tau^{-R_{0}}\tau^{-R_{1}}\Psi^{const},$$
$\Psi^{const}$ denoting a constant matrix. 
\end{remark}

\begin{lemma}\label{lemma:compareCan}
Let $P$ be a conformal, symmetric solution and
$$R(\tau )=R_{0}+R_{1}\tau +\cdots +R_{\delta}\tau^{\delta}$$
the matrix associated with it by lemma \ref{lemma:levelt}. 
Then there exists a canonical fundamental solution if and only if there exists a homogeneous matrix
$\alpha (\tau )$ (see remark \ref{rem:conforme}) 
$$\alpha (\tau)=I+\sum_{k\geq 1}\alpha_{k}^{(-k)}\tau^{k}$$
such that
\begin{itemize}
\item $[\alpha_{k}^{(-k)},R_{1}]=R_{k+1}$ for all $k\geq 1$,
\item $\alpha^{*}(-\tau)\alpha (\tau )=I$.
\end{itemize}
In this case, $\tilde{P}=P\alpha (\tau )$  is a canonical solution and the matrix $ \tilde{R}(\tau )$ associated with $\tilde{P}$ by lemma \ref{lemma:levelt} is 
$R_{0}+R_{1}\tau$.
\end{lemma}
\begin{proof}
By remark $\tilde{P}=P\alpha (\tau )$ is a conformal and fundamental solution if and only if 
$\alpha (\tau )=\sum_{i\geq 0}\alpha_{i}^{(-i)}\tau^{i}$ is a homogeneous matrix. 
Without loss of generality (replace $\alpha$ by $\alpha_0^{-1}\alpha$ if necessary) we can assume that $\alpha_0 =I$.
The matrix $\tilde{R}(\tau)$ attached to $\tilde{P}$ by lemma \ref{lemma:levelt} is
$\tilde{R}_{0}+\tilde{R}_{1}\tau$ if and only if
\begin{equation}
\tau\alpha '=\alpha (\tilde{R}_{0}+\tilde{R}_{1}\tau )-(R_{0}+R_{1}\tau +\cdots +R_{\delta}\tau^{\delta} )\alpha
\end{equation}
The constant term gives $\tilde{R}_{0}=R_{0}$, the one of degree $1$ (in $\tau$) yields  
$$(\mbox{Ad} R_{0}+I)(\alpha_{1}^{(-1)})=\tilde{R}_{1}-R_{1}$$
and more generally the one of degree $k$ in $\tau$ ($k\geq 2$)
$$(\mbox{Ad} R_{0}+kI)(\alpha_{k}^{(-k)})=[\alpha_{k-1}, \tilde{R}_{1}]-R_{k}.$$
The first assertion follows because $\alpha$ is homogeneous. follows that $\alpha_{1}\in V^{(-1)}$ and then $\tilde{R}_{1}=R_{1}$.
The last assertion about symmetry is clear.
\end{proof}

\noindent A canonical solution is thus unique up to multiplication by constant homogeneous matrices. More precisely,

\begin{corollary}
Assume that $\tilde{P}$ and $\tilde{Q}$ are two canonical solutions. The matrices of the connection $\nabla$ in the bases $\omega\tilde{P}$ and $\omega\tilde{Q}$
are respectively $(R_{0}+R_{1}\tau )\frac{d\tau}{\tau}$ and $\alpha_0^{-1}(R_{0}+R_{1}\tau )\alpha_0 \frac{d\tau}{\tau}$ where $\alpha_0$ is a homogeneous matrix of degree $0$. 
\end{corollary}
In particular, formula (\ref{eq:Leveltcan}) is unique up to conjugation by a constant 
homogeneous matrix.\\

The following definition is now natural (recall the pre-primitive sections defined in \ref{def:preprimcan}) :

\begin{definition}\label{def:fonctionsJcan} The {\em canonical} $J$-functions are
$J_{{\cal Q}}^{P, \omega_{0}}=P^{-1}\omega_{0}$
where $P$ is a canonical fundamental solution and $\omega_{0}$ is a canonical pre-primitive section of the quantum differential system ${\cal Q}$.
\end{definition}

\section{Non-resonant logarithmic quantum differential systems}

\label{nonresonnant}
The goal of this section is to show that there exist explicit canonical fundamental solutions under the assumption that the quantum differential systems are {\em logarithmic} and {\em non-resonant}. It happens to be the case for systems are associated with the small quantum cohomology described in example \ref{exemplebasiqueA}, thanks to the divisor axiom. Notice that logarithmic Frobenius manifolds have been defined by T. Reichelt \cite{R}.

\subsection{Logarithmic quantum differential systems}
\label{sec:SDQL}

Let ${\cal Q}=(M, {\cal G}, \nabla ,S, d)$ be a quantum differential system. We will use the following version of definition \ref{def:LogSDQ}.

\begin{definition}\label{def:SDQL}
We will say that ${\cal Q}$ is {\em logarithmic} if $M=(\cit )^{r+1}$ and if its characteristic equation
(definition \ref{def:eqcar}) has the form
\begin{equation}\label{matCon}
\sum_{i=0}^{r}M^{(i)}(\underline{x},\tau )\frac{dx_i}{x_i}+N(\underline{x},\tau )\frac{d\tau}{\tau}
\end{equation}
where the matrices $M^{(i)}(\underline{x}, \tau )$ and $N(\underline{x},\tau )$ are matrices of holomorphic
fonctions on $\cit^{r+1}\times\cit$ and $\underline{x}=(x_0 ,\cdots ,x_r )\in M$.
\end{definition}

\noindent One can of course replace $M$ by an open neighbourhood of the origin in $(\cit )^{r+1}$.\\

It follows from proposition \ref{keyproposition} that
\begin{equation}\label{eq:Mi}
M^{(i)}(\underline{x},\tau )=M_{0}^{(i)}(\underline{x})+M_{1}^{(i)}(\underline{x})\tau \ \mbox{and}\ N(\underline{x},\tau )=N_{0}(\underline{x})+N_{1}(\underline{x})\tau
\end{equation}
if the quantum differential system ${\cal Q}$ is logarithmic.

\begin{example}
The quantum differential systems considered in examples \ref{exemplebasiqueB} and \ref{exemplebasiqueA} are logarithmic. Rescalings 
(see section \ref{sec:rescaling}) provide other examples of such systems.
\end{example}

\begin{remark}\label{rem:classlim}
The classical limit (as $x\rightarrow 0$) of a logarithmic differential system ${\cal Q}$ on $M=\cit^*$ is a quantum differential system on a point, that is a tuple
$${\cal Q}^{cl}=({\cal G}^{cl}, \nabla^{cl} ,S^{cl}, d^{c})$$
satisfying the conditions of definition \ref{def:SDQ}.
See section \ref{sec:limite} below for a discussion about this. 
\end{remark}

\subsection{Non-resonant logarithmic quantum differential systems on curves}
\label{sec:NonResCurves}
Let ${\cal Q}=(M, {\cal G}, \nabla ,S, d)$ be a logarithmic quantum differential system on $M=\cit$ (definition \ref{def:SDQL} with $r=0$) with pole at the origin: its characteristic equation is, in the basis $\omega =(\omega_{0},\cdots ,\omega_{\mu -1})$,
\begin{equation}
M(x,\tau )\frac{dx}{x}+N(x,\tau )\frac{d\tau}{\tau}
\end{equation}
where $M(x, \tau )=M_0 (x)+M_1(x)\tau$ and $N(x,\tau )=N_0 (x)+N_1 (x)\tau$ are matrices of holomorphic functions
on $\cit\times\cit$ (it could also be on $(\cit ,0)\times\cit$, see above). We will also write
$$N(x,\tau )=-(A_{0}(x)\tau +A_{\infty}(x))$$
to match with the notations of section \ref{sec:SDQ}.

\begin{definition}
The logarithmic quantum differential system ${\cal Q}$ is {\em non-resonant} at $\tau_{0}$ if the eigenvalues of $M(0,\tau_{0})$ do not differ from a non-zero integer.
\end{definition}

\noindent  Note that the quantum differential system is non-resonant at $\tau_{0}=0$ if $M_{0}(0)=0$, in particular if $M_{0}(x)=0$ for all $x$. This will be our favorite situation:

\begin{definition}\label{eq:FlatQDS}
We will say that the logarithmic quantum differential system ${\cal Q}$ is {\em flat} if $M_{0}(x)=0$ for all $x$.
\end{definition}

\noindent The word ``flat'' recalls the flatness with respect to the residual connection $\bigtriangledown$, see section \ref{sec:SDQ}.

\subsection{Fundamental solutions of a non-resonant logarithmic quantum differential system}

\label{sec:principale}
The main result of this section (theorem \ref{cor:casvariete} below) is a variation of \cite[Isomonodromicity Theorem]{Dub2}.
Let ${\cal Q}$ be a logarithmic quantum differential system.

\begin{lemma}\label{lemma:vpconst}
The eigenvalues of $M(0,\tau )$ do not depend on $\tau$.
\end{lemma}
\begin{proof}
Indeed, by isomonodromy (the connection is flat), the eigenvalues of the monodromy around $x=0$ do not depend on $\tau$. 
\end{proof}

\noindent In particular, it follows that $M_{1}(0)$ is nilpotent if $M_{0}(0)=0$.

\begin{lemma}\label{lemma:nonres} Let us assume that ${\cal Q}$ is non-resonant at $\tau_{0}$.
\begin{enumerate}
\item The matrix $M(0,\tau )$ is non-resonant for all $\tau\in U:=\cit$.
\item There exists a matrix $H(x,\tau )$ of holomorphic function on $(\cit ,0)\times U$, uniquely determined by the initial condition $H(0,\tau )=I$, such that, after the base change of matrix $H$, the matrix of the connection $\nabla$ takes the form
\begin{equation}\label{matConBis}
M(0,\tau )\frac{dx}{x}+V(\tau )\frac{d\tau}{\tau}
\end{equation}
$V(\tau )$ being a matrix of holomorphic functions on $U$. If moreover $0\in U$ then 
$$V(\tau )=N(0, \tau)=N_{0}(0)+N_{1}(0)\tau.$$
\item The matrix 
$P(x,\tau )=H(x,\tau )e^{-M(0,\tau )\ln x}$
is a fundamental solution of the Dubrovin connection.
\end{enumerate}
\end{lemma}
\begin{proof} 1. follows from lemma \ref{lemma:vpconst}. The proof of 2. is classical, but we give some details in order to set the notations and to write 
down explicitely the equations that we will use later  (mainly equations (\ref{eq:DiffxH}), (\ref{eq:DiffxHd}) and (\ref{eq:commutation})): 
since the matrix $M(0,\tau )$ is non-resonant for all $\tau\in U$, there exists a unique matrix
$$H(x,\tau )=I+\sum_{i\geq 1}H^{i}(\tau )x^{i},$$
defined on $\cit\times U$, such that  
\begin{equation}\label{eq:DiffxH}
x\frac{\partial H}{\partial x}(x,\tau )=H(x,\tau )M(0,\tau )-M(x,\tau )H(x,\tau )
\end{equation}
This can be shown for instance as in \cite[Proposition 2.11]{Sab1}, separating the degrees in $x$: equation  (\ref{eq:DiffxH}) is then equivalent to
\begin{equation}\label{eq:DiffxHd}
dH^{d}(\tau )=H^{d}(\tau )M(0,\tau )-M(0,\tau )H^{d}(\tau )-\sum_{i=1}^{d}M^{i}(\tau)H^{d-i}(\tau)
\end{equation}
for $d\geq 1$ and $M(x,\tau)=M(0,\tau)+\sum_{i\geq 1}M^{i}(\tau)x^{i}$. 
The non-resonant assumption shows that these equations are solved in an unique way because $\mbox{Ad} M(0,\tau) +dI$ is invertible for all integer $d$.
After the base change of matrix $H$, the matrix of the connection is 
$$M(0,\tau )\frac{dx}{x}+V(x,\tau)\frac{d\tau}{\tau}.$$
Equation (\ref{eq:DiffxH}), together with a classical argument of regularity, shows that $H(x,\tau )$ holomorphic on $(\cit ,0)\times U$, 
because $M(x,\tau )$ is holomorphic on $\cit\times U$. In particular, $V(x,\tau )$ is also holomorphic on $(\cit ,0)\times U$.
Now, it follows from the flatness of the connection that
$$\tau\frac{\partial M}{\partial\tau}(0,\tau )-x\frac{\partial V}{\partial x}(x,\tau )=[M(0,\tau ), V(x,\tau )].$$
This gives, putting $V(x,\tau )=V(\tau )+\sum_{k\geq 1}V^{k}(\tau )x^{k}$ and comparing the degrees in $x$,
\begin{equation}\label{eq:commutation}
\tau M_{1}(0)=[M_{0}(0)+\tau M_{1}(0), V(\tau )]
\end{equation}
for $k=0$
and
$$(\mbox{Ad}  M(0,\tau )+kI)V^{k}(\tau )=0$$
for $k\geq 1$ and
for all $\tau\in U$.  By assumption $\mbox{Ad} M(0,\tau )+kI$ is invertible for $k\geq 1$, and we finally get
$V^{k}(\tau )=0$ for all $\tau\in U$ and all $k\geq 1$. This shows that $V(x,\tau )=V(\tau)$ and (\ref{matConBis}) follows.

Let us show the last assertion and let us assume that $\tau =0\in U$.
We also have
\begin{equation}\label{eq:DiffTauH}
\tau\frac{\partial H}{\partial \tau}(x,\tau )=H(x,\tau )V(\tau )-N(x,\tau )H(x,\tau ).
\end{equation}
Writing $V(\tau)=\sum_{k\geq 0}V_{k}\tau^{k}$ and $H(x,\tau )=\sum_{i\geq 0}H_{i}(x)\tau^{i}$ (this is possible because $\tau =0\in U$),
this equation gives, separating now the degrees in $\tau$,
\begin{equation}\label{eq:xH}
kH_{k}(x)=H_{0}(x)V_{k}-N_{k}(x)H_{0}(x)+\Psi_{k}(x)
\end{equation}
for all $k\geq 0$ (recall that $N_{k}(x)=0$ for $k\geq 2$) with $\Psi_{0}(x)=0$ and
$$\Psi_{k}(x)=\sum_{i=1}^{k}(H_{i}(x)V_{k-i}-N_{k-i}(x)H_{i}(x))$$
 for $k\geq 1$. It follows that 
$V_{k}=N_{k}(0)$ for all $k\geq 0$ because $H_{i}(0)=0$ for all $i\geq 1$ and $H_{0}(0)=I$, thanks to the initial condition $H(0,\tau )=I$. 
This completes the proof of 2. and
3. follows.
\end{proof}

\begin{corollary}\label{cor:casorbifolde} Assume that the eigenvalues of $M_{0}(0)$ are contained in an interval of length strictly smaller than  $1$.
\begin{enumerate}
\item There exists a matrix $H(x,\tau )$ of holomorphic functions on $(\cit ,0)\times\cit$, uniquely determined by the initial condition $H(0,\tau )=I$, such that the matrix of  $\nabla$ takes the form
$$(M_{0}(0)+M_{1}(0)\tau )\frac{dx}{x}+(N_{0}(0)+N_{1}(0)\tau )\frac{d\tau}{\tau}$$
after the base change of matrix $H$. The matrix
$P(x,\tau )=H(x,\tau )e^{-(M_{0}(0)+\tau M_{1}(0))\ln x}$
is a fundamental solution of the Dubrovin connection.
\item We have the relations
\begin{equation}\label{RelInt}
[N_{0}(0), M_{0}(0)]=[N_{1}(0), M_{1}(0)]=0\ \mbox{and}\ [ N_{0}(0), M_{1}(0)]+[N_{1}(0), M_{0}(0)]=-M_{1}(0)
\end{equation}
\end{enumerate}
\end{corollary}
\begin{proof}Follows from lemma \ref{lemma:nonres} because
the assumption on $M_{0}(0)$ shows that the quantum differential system is non-resonant at $\tau_{0}=0$.
The commutation relations follows from formula (\ref{eq:commutation}).
\end{proof}

Of course, the fundamental solution $P$ in corollary \ref{cor:casorbifolde} does not need to be conformal, neither symmetric. 
If $M_{0}(x)=0$ for all $x$\footnote{In this case $M_{1}(0)$ is nilpotent, see lemma \ref{lemma:vpconst}.}, the situation becomes better:

\begin{theorem}\label{cor:casvariete} Let ${\cal Q}$ be a flat logarithmic quantum differential system.
\begin{enumerate}
\item There exists a matrix $H(x,\tau )$ of holomorphic functions on $(\cit ,0)\times\cit$, uniquely determined by the initial condition $H(0,\tau )=I$, such that the matrix
$P(x,\tau )=H(x,\tau )e^{-\tau M_{1}(0)\ln x}$
is a fundamental solution of the Dubrovin connection.
\item Any fundamental solution takes the form $H(x,\tau )e^{-\tau M_{1}(0)\ln x}P_{cl}(\tau )$ where $P_{cl}(\tau )$ is a matrix depending only on  $\tau$. 
\item The matrix of the connection takes the form
\begin{equation}\label{eq:MatSolFonda}
(N_{0}(0)+N_{1}(0)\tau )\frac{d\tau}{\tau}
\end{equation}
after the base change of matrix $P$.
\item The fundamental solution $P$ is symmetric and we have
$$\tau\partial_{\tau}S(e_{i},e_{j})=S(\tau\nabla_{\partial_{\tau}}e_{i},e_{j})+S(e_{i},\tau\nabla_{\partial_{\tau}}e_{j})$$
if $(e_{0},\cdots ,e_{\mu -1})=(\omega_{0},\cdots ,\omega_{\mu -1})P$.
\item Assume that $M_{1}(0)=cN_{1}(0)$ for some non-zero constant $c$. Then the fundamental solution $P$ is conformal.  
\end{enumerate}
\end{theorem}
\begin{proof} 1. Follows from corollary \ref{cor:casorbifolde} because $M_0 (0)=0$ and
2. then follows from remark \ref{rem:compsolfond}. For
3. we can proceed as follows: after the base change of matrix $P(x,\tau )$, the matrix of the connection takes the form
$$[e^{\tau M_{1}(0)\ln x}(N_{0}(0)+\tau N_{1}(0))e^{-\tau M_{1}(0)\ln x}-\tau M_{1}(0)\ln x ]\frac{d\tau}{\tau}.$$
Now, we have $[N_{0}(0), M_{1}(0)]=-M_{1}(0)$ by relations (\ref{RelInt}) because $M_{0}(0)=0$ and this yields 
$$e^{\tau M_{1}(0)\ln x}N_{0}(0)=N_{0}(0)e^{\tau M_{1}(0)\ln x}+\tau M_{1}(0)e^{\tau M_{1}(0)\ln x}\ln x $$
and we get the expected formula using then relation $[N_{1}(0), M_{1}(0)]=0$.\\
4. The first assertion thus follows from proposition \ref{ex:solSymfondPCQO} because $S(\omega_{i},\omega_{j})\in\cit\tau^{-n}$ for all $i$ and for all $j$ if $M_{0}(x)=0$ (see remark \ref{rem:metglob}).
Since $P$ is symmetric, we have 
$$\tau\partial_{\tau}S(e_{i}, e_{j})=\tau\partial_{\tau}S(\omega_{i}, \omega_{j})=S(\tau\nabla_{\partial_{\tau}}\omega_{i}, \omega_{j})
+S(\omega_{i}, \tau\nabla_{\partial_{\tau}}\omega_{j})$$
and, by formula (\ref{eq:MatSolFonda}), the right hand side is equal to $S(\tau\nabla_{\partial_{\tau}}e_{i}, e_{j})
+S(e_{i}, \tau\nabla_{\partial_{\tau}}e_{j})+o(1)$
as $x$ tends to $0$. We conclude using the fact that $S(e_{i},e_{j})$ does not depend on $x$ (see lemma \ref{lemme:Sconst}).\\
5. Indeed, we have $[N_{0}(0),M_{1}(0)]=-M_{1}(0)$ by relations (\ref{RelInt}) because $M_{0}(0)=0$. 
\end{proof}

\begin{remark} The assumption in item 5 will be satisfied in the geometric situations considered below (small quantum cohomology and/or its mirror partner 
on the $B$-side), see section \ref{ex:CohQuant} below. 
\end{remark}

We will use mainly the following corollary:

\begin{corollary}\label{coro:sechor}Let ${\cal Q}$ be a flat quantum differential system.
\begin{enumerate}
\item Any basis of flat sections of  $\nabla$ takes the form $(\omega_{0},\cdots ,\omega_{\mu -1})\Psi$ where
\begin{equation}\label{sectionhorizontalevariete}
\Psi (x,\tau)=H(x,\tau)e^{-\tau M_{1}(0)\ln x}\Psi^{cl}(\tau) ,
\end{equation}
the matrix $\Psi^{cl}(\tau)$ satisfying 
$$\tau\partial_{\tau}\Psi^{cl}(\tau)=-(N_{0}(0)+\tau N_{1}(0))\Psi^{cl}(\tau).$$
\item If moreover $P(x,\tau )=H(x,\tau)e^{-\tau M_{1}(0)\ln x}$ is conformal, we have
$$\Psi^{cl}(\tau)=\tau^{-N_{0}(0)}\tau^{-N_{1}(0)}\Psi^{const}$$
where $\Psi^{const}$ is a constant matrix.
\end{enumerate}
\end{corollary}
\begin{proof}
1. Follows from
theorem \ref{cor:casvariete} (items 2. and 3). We then get 2. using proposition \ref{proposition:sectionshorizontales}. 
\end{proof}

\begin{corollary}\label{coro:DLJ}Let ${\cal Q}$ be a flat quantum differential system and $P$ be a fundamental solution as in theorem \ref{cor:casvariete} 1. 
Then
$$J^{P, \omega_{0}}_{{\cal Q}}=e^{\tau M_{1}(0)\ln x}(\omega_{0}+O(\tau))$$
\end{corollary}

On the A-side (example \ref{exemplebasiqueA}) we even have the asymptotic expansion, for small quantum cohomology 
(and we assume here that the Picard group of $X$ is of rank one),
$$J_{X}=e^{\tau \cup_{p} \ln q}(1+o(\tau))$$
where $\cup_{p}$ stands for the cup-product by the generator of $H^{2}(X,\cit )$ and $q=e^{t}$.
The previous formula is important in order to show for instance Givental's mirror formula $I_{X}=J_{X}$ for $X$ a 
hypersurface of degree $\ell <n$ in $\ppit^{n}$, see \cite[Section 11.2]{CK}..

\subsection{Classical limit of a non-resonant quantum differential system}
\label{sec:limite}

The classical limit (as $x\rightarrow 0$) of a logarithmic differential system ${\cal Q}$ on $M=\cit$ is a quantum differential system on a point, 
that is a tuple
$${\cal Q}^{cl}=({\cal G}^{cl}, \nabla^{cl} ,S^{cl},d)$$
satisfying the conditions of definition \ref{def:SDQ}, see also remark \ref{rem.MiscSDQ}. We explain here why such a limit exists, how to compute it and why 
it produces a meromorphic connection with regular singularity at $\theta =0$ in the geometric situation.

\subsubsection{Flat case}\label{sec:flatlimite}
The classical limit of a {\em flat} quantum differentiel system ${\cal Q}$ is the tuple
$${\cal Q}^{cl}=({\cal G}^{cl}, \nabla^{cl} ,S^{cl}, d)$$
where\\

$\bullet$ ${\cal G}^{cl}$ is the ${\cal O}_{\ppit^{1}}$-free module 
associated with the lattice $(G_{0}^{cl},G_{\infty}^{cl})$ (see \cite[I, proposition 4.15]{Sab1}) where $G_{0}^{cl}$ ({\em resp.} $G_{\infty}^{cl}$) is the 
$\cit [\theta]$ ({\em resp.}\footnote{As usual, $\tau :=\theta^{-1}$.} $\cit[\tau]$) free-module
generated by the 
classes $[\![\omega_{i}]\!]$ of the sections $\omega_{i}$ in ${\cal L}/x{\cal L}$, ${\cal L}$ being the $\cit \{x\}[\theta ,\theta^{-1}]$-module generated by $\omega$,\\

$\bullet$ $\nabla^{cl}$ is the connection whose matrix is 
\begin{equation}\label{eq:ConnexionLimite}
(\frac{A_{0}(0)}{\theta}+A_{\infty}(0))\frac{d\theta}{\theta}
\end{equation} 
in the basis $[\![\omega]\!]$ of ${\cal G}^{cl}$,\\

$\bullet$ $S^{cl}([\![\omega_{i}]\!],[\![\omega_{j}]\!])=S(\omega_{i},\omega_{j})$ for all $i$ and for all $j$.\\

\noindent The tuple ${\cal Q}^{cl}$ satisfies the conditions of definition \ref{def:SDQ}: $S^{cl}$ is $\nabla^{cl}$-flat because of theorem \ref{cor:casvariete} (4). In the setting of example \ref{exemplebasiqueA}, ${\cal G}^{cl}$ is the trivial bundle on $\ppit^{1}$ whose fibers are $H^{*}(X,\cit )$
and the limit metric $S^{cl}$ is the usual cup-product.

Assume moreover that $[A_{\infty}(0),A_{0}(0)]=A_{0}(0)$ (conformality). Then, and after the base change of matrix $\theta^{-D}$, where $D$ is the diagonal matrix 
whose eigenvalues are the integral part of the ones of $A_{\infty}$, system  (\ref{eq:ConnexionLimite})
becomes
$$(A_{\infty}(0)-D +A_{0}(0))\frac{d\theta}{\theta}$$
In particular, the meromorphic connection $(G_{0}^{cl}[\theta^{-1}],\nabla^{cl})$ has a regular singularity at the origin.

\subsubsection{General case}
Assume now that ${\cal Q}$ is non-resonant at $\tau_{0}$ but $M_{0}(0)\neq 0$. In this case, the definition of the limiting quantum differential system 
is more complicated: one has to take into account a monodromy phenomenon and to work in a graded module with respect to the $V$-filtration in order to 
reduce to the previous situation (the difficult point is to get a ``limit'' bilinear form), and this is in fact what we do in example 
\ref{ex:PetiteCohQuantOrbi}. 

In order to make the link with section \ref{sec:flatlimite}, assume that $M_{0}(0)= 0$.
Let us first notice that, by lemma \ref{lemma:vpconst}, the eigenvalues of $M_{1}(0)$ are all equal to zero.
Let $V^{\bullet}$ be the Malgrange-Kashiwara $V$-filtration of the Gauss-Manin system $G$ at $x=0$ (see Appendix): we thus have $V^{\alpha}=V^{0}$ 
(which is, by definition the $\cit \{x\}[\theta ,\theta^{-1}]$-module generated by $\omega$) for $\alpha\leq 0$ and $V^{\alpha}=xV^{\alpha -1}$ for $\alpha>0$.
Finally, ${\cal L}/x{\cal L}=V^{0}G\cap G_{0}/V^{1}G\cap G_{0}$.

This construction is explained in detail in the case of the rescalings in section \ref{sec:rescaling}, and this will be after all the model for small quantum
cohomology (see also \cite{DoMa} in the case of the small quantum cohomology of weighted projective spaces).

\subsection{The $J$-functions of a non-resonant quantum differential system}

We are now able to define the $J$-functions of a non-resonant quantum differential system, removing the ambiguity on the choice of the fundamental solution:

\begin{definition}\label{def:can} Let ${\cal Q}$ be a non-resonant quantum differential system $\tau_{0}$ and $\omega_{0}$ be a pre-primitive section.
The $J$-function $J^{\omega_{0}}_{\cal Q}$ of ${\cal Q}$ is the function defined on $\cit\times U$ by
\begin{equation}\label{fonctionJ}
J^{\omega_{0}}_{\cal Q}=P^{-1}\omega_{0}
\end{equation}
where $P=H(x,\tau )e^{-M(0,\tau )\ln x}$ is the fundamental solution given by lemma \ref{lemma:nonres}.
\end{definition}

\noindent The function $J^{\omega_{0}}$ is characterized by the initial condition $H(0,\tau )=I$ (notice then that $J^{\omega_{0}}_{\cal Q}\sim e^{M(0,\tau )\ln x}\omega_{0}$ as $x\rightarrow 0$) and depends only on the choosen pre-primitive section. It other words, under the assumptions of theorem \ref{cor:casvariete}, $J^{\omega_{0}}$ will be canonical if $\omega_{0}$ is so. This is what happens in examples \ref{exemplebasiqueB} and \ref{exemplebasiqueA}: in the situation
of example \ref{exemplebasiqueA}, Givental's $J$-function is the canonical $J$-function in the previous sense, taking the section $1$ as canonical pre-primitive section (see example \ref{exemple:PrePrimCan}).

\subsection{Higher dimensional case}
\label{sub:geq2}

If $M=\cit^{r+1}$ with $r$ greater or equal to $2$ (in the geometric situations alluded above, this happens if the rank of the cohomology group $H^{2}$ is greater or equal to $2$), one has analogous results in the case of a logarithmic quantum differential system ${\cal Q}$ having a characretistic equation of the form
\begin{equation}
M^{(0)}(\underline{x}, \tau )\frac{dx_{0}}{x_{0}}+\cdots +M^{(r)}(\underline{x},\tau )\frac{dx_{r}}{x_{r}}+N(\underline{x},\tau )\frac{d\tau}{\tau}
\end{equation}
where we put as above $\underline{x}=(x_{1},\cdots ,x_{r})$ (this kind of quantum differential system is produced by the functions considered in \ref{rem:FonctionsMirroir} (1); see also section \ref{sec:hirzebruch} below). It follows from equations (\ref{eq:Mi}) that
\begin{equation}
M^{(i)}(\underline{x},\tau )=M_{0}^{(i)}(\underline{x})+M_{1}^{(i)}(\underline{x})\tau \ \mbox{et}\ N(\underline{x},\tau )=N_{0}(\underline{x})+N_{1}(\underline{x})\tau
\end{equation}
where the matrices involved are matrices of holomorphic functions. We will say that the quantum logarithmic differential system ${\cal Q}$ is {\em flat} if $M_{0}^{(i)}(\underline{x})$ is identically equal to $0$ for all $i=0,\cdots ,r$.\\

\begin{lemma} Assume that the quantum logarithmic differential system ${\cal Q}$ is flat\footnote{The condition $M_{0}^{(i)}(0)=0$ for all $1\leq i\leq r$ would be enough}.
\begin{enumerate}
\item There exists an invertible matrix $H(\underline{x},\tau )$, characterized by the initial condition $H(0,\tau )=I$, such that the matrix of $\nabla$ takes the form, after the base change of matrix $H(\underline{x},\tau )$,
\begin{equation}\label{matConBisgeq2}
\tau M^{(1)}_{1}(0) \frac{dx_{1}}{x_{1}}+\cdots +\tau M^{(r)}_{1}(0)\frac{dx_{r}}{x_{r}}+V(\tau )\frac{d\tau}{\tau}
\end{equation}
where $V(\tau )=N_{0}(0)+N_{1}(0)\tau$.
\item Every fundamental solution of the Dubrovin connection takes the form
$$H(\underline{x},\tau )e^{-\tau \sum_{i}M^{(i)}_{1}(0)\ln x_{i}}P_{cl}(\tau )$$
where $P_{cl}(\tau )$ is a matrix depending only on $\tau$.
\end{enumerate}
\end{lemma}
\begin{proof} The eigenvalues of the residue matrices along $\{x_{i}=0\}$, $i=0,\cdots ,r$, are constant and the assumption of the lemma shows that they are all equal to $0$.
In particular they do not differ from a non-zero integer. Notice that, due to the flatness of the connection, the matrices $M^{(i)}_{1}(0)$ and $M^{(j)}_{1}(0)$ commute.
\end{proof}

\begin{definition} 
\begin{enumerate}
\item The fundamental solution $H(\underline{x},\tau )e^{-\tau\sum_{i}M^{(i)}_{1}(0)\ln x_{i}}$ is called a {\em canonical fundamental solution} of the Dubrovin connection of the quantum differential system ${\cal Q}$.
\item Let $\omega_{0}$ be pre-primitive. The $J$-functions of the quantum differential system ${\cal Q}$ are the sections
$$J^{\omega_{0}}_{\cal Q}=P^{-1}\omega_{0}$$
where $P$ is a canonical fundamental solution of the Dubrovin connection.
\end{enumerate}
\end{definition}

\section{Examples of non-resonant quantum differential systems and canonical fundamental solutions}\label{ex:CohQuant}
We discuss here the existence of quantum differential systems and canonical fundamental solutions in the geometric setting.

\subsection{Rescalings} 
\label{sec:rescaling}
Mirrors models for small quantum cohomology are in essence produced (at least in the toric case) by avatars of {\em rescalings}, 
see remark \ref{rem:RescvsBmodels}. Also, these rescalings give a quite good picture of what happens in general and this is why we first focuse on them.

\subsubsection{Definitions}
We will use notations and definitions of the Appendix. Let $f:U\rightarrow\cit$ be a regular tame function on the affine manifold $U$, 
equipped with the coordinates $\underline{u}$,
$G_{0}^{o}$ ({\em resp.} $G^{o}$) be (the Laplace transform of) its Brieskorn lattice ({\em resp.} (the Laplace transform of) its Gauss-Manin system). 
The Gauss-manin system $G^{o}$ is equipped with a flat meromorphic connection $\nabla^{o}$ and $G_{0}^{o}$ is stable under 
$\theta^{2}\nabla^{o}_{\partial_{\theta}}$.\\

Let us define 
$$F:U\times M\rightarrow\cit$$
where $M=\cit^{*}$ and $F(\underline{u},x)=xf(\underline{u})$. This is a {\em rescaling} of $f$. Let $G_{0}$ ({\em resp.} $G$) be the Brieskorn lattice 
({\em resp.} the Gauss-Manin system) of $F$:
$$G_{0}:=\frac{\Omega^{n}(U)[x, x^{-1},\theta ]}{(\theta d_{u}-d_{u}F)\wedge \Omega^{n-1}(U)[x, x^{-1},\theta]}\ 
\mbox{and}\ G:=\frac{\Omega^{n}(U)[x, x^{-1},\theta ,\theta^{-1}]}{(\theta d_{u}-d_{u}F)\wedge \Omega^{n-1}(U)[x, x^{-1},\theta ,\theta^{-1}]}$$
where $d_{u}$ means that the differential is taken with respect to $\underline{u}$ only.
We have \footnote{$G^{o}$ is associated with the kernel $e^{-f/\theta}$ while $G$ is associated with the kernel $e^{-xf/\theta}$}
$$G_{0}=\pi^{+}G_{0}^{o}=\cit [x,x^{-1},\theta ]\otimes G_{0}^{o}.$$
\noindent where $\pi :(\theta ,x)\mapsto \frac{\theta}{x}$. The Gauss-Manin system $G$ is a free $\cit [x,x^{-1},\theta ,\theta^{-1}]$-module, equipped with a meromorphic 
flat connection $\nabla$: if $\Omega^{o}$ denotes the matrix of $\nabla^{o}$ in the basis $\omega^{o}$ then $\pi^{*}\Omega^{o}$ will be the one of $\nabla$ in 
the basis $\omega:=1\otimes \omega^{o}$. In general, and because $\theta^{2}\nabla^{o}_{\partial_{\theta}}G_{0}^{o}\subset G_{0}^{o}$, 
the matrix of $\nabla^{o}$ in the basis $\omega^{o}$ is
$$ \Omega^{o}=(\frac{A_{0}}{\theta}+A_{\infty}+A_{1}\theta +\cdots +A_{r}\theta^{r})\frac{d\theta}{\theta}$$
for some $r\geq 0$ and
the matrix of $\nabla$ in the basis
$\omega$ will be 
$$ \Omega=(\frac{xA_{0}}{\theta}+A_{\infty}+A_{1}\frac{\theta}{x} +\cdots +A_{r}\frac{\theta^{r}}{x^{r}})(\frac{d\theta}{\theta}-\frac{dx}{x}).$$
Of course, the case $r=0$ will give the connection of a logarithmic quantum differential system.

\begin{remark}\label{rem:RescvsBmodels}
Given positive integers $d_{1},\cdots ,d_{n}$ and $(i_{1},\cdots ,i_{n})\in\zit^{n}$, consider the function
$$F(\underline{u},x )=u_{1}^{d_{1}}+\cdots +u_{n}^{d_{n}}+x u_{1}^{i_{1}}\cdots u_{n}^{i_{n}}$$
where $x$ is a non-zero complex parameter.
Assume  that $\frac{i_{1}}{d_{1}}+\cdots + \frac{i_{n}}{d_{n}}\neq 1$ and  consider the change of variables
$$\underline{v}=(v_{1},\cdots ,v_{n})=(x^{-r/d_{1}}u_{1},\cdots , x^{-r/d_{n}}u_{n})$$
where $r=\frac{1}{1-(\frac{i_{1}}{d_{1}}+\cdots + \frac{i_{n}}{d_{n}})}$.
Then, 
$$F(\underline{v},x)=x^{r}f(\underline{v})$$
with $f(\underline{v})=v_{1}^{d_{1}}+\cdots +v_{n}^{d_{n}}+v_{1}^{i_{1}}\cdots v_{n}^{i_{n}}$. In other words, $F$ can be expressed as a rescaling of $f$.
This applies in the situation of example \ref{exemplebasiqueB} where $F(\underline{u},x )=u_{1}+\cdots +u_{n}+\frac{x}{u_{1}^{w_{1}}\cdots u_{n}^{w_{n}}}$: 
we have $F(\underline{v},\lambda )=x^{\frac{1}{\mu}}f(\underline{v})$ where $f(\underline{v})=v_{1}+\cdots +v_{n}+\frac{1}{v_{1}^{w_{1}}\cdots v_{n}^{w_{n}}}$,
$\underline{v}=(v_{1},\cdots ,v_{n})=x^{-1/\mu}(u_{1},\cdots ,u_{n})$ and $\mu =1+w_{1}+\cdots +w_{n}$.
\end{remark}

\subsubsection{The quantum differential system associated with a rescaling}

Let $\omega^{o}=(\omega^{o}_{0},\cdots ,\omega^{o}_{\mu -1})$ be the canonical solution of the Birkhoff problem for the Brieskorn lattice of $f$ 
defined in the Appendix. Let us recall the two main properties of this solution: the matrix of $\nabla^{o}$ takes the form  
\begin{equation}\label{Bir}
(\frac{A_{0}^{o}}{\theta}+A_{\infty})\frac{d\theta }{\theta }.
\end{equation}
in this basis and 
\begin{equation}\label{eq:So}
S^{o}(\omega^{o}_{i},\omega^{o}_{j})\in\cit\theta^{n} 
\end{equation}
for all $i,j$ where
$S^{o}:G_{0}^{o}\times j^{*}G_{0}^{o}\rightarrow\cit[\theta ]\theta^{n}$
is the non-degenerate, $\nabla^{o}$-flat, symmetric bilinear
form defined in step 3 of the Appendix (the involution $j$ is defined in definition \ref{def:SDQ}).\\

The matrix $A_{\infty}$ is diagonal,
$$A_{\infty}=diag (\alpha_{0},\cdots ,\alpha_{\mu -1})$$ 
where $\alpha_{0}\leq\cdots \leq\alpha_{\mu -1}$ is the ordered (unless specified) spectrum at infinity (the spectrum of a limit mixed Hodge structure,
see \cite{Sab3}) of the function $f$. Due to the $\nabla^{o}$-flatness of $S^{o}$ and formula (\ref{eq:So}), we can arrange the $\alpha_{i}$´s in such a way 
that
\begin{equation}\label{eq:RangeAlpha}
\alpha_{\ell}+\alpha_{\mu -1-\ell}=n 
\end{equation}
for $\ell =0,\cdots ,\mu -1$.
It should be emphasized that the matrix $A_{0}^{o}$ is not nilpotent in general: this happens for instance if $f$ has $\mu$ distinct critical values and
$\nabla^{o}$ has then an {\em irregular} singularity at $\theta =0$.\\

By construction, the basis $\omega^{o}$ is adapted to the Kashiwara-Malgrange $V$-filtration $V^{\tau}_{\bullet}$ at $\tau =0$, that is
$$V_{\alpha}^{\tau}G^{o}=\sum_{k_{i}\in\zit | \alpha_{i}-k_{i}\leq\alpha}\tau^{k_{i}}\omega_{i}^{o}$$
for all $\alpha\in\qit$. 
In particular, we have, and this is a key observation,
\begin{equation}\label{adaptedtoV}
A_{0}^{o}(\omega_{i}^{o})\in\sum_{\alpha_{j}\leq \alpha_{i}+1}\cit\omega_{j}^{o} 
\end{equation}
where $\alpha_{k}$ denotes the $V$-order of $\omega_{k}^{o}$, because $\partial_{\tau}V_{\alpha}^{\tau}G^{o}\subset V_{\alpha +1}^{\tau}G^{o}$.\\
 
Let $D=(d_{0},\cdots ,d_{\mu -1})$ be the diagonal matrix whose entries are the integral part of the eigenvalues of $A_{\infty}$ 
and
$$\omega^{Del}=(\omega_{1}^{Del},\cdots ,\omega_{\mu}^{Del})=(\omega_{1},\cdots ,\omega_{\mu})x^{D}$$
where, as above, $\omega =(\omega_{1},\cdots ,\omega_{\mu})$ is the basis of $G_{0}$ induced by $\omega^{o}$.
Then:

\begin{lemma} \label{lemma:CoRescDel} The matrix of the connection $\nabla$ in the basis 
$\omega^{Del}$ is
\begin{equation}\label{Del2}
(\frac{\sum_{i=0}^{n+1}x^{i}A_{i}^{Del}}{\theta}+A_{\infty})\frac{d\theta }{\theta }-
( \frac{\sum_{i=0}^{n+1}x^{i}A_{i}^{Del}}{\theta}+ A_{\infty}-D) \frac{dx}{x}    
\end{equation}
where the constant matrices $A_{i}^{Del}$ satisfy
$[D, A_{i}^{Del} ]=-(i-1)A_{i}^{Del}$
for $i=0\cdots n+1$.
\end{lemma}
\begin{proof} (1) The matrix of
$x\nabla_{\partial_{x}}$ in the basis $\omega^{Del}$ is
$$-\frac{x^{-D}(xA_{0}^{o})x^{D}}{\theta}-(A_{\infty}-D)$$
and we have $(x^{-D}(xA_{0}^{o})x^{D})_{ij}=x^{d_{j}-d_{i}+1}(A_{0}^{o})_{ij}$.
By condition $(\ref{adaptedtoV})$, $(A_{0}^{o})_{ij}\neq 0$ implies $\alpha_{i}\leq \alpha_{j}+1$ hence $d_{i}\leq d_{j}+1$. 
Since the $d_{i}$'s are contained in $[0,n]$ (because the $\alpha_{i}$'s are so) we get
$$x^{-D}(xA_{0}^{o})x^{D}=\sum_{i=0}^{n+1}x^{i}A_{i}^{Del}$$
with $[D, A_{i}^{Del}]=-(i-1)A_{i}^{Del}$.  
\end{proof}

Define $S:=\pi^{*}S^{o}$, where $S^{o}$ is defined in formula (\ref{eq:So}): $S$ is a non-degenerate, $\nabla$-flat, symmetric bilinear
form
\begin{equation}\label{eq:SRescaling}
S:G_{0}\times j^{*}G_{0}\rightarrow\cit[x,x^{-1},\theta ]\theta^{n}
\end{equation}
By definition we have $S(\omega_{i},\omega_{j})\in\cit x^{-n}\theta^{n}$ and thus
\begin{equation}\label{eq:DualiteRescaling}
 S(\omega_{i}^{Del},\omega_{j}^{Del}) =
    \begin{cases}
      x^{-1}S^{o}(\omega_{i}^{o},\omega_{j}^{o}) & \mbox{ if  } \alpha_{i}+\alpha_{j}=n \mbox{ and } \alpha_{i} \mbox{ is not an integer }\\
      S^{o}(\omega_{i}^{o},\omega_{j}^{o}) & \mbox{ if  } \alpha_{i}+\alpha_{j}=n \mbox{ and } \alpha_{i} \mbox{ is an integer }\\
      0 & \mbox{ otherwise }
     \end{cases}
\end{equation}

We are now ready to describe the expected quantum differential system.
Let $M=\cit^{*}$ and ${\cal G}$ be the trivial bundle on $\ppit^{1}\times M$ defined by the 
lattice $(G_{0},G_{\infty})$ where $G_{0}=\sum_{i}\cit [x,x^{-1},\theta ]\omega_{i}^{Del}$ and 
 $G_{\infty}=\sum_{i}\cit [x,x^{-1},\theta^{-1}]\omega_{i}^{Del}$.

\begin{proposition}
The tuple 
${\cal Q}=(M, {\cal G}, \nabla ,S, n)$
is a quantum differential system on $M$. 
\end{proposition}

\noindent In some cases, we get also a non-resonant logarithmic quantum differential system on $N=\cit$ (see section \ref{sec:SDQL}): indeed, let
${\cal G}^{log}$ be the trivial bundle on $\ppit^{1}\times N$ defined by the lattice $(G_{0}^{log},G_{\infty}^{log})$ 
where $G_{0}^{log}=\sum_{i}\cit [x,\theta ]\omega_{i}^{Del}$ and 
 $G_{\infty}^{log}=\sum_{i}\cit [x,\theta^{-1}]\omega_{i}^{Del}$. By equation (\ref{eq:DualiteRescaling}), we get

\begin{corollary} Assume that $\alpha_{j}$ is an integer for all $j$. The tuple 
${\cal Q}^{log}=(N, {\cal G}^{log}, \nabla ,S, n)$
is a non-resonant logarithmic quantum differential system on $N$, with pole at the origin of $N$.
\end{corollary}

\subsubsection{Classical limit}
\label{extensions}

We explain here how to construct the classical limit of ${\cal Q}$ at $x=0$ using the theory of the Malgrange-Kashiwara $V$-filtration.\\

\noindent {\bf {\em The $V$-filtration at $x=0$.}} For $k=0,\cdots ,\mu -1$, put $v(\omega^{Del}_{k})=d_{k}-\alpha_{k}\in ]-1,0]$.
Define, for $-1<\alpha\leq 0$,
$$V^{\alpha}G=\sum_{\alpha\leq v(\omega^{Del}_{k})}\cit [x][\theta ,\theta^{-1}]\omega^{Del}_{k}+
x\sum_{\alpha > v(\omega^{Del}_{k})}\cit [x][\theta ,\theta^{-1}]\omega^{Del}_{k},$$
$$V^{>\alpha}G=\sum_{\alpha < v(\omega^{Del}_{k})}\cit [x][\theta ,\theta^{-1}]\omega^{Del}_{k}+
x\sum_{\alpha \geq v(\omega^{Del}_{k})}\cit [x][\theta ,\theta^{-1}]\omega^{Del}_{k}$$
and $V^{\alpha +p}G=x^{p}V^{\alpha}G$ for $p\in\zit$ and $\alpha\in ]-1,0]$. This defines a decreasing filtration $V^{\bullet}$ of $G$ 
by $\cit [x][\theta ,\theta^{-1}]$-submodules. 
We will put $G_{\alpha}:=V^{\alpha}G/V^{>\alpha}G$ and $\psi_{x}G:=\oplus_{\alpha\in ]-1,0]}G_{\alpha}$.
Notice that 
$$V^{>-1}G=\sum_{k}\cit [x][\theta ,\theta^{-1}]\omega^{Del}_{k}$$
By lemma \ref{lemma:CoRescDel} and the definition of the matrix $D$, it follows that $V^{>-1}G$ ({\em resp.} $V^{>k}G$, $k\in\zit$) is 
Deligne's canonical extension of $G$ at $x=0$ such that the eigenvalues of the residue are contained in $]-1,0]$ ({\em resp.} $]k,k+1]$).

\begin{lemma}\label{lemma:VfiltrResc}
 The filtration $V^{\bullet}$ is the Kashiwara-Malgrange filtration at $x=0$.
\end{lemma}
\begin{proof}
It is directly checked that the filtration $V^{\bullet}$ satisfies all the characteristic properties of the Kashiwara-Malgrange filtration: the only point which is not completely obvious is the fact that $(x\nabla_{\partial_{x}}-\alpha )$ is nilpotent on  
$G_{\alpha}$, but this follows from lemma \ref{lemma:CoRescDel} because $\alpha_{i}\leq \alpha_{j}+1$ and $d_{i}=d_{j}+1$ imply $d_{i}-\alpha_{i}\geq d_{j}-\alpha_{j}$.
\end{proof}

\noindent Filtration $V^{\bullet}$ yields also an decreasing filtration $V^{\bullet}$ of $G_{0}$ by $\cit [x][\theta ]$-submodules by
$$V^{\alpha}G_{0}=V^{\alpha}G\cap G_{0}\ \mbox{and}\ V^{>\alpha}G_{0}=V^{>\alpha}G\cap G_{0}.$$
We will write $G_{0,\alpha}=V^{\alpha}G_{0}/V^{>\alpha}G_{0}$ and $\psi_{x}G_{0}:=\oplus_{\alpha\in ]-1,0]}G_{0,\alpha}$.
The metric $S$ in formula (\ref{eq:SRescaling})
induces, on each $V^{>k-1}G_{0}$, a bilinear form
$$S:V^{>k-1}G_{0}\times V^{>k-1}G_{0}\rightarrow x^{2k-1}\theta^{n}\cit [x,\theta ].$$
We have moreover
$$S:V^{>k-1}G_{0}\times V^{>k-1}G_{0}\rightarrow x^{2k}\theta^{n}\cit [x,\theta ]$$
if all the $\alpha_{i}$'s are integers.\\

\noindent {\bf {\em The limit.}} We now construct a (the) limit of system (\ref{Del2}) at $x=0$, using the $V$-filtration at $x=0$, and more precisely the 
nearby cycles. Let us define   
$$\psi_{x}G_{0}:=\oplus_{\alpha\in ]-1,0]}V^{\alpha}G_{0}/V^{>\alpha}G_{0}.$$ 
It is a free $\cit [\theta ]$-module, equipped with a connection $\psi_{x}\nabla$.
We will denote by $e^{\psi}$ the basis of $\psi_{x}G_{0}$ induced by $\omega^{Del}$.

\begin{lemma} \label{lemma:FibreLimit}
$x\nabla_{\partial_{x}}$ induces a map on $\psi_{x}G_{0}$
whose matrix, in the basis $e^{\psi}$, is
$$-\frac{\psi_{x}A_{0}}{\theta}-(A_{\infty}-D)$$
and
the matrix of $\psi_{x}\nabla_{\partial_{\theta}}$ takes the form, in the same basis,
\begin{equation}\label{eq:LimiteCoResc}
\frac{\psi_{x}A_{0}}{\theta^{2}}+\frac{A_{\infty}}{\theta}.
\end{equation}
We have moreover $[A_{\infty}, \psi_{x}A_{0}]=\psi_{x}A_{0}$.
\end{lemma}
\begin{proof} Follows from lemma \ref{lemma:VfiltrResc}: 
we have $(\psi_{x}A_{0})_{ij}=(A_{0}^{o})_{ij}$ if $\alpha_{i}= \alpha_{j}+1$, $(\psi_{x}A_{0})_{ij}=0$ otherwise: in other words
$[A_{\infty}, \psi_{x}A_{0}]=\psi_{x}A_{0}$. 
\end{proof}

\begin{remark} 
As already quoted, the connection $\nabla^{o}$ has in general an irregular singularity at $\theta =0$ (see formula (\ref{Bir})) while our limit,
that is system (\ref{eq:LimiteCoResc}) has a regular singularity at $\theta =0$: indeed,
after a base change of matrix $\theta^{-D}$, it becomes   
$$\frac{1}{\theta}(A_{\infty}-D+\psi_{x}A_{0}).$$
The construction of our limit thus yields a canonical ``regularization'' of system (\ref{Bir}).
\end{remark}

We define now the limit metric on the $\cit [\theta ]$-free module $\psi_{x}G_{0}$.
Let $\alpha \in ]-1,0[$. We have
$$S(V^{\alpha}G_{0}, V^{-1-\alpha}G_{0})\subset x^{-1}\cit [x, \theta ]\theta^{n}$$
and this yields (compose the previous one with the residue at $x=0$) a non-degenerate bilinear form 
$$\psi_{x}S_{\alpha}: gr_{V}^{\alpha}G_{0}\times j^{*}gr_{V}^{-1-\alpha}G_{0}\rightarrow \cit [\theta ]\theta^{n}$$
where $\psi_{x}S_{\alpha}(e^{\psi}_{i}, e^{\psi}_{j})=S^{o}(\omega_{i}^{o}, \omega_{j}^{o})$.
In the same way,
$$S(V^{0}G_{0}, V^{0}G_{0})\subset \cit [x, \theta ]\theta^{n}$$
induces
$$\psi_{x}S_{0}: gr_{V}^{0}G_{0}\times gr_{V}^{0}G_{0}\rightarrow \cit [\theta ]\theta^{n}$$
where $\psi_{x}S_{0}(e^{\psi}_{i}, e^{\psi}_{j})=S^{o}(\omega_{i}^{o}, \omega_{j}^{o})$. All this gives the expected limit metric 
$\psi_{x}S=\oplus_{\alpha\in ]0,-1]}\psi_{x}S_{\alpha}$ on $\psi_{x}G_{0}$.

\begin{lemma}\label{LimitSflat}
The form $\psi_{x}S$ is 
$\psi_{x}\nabla$-flat.
\end{lemma}
\begin{proof} It is enough to show that $\psi_{x}A_{0}$ is self-dual with respect to $\psi_{x}S$.
Because $(\psi_{x}A_{0})_{ij}=(A_{0}^{o})_{ij}$ if $\alpha_{i}=\alpha_{j}+1$ and $(\psi_{x}A_{0})_{ij}=0$ otherwise, this follows from the following two facts
: $A_{0}^{o}$ is self-adjoint with respect $S^{o}$ and $\alpha_{\mu -k}=\alpha_{i}+1$ if and only if $\alpha_{\mu -i}=\alpha_{k}+1$ (by formula (\ref{eq:RangeAlpha})).
\end{proof}

\noindent {\bf {\em R\'esum\'e (classical limit)}}:
By lemma \ref{lemma:FibreLimit}, the basis $e^{\psi}$ gives an extension ${\cal G}^{cl}$ of $\psi_{x}G_{0}$ as a
trivial bundle on $\ppit^{1}$, equipped with a meromorphic connection with poles of rank less or equal to 
$1$ at $\theta =0$ and with logarithmic pole at $\theta =\infty$ (see formula (\ref{eq:LimiteCoResc})). 

\begin{proposition}
The triple 
$${\cal Q}^{cl}=({\cal G}^{cl}, \psi_{x}\nabla , \psi_{x}S)$$
 is a quantum differential system. 
This is the classical limit of the quantum differential system ${\cal Q}$.
\end{proposition}

\noindent The quantum differential system ${\cal Q}^{cl}$ is {\em the} classical limit of the quantum differential system ${\cal Q}$.

\begin{remark}
One could also consider the free $\cit [\theta]$-module of rank $\mu$ 
$${\cal L}:=V^{>-1}G_{0}/V^{>0}G_{0}=V^{>-1}G_{0}/xV^{>-1}G_{0}$$
It is naturally equipped with a connection $\overline{\nabla}$ induced by $\nabla$ whose matrix in the basis $e$ induced by $\omega^{Del}$ is
$$(\frac{\overline{A}_{0}^{Del}}{\theta}+A_{\infty})\frac{d\theta }{\theta }$$
We have also $S(V^{>-1}G_{0}, V^{>-1}G_{0})\subset x^{-1}\cit [x, \theta ]\theta^{n}$ and we thus get (composing with the residue at $x=0$)
$$\overline{S}: {\cal L}\times {\cal L}\rightarrow \cit [\theta ]\theta^{n}$$
These data could also define a limit: the point is that $\overline{S}$ is not always $\overline{\nabla}$-flat. This is what happens for instance for 
$f(u_{1},u_{2})=u_{1}+u_{2}+\frac{1}{u_{1}^{2}u_{2}^{5}}$, in which case the matrix $\overline{A}_{0}^{Del}$ is not self-dual. Indeed, by example 
\ref{exemplebasiqueB}, we have in this situation $\mu =8$ and 
$$(\alpha_{0}, \alpha_{1}, \alpha_{2}, \alpha_{3}, \alpha_{4}, \alpha_{5}, \alpha_{6}, \alpha_{7})=(0,1,2,7/5, 4/5, 1, 6/5, 3/5).$$
The matrix $\overline{A}_{0}^{Del}$ is defined by 
\begin{equation}
\overline{A}_{0}^{Del}(e_{i}) =
    \begin{cases}
      8e_{i+1} & \mbox{ if  } i=0,1 \mbox{ and } 4\\
      0 & \mbox{ otherwise}
     \end{cases}
\end{equation}
while the matrix $\psi_{x}A_{0}$ is defined by 
\begin{equation}
\psi_{x}A_{0}(e_{i}) =
    \begin{cases}
      8e_{i+1} & \mbox{ if  } i=0 \mbox{ and } 1\\
      0 & \mbox{ otherwise}
     \end{cases}
\end{equation} 
On the other hand, we have
$$\overline{S}(\overline{A}_{0}^{Del}(e_{4}), e_{5})=8/w^{w}\ \mbox{and}\ 
\overline{S}(e_{4},\overline{A}_{0}^{Del}(e_{5}))=0.$$
In particular, $\overline{A}_{0}^{Del}$ is not self-dual.  
\end{remark}

\subsection{The small quantum cohomology of manifolds}\label{ex:PetiteCohQuant}
Let $X$ be a projective manifold with cohomology only in even degree. Let us assume that the rank of  $H^{2}(X,\zit )$ is equal to  $1$. 
Let ${\cal Q}$ be the logarithmic quantum differential system 
on $M=\cit$
associated with the small quantum cohomology of $X$ by example \ref{exemplebasiqueA}.

\begin{proposition} 
\begin{enumerate}
\item The quantum differentiel system ${\cal Q}$ is flat and logarithmic.
\item There exists a canonical fundamental solution $P(x,\tau )=H(x,\tau )e^{-\tau M_{1}(0)\ln x}$ of the Dubrovin connection of the quantum differential system ${\cal Q}$, uniquely characterized by the
initial condition $H(0,\tau )=I$. 
\end{enumerate}

\end{proposition}
\begin{proof}
1. Follows from the definitions. 2. By 1. and theorem \ref{cor:casvariete}, the solution $P$ is fundamental and symmetric. It remains to show ``conformality'',
and we keep the notations of section \ref{sec:NonResCurves}: the matrix $M_{1}(0)$ is by definition the matrix of $p\cup$ in a suitable basis of the cohomology algebra, while the matrix $N_{1}(0)$ represents the multiplication (with respect to the cup-product) by $E_{|x=0}$ in the same basis. Now, $E_{|x=0}=c_{1}(TX)$   
so that $N_{1}(0)=rM_{1}(0)$ for some $r\in\zit$ and we get the assertion using theorem \ref{cor:casvariete} 5. 
\end{proof}

\begin{remark}
Denote by $p^{i}:=p\cup\cdots\cup p$
the iteration ($i$-time) by the usual cup-product $\cup$.
We will say that $H^{*}(X)$ is $H^{2}$-generated if $p$ and its iterations $p^{i}$ are a basis of it.
In this case, the matrix  $M_{1}(0)$ is regular and the condition (GC) in theorem \ref{theo:reconstruction} is satisfied.
\end{remark}

\subsection{The small quantum orbifold cohomology of weighted projective spaces}
\label{ex:PetiteCohQuantOrbi1}
We now come back to the quantum differential system of example \ref{exemplebasiqueB}. We thus have a basis 
(see section \ref{ex:PetiteCohQuantOrbi} below for a precise definition of this basis) of the Brieskorn lattice in which the matrix 
of the Gauss-Manin connection takes the form
\begin{equation}\label{eq:CarEPP}
(M_0 (x)+M_1(x)\tau )\frac{dx}{x}-(A_{0}(x)\tau +A_{\infty})\frac{d\tau}{\tau}
\end{equation}
where 
$$M_{1}(x)=-\left ( \begin{array}{cccccc}
0  & 0  & 0 & 0  & 0 & x\\
1  & 0  & 0 & 0  & 0 & 0\\
0  & 1  & 0 & 0  & 0 & 0\\
0  & 0  & 1 & 0  & 0 & 0\\
0  & 0  & 0 & .. & 0 & 0\\
0  & 0  & 0 & 0  & 1 & 0
\end{array}
\right )$$
\noindent which is a $\mu\times\mu$ matrix with $\mu =1+w_{1}+\cdots +w_{n}$,
\begin{equation}
M_{0}(x)=diag (c_{0}, c_{2},\cdots ,c_{\mu -1}),\ A_{0}(x)=-\mu M_{1}(x),\ \mbox{and}\ A_{\infty}=diag (\alpha_{0}, \alpha_{1},\cdots , \alpha_{\mu -1}),
\end{equation}
the $c_{i}$'s being rational numbers contained in $[0,1[$. This quantum differential system is thus non-resonant at $\tau_0 =0$ but does not yield 
directly a canonical fundamental solution, for a conformality reason: indeed, $[A_{\infty}, A_{0}(0)]\neq A_{0}(0)$ in general.
Nevertheless, one can get a flat quantum differential system as follows:
let us put 
$$r= \frac{1}{\lcm (w_{0},\cdots ,w_{n})}$$
and $\zeta = x^{r}$; the characteristic equation (\ref{eq:CarEPP}) takes the form, 
in the basis $\widetilde{\omega}:=\omega x^{-R}$ of $G_{0}[x^{r}]$,
\begin{equation}\label{eq:carorb}
\tau \tilde{M}_{1}(\zeta)\frac{d\zeta}{\zeta}-(\tilde{A}_{0}({\zeta})\tau +A_{\infty})\frac{d\tau}{\tau}
\end{equation}
where 
$$\tilde{A}_{0}(\zeta)=\mu\left ( \begin{array}{cccccc}
0   & 0   & 0 & \cdots & 0   &  \frac{\zeta^{(1-c_{\mu -1})/r}}{w^{w}}\\
\zeta^{(c_{1}-c_{0})/r}   & 0   & 0 & \cdots & 0   & 0\\
0   &  \zeta^{(c_{2}-c_{1})/r} & 0 & \cdots & 0   & 0\\
..  & ... & . & \cdots & .   & .\\
..  & ... & . & \cdots & .   & .\\
0   & 0   & . & \cdots &  \zeta^{(c_{\mu -1}-c_{\mu -2})/r}  & 0
\end{array} \right )$$
and $\tilde{M}_{1}(\zeta)=-\frac{1}{r\mu}\tilde{A}_{0}(\zeta)$. The matrix $\tilde{A}_{0}(\zeta)$ has polynomial coefficients, see formula (\ref{eq:ck}) below.
The metric $S^{B}$ is defined, in the basis $\widetilde{\omega}$, by
\footnote{Keeping in mind orbifold Poincar\'e duality, we choose the
normalization
$S^{B}(\omega_{0},\omega_{n})=\frac{1}{w_{0}\cdots w_{n}}\theta^{n}$.}
\begin{equation}\label{eq:Sflat}
S^{B}(\widetilde{\omega}_{k},\widetilde{\omega}_{\ell})=\left\{ \begin{array}{ll}
\frac{1}{w_{0}\cdots w_{n}}\theta^{n} &  \mbox{if $0\leq k\leq n$ and $k+\ell =n$,}\\
\frac{1}{w^{w}}\frac{1}{w_{0}\cdots w_{n}}\theta^{n} & \mbox{if  $n+1\leq k\leq \mu -1$ and $k+\ell =\mu +n$,}\\
0 & \mbox{otherwise}
\end{array}
\right .
\end{equation}

\begin{proposition}\label{prop:sechorEPP}
The matrix
$P(\zeta ,\tau )=H(\zeta ,\tau )e^{-\tau \tilde{M}_{1}(0)\ln \zeta }$,
where $H(\zeta ,\tau )$ is the matrix associated with the characteristic equation (\ref{eq:carorb}) by lemma \ref{lemma:nonres},
is a canonical fundamental solution of the Dubrovin connection.
\end{proposition}
\begin{proof}
One uses assertions (1), (2) et (3) of theorem
\ref{cor:casvariete}. Conformality follows from formulas $[A_{\infty},\tilde{M}_{1}(0)]=\tilde{M}_{1}(0)$
and $\tilde{M}_{1}(0)=-\frac{1}{r\mu}\tilde{A}_{0}(0)$ (see theorem \ref{cor:casvariete} 5.) and the symmetry follows from proposition \ref{ex:solSymfondPCQO}.
\end{proof}

\section{Rational structures {\em via} quantum differential systems}

\label{sec:RationalStructure}
We apply here the previous results in order 
to construct first a distinguished rational structure on the $B$-side. Then, we derive from this one a rational structure on the $A$-side. This provides a
generalization of the method exposed in \cite[Proposition 3.1]{KKP}  
(see also \cite{T}).

\subsection{Preamble: sketch of the method}
\label{sec:Preamble}
Let ${\cal Q}=(M, {\cal G}, \nabla ,S, d)$ be a flat logarithmic quantum differential system on $M=\cit$. We will denote 
by
\begin{equation}\label{eq:CarRat}
\tau M_1(x) \frac{dx}{x}-(A_{\infty}+\tau A_0 (x))\frac{d\tau}{\tau}
\end{equation}
the matrix 
\footnote{The matrix $A_{\infty}$ is constant because the quantum differential system is flat.} of the connection $\nabla$ 
in the basis $\omega =(\omega_{0},\cdots ,\omega_{\mu -1})$. Let us summarize the previous results:
starting from the trivial bundle ${\cal G}$ on $\ppit^{1}\times M$, 
one constructs a trivial bundle ${\cal G}^{cl}$ on $\ppit^{1}$ (the limiting bundle, see section \ref{sec:limite}) 
whose fiber ${\cal G}^{const}:={\cal G}^{cl}_{\{\tau=-1\}}$ at $\tau =-1$ is a finite dimensional vector space.

\begin{proposition}\label{prop:classescar} Let us denote ${\cal G}^{\nabla}=\ker\nabla$ and ${\cal G}^{cl, \nabla^{cl}}=\ker\nabla^{cl}$.
Let us assume that the fundamental solution $P(\tau ,x )=H(\tau ,x)e^{-\tau M_{1}(0)\ln x}$ is canonical with $H(0,\tau)=I$.
Then we have isomorphisms
$$\begin{array}{ccccc}
{\cal G}^{\nabla}  & \longrightarrow  & {\cal G}^{cl, \nabla^{cl}} & \longrightarrow & {\cal G}^{const}\\
\Psi (\tau ,x ) & \mapsto  & \Psi^{cl}(\tau) & \mapsto  & \Psi^{const}
\end{array}$$
where 
\begin{itemize} 
\item $\Psi (\tau ,x)=H(\tau ,x)e^{-\tau M_{1}(0)\ln x}\Psi^{cl}(\tau)$,
\item $\Psi^{cl}(\tau)=(-\tau)^{A_{\infty}}(-\tau)^{A_{0}(0)}\Psi^{const}$,
\end{itemize}
$\Psi^{const}$ being a constant vector.
\end{proposition}
\begin{proof}
See corollary \ref{coro:sechor}.
\end{proof}

\noindent Using proposition \ref{prop:classescar}, one can thus shift on ${\cal G}^{const}$ 
the natural structures of ${\cal G}^{\nabla}$ (and vice versa): this is one of the interest of the quantization. 
Assume for instance that one has a distinguished rational structure $\Sigma^{quant}_{\qit}$ on ${\cal G}^{quant}$:
this structure shifts to a rational structure $\Sigma^{const}_{\qit}$ on  ${\cal G}^{const}$, but also on its mirror partners (if any).

\subsection{Rational structures {\em via} mirror symmetry and quantization: from the B-side to the A-side}

\label{sec:RationalStructuresviaMirror}

Let us start from the $B$-side and 
assume that the quantum differential system 
$${\cal Q}^{B}=(M^{B}, {\cal G}^{B}, \nabla^{B} ,S^{B}, n)$$
is produced, as in section \ref{sec:QDSregular}, by a function $F$ on $U\times M^{B}$ ($n=\dim_{\cit}U$) such that
\begin{itemize}
\item $F(\bullet ,x)$ is a tame regular function on $U$ for all $x\in M^{B}$,
\item the global Milnor number of the function $F(\bullet ,x)$ does not depend on $x\in M^{B}$ (we will denote by $\mu$ this constant value),
\item its Brieskorn lattice $G_{0}$ is free of rank $\mu$ over ${\cal O}_{M^{B}}(M^{B})[\theta]$
\end{itemize}
\noindent Typically, $M^{B}=\cit^{*}$, $F(\bullet ,x)$ is a convenient and nondegenerate Laurent polynomial for all $x\in M^{B}$ 
(with the same Newton polyhedron at infinity for all $x\in M^{B}$) and $G_{0}$ is a free
$\cit [x,x^{-1},\theta ]$-module.
We will consider this situation in section \ref{subsec:RatWPS}.

\subsubsection{Oscillating integrals}
We will denote by $\omega =(\omega_{0},\cdots \omega_{\mu -1})$ 
the basis of the Brieskorn lattice $G_{0}$, adapted to $S$ (see formula (\ref{eq:Sadaptee})), 
in which the connection $\nabla^{B}$ takes the form (\ref{eq:CarRat}).\\ 

On the $B$-side, the relation between the basis $\omega$ and the rational structure is
given by the oscillating integrals
$$I_{\Gamma}^{(i)}(\tau ,x)=\int_{\Gamma}e^{\tau F}\omega_i$$
where $\Gamma$ is a cycle with support on a ``family of supports'' $\Phi$ as in \cite[Section 1]{Ph1}. The integral
depends only on the homology class of $\Gamma$ in the nth homology group (with integral coefficients)
with support in $\Phi$. Let us be more precise about that: fix  $(\tau ,x)\in \cit^{*}\times M^{B}$. The homology group alluded to is $H_{n}^{\Phi_{\tau ,x}}(U, \zit )$ 
where $\Phi_{\tau ,x}$ 
is a family of supports
$A\subset U$ such that
$$Re (\tau F(\bullet , x))_{|A}\rightarrow -\infty$$
as $u\rightarrow +\infty$ or $0$. We have
$$H_{n}^{\Phi_{\tau ,x}}(U, \zit )=H_{n}(U, Re(\tau F)<C; \zit )$$
for $C<<0$ (see \cite[Formula (1.0), p. 13]{Ph}). Because $F(\bullet ,x)$ is tame, this is a free $\zit$-module of rank $\mu$, the global Milnor number of $F$.
If the critical points of $F(\bullet ,x)$, $x\in M$, are nondegenerate 
and the critical values are distinct, the cycles $\Gamma$ are called Lefschetz thimbles \cite[1.5]{Ph1}.
We will denote by $H_{n}^{\Phi_{\tau ,x}}(U, \qit )$ ({\em resp.} $H_{n}^{\Phi_{\tau ,x}}(U, \cit )$)
the $\qit$-({\em resp.} $\cit$)vector space generated by the linear combinations with rational coefficients of such cycles and
we will assume that these vector spaces are organized into a local
system $H^{\Phi}_{n}(\qit )$ ({\em resp.} $H_{n}^{\Phi}(\cit )$) of $\qit$-({\em resp.} $\cit$)vector spaces on $\cit^{*}\times M^{B}$: 
this follows from the assumptions above and this will be the case in our situation, see section \ref{subsubsec:Bside} below but also {\em f.i} \cite[1.5]{Ph1}, 
\cite[4.1]{Ph} and \cite[Proposition 3.12]{Ir}.

\subsubsection{Flat sections}
We will denote by $A^{\intercal}$ the transposed of a matrix $A$. 
If $\Omega (x,\tau)$ is the matrix of the connection $\nabla^{B}$ in the basis $\omega$, $\nabla^{B,*}$ 
will denote the connection with matrix $-\Omega^{\intercal}(x, -\tau )$ in 
the same basis\footnote{The twist by the minus sign is explained by the fact that we consider the kernel $e^{\tau f}$ instead of $e^{-\tau f}$.}.

\begin{lemma}\label{lemma:RealStructureLT}
The local system $H^{\Phi}_{n}(\cit )$ is identified with ${\cal G}^{\nabla^{B,*}}:=\ker \nabla^{B,*}$ via the map 
$\Psi^{*}$ defined by 
\begin{equation}\label{eq:vector}
\Psi^{*}(\Gamma )(\tau ,x )=\sum_{i=0}^{\mu -1}I_{\Gamma}^{(i)}(\tau ,x)\omega_{i}.
\end{equation}
\end{lemma}
\begin{proof} Write $I_{\Gamma}(\tau ,x)=(I_{\Gamma}^{(0)}(\tau ,x),\cdots , I_{\Gamma}^{(\mu -1)}(\tau ,x ))$.
We have 
\begin{equation} 
dI_{\Gamma}(\tau ,x)^{\intercal}=\Omega^{\intercal}(x, -\tau )I_{\Gamma}(\tau ,x )^{\intercal},
\end{equation}
where $\Omega (x, \tau)$ is the matrix of $\nabla^{B}$ in the basis $\omega$ (this is precisely what $\nabla^{B}$ is made for, see \cite[1\`ere partie, 6]{Ph}) 
and this shows that the map is well defined. The fact that it is an isomorphism follows from a dimension argument: the assumption on the Brieskorn lattice shows 
that ${\cal G}^{B}_{|\cit^{*}\times M^{B}}$ is a connection and its solutions are therefore organized  
in a local system on $\cit^{*}\times M$ whose fiber at $(\tau ,x)$ is, thanks to the tameness,  $H_{n}^{\Phi_{\tau ,x}}(U, \cit )$ (see \cite{Sab2}, \cite{Sab3}).   
\end{proof}

\noindent In some cases, this general construction will yield an identification between $H^{\Phi}_{n}(\cit)$ and ${\cal G}^{\nabla^{B}}$:  

\begin{corollary}\label{coro:LT} 
Assume that the basis $\omega$ is adapted to the bilinear form $S^{B}$ 
and let $\omega^{j}$ be the dual of $\omega_{j}$ with respect to $S^{B}$ (see remark \ref{rem:metglob}).
The map $\Psi :H^{\Phi}_{n}(\cit)\rightarrow {\cal G}^{\nabla^{B}}$ defined by
\begin{equation}\label{eq:LT}
\Psi (\Gamma )(\tau ,x )=(-\tau )^{n}\sum_{i=0}^{\mu -1}I_{\Gamma}^{(i)}(\tau ,x )\omega^{i}
\end{equation}
is an isomorphism. In particular, ${\cal O}_{\cit^{*}\times M^{B}}\otimes H^{\Phi}_{n}(\cit)\stackrel{\sim}{\rightarrow}{\cal G}_{|\cit^{*}\times M^{B}}$.
\end{corollary}
\begin{proof}
Because the matrices $A_{\infty}$ and $A_0 (x)$ in equality (\ref{eq:CarRat}) satisfy $A_{\infty}+A_{\infty}^{*}=nI$ and $A_0 (x)^{*}=A_0 (x)$ where $^{*}$ 
denotes the adjoint with respect to $S^{B}$, see remark \ref{rem:metglob}.

\end{proof}

\begin{notation} From now on, we will write $\Psi_{\Gamma}(\tau ,x )$ instead of $\Psi (\Gamma )(\tau ,x)$.
\end{notation}

\subsubsection{Rational structures (A-side/ B-side)}
\label{sec:RatAB}

As announced, we define:

\begin{definition}\label{def:RealStructureLT}
The image $\Sigma^{B, quant}_{\qit}$ of 
$H^{\Phi}_{n}(\qit )$
in ${\cal G}^{\nabla^{B}}:=\ker \nabla^{B}$ under the isomorphism $\Psi$ is called {\em the rational structure} on ${\cal G}^{\nabla^{B}}$. 
\end{definition}

\noindent In other words, $\Sigma^{B, quant}_{\qit}$ is the $\qit$-lattice generated by the solutions obtained after integration over cycles which are linear 
combinations, with rational coefficients, of the Lefschetz thimbles.

\begin{corollary}
\label{cor:RealStructureLT}
\begin{enumerate}
\item The rational structure $\Sigma^{B, quant}_{\qit}$ on ${\cal G}^{\nabla^{B}}$ provides a  rational structure $\Sigma^{B, const}_{\qit}$ on ${\cal G}^{B, const}$ via the isomorphisms of proposition \ref{prop:classescar}. 
Moreover, $\Sigma^{B, quant}_{\qit}$ is completely determined by $\Sigma^{B, const}_{\qit}$.
\item Let ${\cal Q}^{A}=(M^{A}, {\cal G}^{A}, \nabla^{A} ,S^{A}, n)$ be a quantum differential system isomorphic to the quantum differential
system ${\cal Q}^{B}$ in the sense of definition \ref{def:IsoSDQ}. Then the rational structure 
$\Sigma^{B, quant}_{\qit}$ on ${\cal G}^{B}$ defines a rational structure $\Sigma^{A, quant}_{\qit}$ on 
${\cal G}^{\nabla^{A}}:=\ker \nabla^{A}$ via this isomorphism.
\end{enumerate}
\end{corollary}
\begin{proof}
Follows from proposition \ref{prop:classescar}.
\end{proof}

\noindent The challenge is then
\begin{itemize}
\item to understand $\Sigma^{B, const}_{\qit}$,
\item to describe the rational structure $\Sigma^{A, const}_{\qit}$ on the cohomology of the mirror partner using a mirror theorem and to get a rational 
structure $\Sigma^{A, quant}_{\qit}$ on $\ker \nabla^{A}$.
\end{itemize}

\section{Application: rational structure for weighted projective spaces and their Landau-Ginzburg models.}

\label{subsec:RatWPS}

We apply in this section the previous recipe for the weighted projective spaces $\ppit (1,w_{1},\cdots w_{n})$ and their Landau-Ginzburg models. 
The main results of this section are theorem \ref{theo:descPsiEll} (description of the rational structure on the $B$-side)
and its corollary, theorem \ref{theo:Rat} (description of the rational structure on the $A$-side). 

\subsection{$B$-side}

\label{subsubsec:Bside}
In what follows, we will use the notations of example \ref{exemplebasiqueB}.

\subsubsection{The setting. Combinatorics.}

Recall the function $F$ defined by
$$F(u_{1},\cdots ,u_{n},x)=u_{1}+\cdots +u_{n}+\frac{x}{u_{1}^{w_{1}}\cdots u_{n}^{w_{n}}}$$
on $U\times \cit^{*}$ where $U=(\cit^*)^{n}$ and $w_{1},\cdots ,w_{n}$ are positive integers. For each $x\in\cit^*$, the function
$$(u_{1},\cdots ,u_{n})\mapsto F(u_{1},\cdots ,u_{n},x)$$
has $\mu$ non-degenerate critical points\footnote{Recall that $\mu =1+w_{1}+\cdots +w_{n}$.}
with distinct critical values and satisfies the assumptions of the beginning of section
\ref{sec:RationalStructuresviaMirror}.\\

We will need the following combinatorial tools:
let 
\begin{equation}\label{eq:F}
{\cal F}=\{\frac{i}{w_{j}}|\ 0\leq i\leq w_{j}-1, \ 0\leq j\leq n\}=\{f_{0},\cdots ,f_{k}\}
\end{equation}
(we put $w_{0}=1$), the numbers $f_{\ell}$ satisfying $0=f_{0}<\cdots <f_{k}<1$. 
Define
\begin{equation}\label{eq:Il}
I_{\ell}=\{j\in [0,n], \ w_{j}f_{\ell}\in\zit \}
\end{equation}
and let $d_{\ell}$ be its cardinal. We will write 
\begin{equation}\label{eq:pell}
p_{\ell}=d_{0}+\cdots +d_{\ell}
\end{equation}
and $p_{-1}=0$. Last,
let $c_{0},c_{1},\cdots , c_{\mu -1}$ be the sequence
$$\underbrace{f_{0},\cdots ,f_{0}}_{d_{0}},\underbrace{f_{1},\cdots ,f_{1}}_{d_{1}},\cdots ,\underbrace{f_{k},\cdots ,f_{k}}_{d_{k}}$$
arranged in increasing order. This sequence can be described as follows (see \cite[p. 3]{DoSa2}): define inductively the 
sequence $(a(k), i(k))\in\nit^{n+1}\times \{0,\cdots ,n\}$ by $a(0)=(0,\cdots ,0)$ , $i(0)=0$ and
\begin{equation}\label{eq:a(i)}
a(k+1)=a(k)+ \mbox{\bf{1}}_{i(k)} \mbox{ where } i(k):=\min \{i|
a(k)_{i}/w_{i}=\min_{j} a(k)_{j}/w_{j}\}.
\end{equation} 
Then we have 
\begin{equation}\label{eq:ck}
c_{k}=a(k)_{i(k)}/w_{i(k)}.
\end{equation}
Notice that 
\begin{itemize}
\item $a(1)=(1,0,\cdots ,0)$, 
\item $a(n+1)=(1,\cdots ,1)$, 
\item $a(\mu )=(1, w_{1},\cdots , w_{n})$,
\item $\sum_{i=0}^{n}a(k)_{i}=k$. 
\end{itemize}

\subsubsection{A flat quantum differential system and a canonical fundamental solution}
\label{ex:PetiteCohQuantOrbi}

In order to apply the results of section \ref{sec:RationalStructure}, we need first a flat quantum differential system.
It is provided by section \ref{ex:PetiteCohQuantOrbi1}. Let us recall the setting:
example \ref{exemplebasiqueB} provides a basis 
$\omega =(\omega_0 ,\cdots ,\omega_{\mu -1})$
of the Brieskorn lattice in which the matrix 
of the Gauss-Manin connection takes the form (see equation (\ref{eq:CarEPP}))
\begin{equation}
(R-\frac{1}{\mu}A_0(x)\tau )\frac{dx}{x}-(A_{0}(x)\tau +A_{\infty})\frac{d\tau}{\tau}
\end{equation}
This basis is precisely defined as follows: 
we have, for $k=1,\cdots ,\mu-1$, 
\begin{equation}\label{eq:defOmega}
\omega_{k}=\frac{x}{w_{1}^{a(k)_{1}}\cdots w_{n}^{a(k)_{n}}}u_{1}^{a(k)_{1}-w_{1}}\cdots u_{n}^{a(k)_{n}-w_{n}}\omega_{0}
\end{equation}
where $g\omega_{0}$ denotes the class of 
$g\frac{du_{1}}{u_{1}}\wedge\cdots \wedge \frac{du_{n}}{u_{n}}$ in the Gauss-Manin system $G$ of $F$, see \cite{DoMa}. Notice that in this situation $G$ is a
free $\cit [x,x^{-1},\tau ,\tau^{-1}]$-module equipped with a connection $\partial_{\tau}$ and thus 
$H^{\Phi}_{n}(\cit )$ is a local system on $\cit^{*}\times M$, $M=(\cit )^{*}$. 
The flat quantum differential system alluded to is the following:
put 
\begin{equation}\label{eq:r}
r= \frac{1}{\lcm (w_{0},\cdots ,w_{n})} 
\end{equation}
and $\zeta = x^{r}$; the characteristic equation (\ref{eq:CarEPP}) takes the form (see equation (\ref{eq:carorb}))
\begin{equation}
\tau \tilde{M}_{1}(\zeta)\frac{d\zeta}{\zeta}-(\tilde{A}_{0}({\zeta})\tau +A_{\infty})\frac{d\tau}{\tau}
\end{equation}
in the basis $\widetilde{\omega}:=\omega x^{-R}$ of $G_{0}[\zeta ]$,
where\footnote{$w^{w}=w_{1}^{w_{1}}\cdots w_{n}^{w_{n}}$} 
$$\tilde{A}_{0}(\zeta)=\mu\left ( \begin{array}{cccccc}
0   & 0   & 0 & \cdots & 0   &  \frac{\zeta^{(1-c_{\mu -1})/r}}{w^{w}}\\
\zeta^{(c_{1}-c_{0})/r}   & 0   & 0 & \cdots & 0   & 0\\
0   &  \zeta^{(c_{2}-c_{1})/r} & 0 & \cdots & 0   & 0\\
..  & ... & . & \cdots & .   & .\\
..  & ... & . & \cdots & .   & .\\
0   & 0   & . & \cdots &  \zeta^{(c_{\mu -1}-c_{\mu -2})/r}  & 0
\end{array} \right )$$
and $\tilde{M}_{1}(\zeta)=-\frac{1}{r\mu}\tilde{A}_{0}(\zeta)$. 
Notice that the 
matrix $\tilde{M}_{1}(0)$ has $k+1$ Jordan blocks $B_{0},\cdots ,B_{k}$, all associated with the eigenvalue $0$, of respective size 
$d_{\ell}$ for $\ell =0, \cdots ,k$.


By proposition \ref{prop:sechorEPP}, the matrix
$P(\zeta ,\tau )=H(\zeta ,\tau )e^{-\tau \tilde{M}_{1}(0)\ln \zeta }$,
where $H(\zeta ,\tau )$ is the matrix associated with the characteristic equation (\ref{eq:carorb}) by lemma \ref{lemma:nonres},
is a canonical fundamental solution of the Dubrovin connection.
We are now able to apply the results of section \ref{sec:RationalStructure}.

\subsubsection{Definition of the rational structures $\Sigma^{B, const}_{\qit}$ and $\Sigma^{B, quant}_{\qit}$} 
Let us define 
$$I_{\Gamma}^{(i)}(\tau ,x )=\int_{\Gamma}e^{\tau F}x^{-c_{i}}\tilde{\omega}_{i}$$
where $\tilde{\omega}_{i}$ denotes a representative of $\omega_{i}$ in $\Omega^{n}((\cit^{*})^{n})[\tau ,\tau^{-1}]$ and $\Gamma\in H^{\Phi}_{n}(\qit )$.\\

\noindent We will denote by $\overline{i}$ the index defined by 
\begin{equation}\label{eq:overlinei}
\overline{i}=\left\{ \begin{array}{ll}
n-i &  \mbox{if $0\leq i\leq n$,}\\
\mu +n-i & \mbox{if  $n+1\leq i\leq \mu -1$ and $k+\ell =\mu +n$}
\end{array}
\right .
\end{equation}
and
\begin{equation}\label{eq:tildeI}
\tilde{I}^{(i)}(\tau ,x)=\left\{ \begin{array}{ll}
w I^{(\overline{i})}(\tau ,x ) &  \mbox{if $0\leq i\leq n$,}\\
w^{w+1}I^{(\overline{i})}(\tau ,x ) & \mbox{if  $n+1\leq i\leq \mu -1$}
\end{array}
\right .
\end{equation}
in order to take into account corollary \ref{coro:LT}, in the light of formula (\ref{eq:Sflat}). We put $w:=w_{1}\cdots w_{n}$ and 
$w^{w+1}:=w_{1}^{w_{1}+1}\cdots w_{n}^{w_{n}+1}$.

\begin{lemma}\label{lemma:SecHorOrbiCoteB} 
\begin{enumerate}
\item The section
$$(\omega_{0}x^{-c_{0}},\cdots , \omega_{\mu -1}x^{-c_{\mu -1}})\Psi_{\Gamma}(x,\tau )^{\intercal},$$
where 
\begin{equation}\label{eq:VectorPsi0}
\Psi_{\Gamma}(\tau ,x)=(-\tau )^{n}
(\tilde{I}_{\Gamma}^{(0)}(\tau ,x ),\cdots ,\tilde{I}_{\Gamma}^{(\mu -1)}(\tau ,x ))
\end{equation}
is a flat section of $\nabla^{B}$. 
\item We have
\begin{equation}\label{eq:VectorPsi}
\Psi_{\Gamma}(\tau ,\zeta )^{\intercal}=H(\zeta ,\tau )\exp(-\tau \tilde{M}_{1}(0)\ln \zeta)\Psi_{\Gamma}^{cl}(\tau )^{\intercal}
\end{equation}
where $H(0,\tau )=I$ and
\begin{equation}\label{eq:VectorPsicl}
\Psi_{\Gamma}^{cl}(\tau )^{\intercal}=(-\tau )^{A_{\infty}}(-\tau )^{\tilde{A}_{0}(0)}(\Psi_{\Gamma}^{const})^{\intercal},
\end{equation}
$\Psi_{\Gamma}^{const}$ being a constant vector, the superscript $^{\intercal}$ denoting the transpose vector. 
\end{enumerate}
\end{lemma}
\begin{proof} 
The first assertion follows from corollary \ref{coro:LT} and formula (\ref{eq:Sflat}).
The second one is then a consequence of proposition \ref{prop:classescar}.
\end{proof}

\begin{corollary}\label{coro:Bconst}
The rational structure $\Sigma^{B,const}_{\qit}$ is the $\qit$-vector subspace of ${\cal G}^{B, const}$ generated by the vectors $\Psi_{\Gamma}^{const}$,
$\Gamma\in H^{\Phi}_{n}(\qit )$.
\end{corollary}
\begin{proof}
Use definition \ref{def:RealStructureLT} and corollary \ref{cor:RealStructureLT}.
\end{proof}

\noindent Recall that $\Sigma^{B,const}_{\qit}$ determines the rational structure $\Sigma^{B, quant}_{\qit}$ of definition \ref{def:RealStructureLT}.

\subsubsection{Description of $\Sigma^{B, const}_{\qit}$} 
In order to get an explicit description of $\Sigma^{B, const}_{\qit}$,
 we thus have to compute the vectors $\Psi_{\Gamma}^{const}$.
First, the identification of the Lefschetz thimbles can be done as in \cite[1.5, p.323]{Ph1}:
let
$$\Gamma_{0} =\rit^{*}_{+}\times\cdots \times \rit^{*}_{+}\subset (\cit^{*})^{n}.$$
This is a Lefschetz thimble \footnote{Strictly speaking, a section over $\rit^{*}_{-}\times\rit^{*}_{+}$ of such a cycle : notice that, for $x>0$, $\Gamma_{0}$ contains one critical point, namely 
$(\frac{x}{w^{w}})^{1/\mu}(w_{1},\cdots ,w_{n})$ and $F|_{\Gamma_{0}}$ is proper and takes values in $[\mu(\frac{x}{w^{w}})^{1/\mu}, +\infty[$.}
and other such cycles
are $\Gamma_{\ell}:=\epsilon^{\ell}\Gamma_{0}$ where, as above, $\epsilon$ is a $\mu$-th primitive root of $1$.
According to lemma \ref{lemma:SecHorOrbiCoteB}, 
the section $\Gamma_{\ell}$ determines a classical vector 
$$\Psi_{\Gamma_{\ell}}^{cl}:=(\Psi_{1,\ell }^{cl,0}, \cdots , \Psi_{n+1 ,\ell}^{cl,0},\Psi_{1,\ell }^{cl, 1},\cdots , \Psi_{d_{1} ,\ell}^{cl,1},
\cdots,\Psi_{1,\ell }^{cl,k},\cdots , \Psi_{d_{k} ,\ell}^{cl,k} )$$
and a constant vector
$$\Psi_{\Gamma_{\ell}}^{const}:=(\Psi_{1,\ell }^{0}, \cdots , \Psi_{n+1 ,\ell}^{0},\Psi_{1,\ell }^{1},\cdots , \Psi_{d_{1} ,\ell}^{1},
\cdots,\Psi_{1,\ell }^{k},\cdots , \Psi_{d_{k} ,\ell}^{k} ).$$
In both cases, the superscripts recall the Jordan blocks. The description of $\Sigma^{B, const}_{\qit}$ is then given by the following theorem: 
the first and the second part say that it is enough to compute the $\Psi_{i,0}^{j}$'s and this is done in the third part using a Mellin transform 
and a trick already used in \cite{KKP}.

\begin{theorem}\label{theo:descPsiEll}
 \begin{enumerate}
\item The rational structure $\Sigma^{B,const}_{\qit}$ is the $\qit$-vector subspace of ${\cal G}^{B, const}$ generated by the vectors 
\begin{equation}\label{eq:VecteursConstants}
\Psi_{\Gamma_{\ell}}^{const}:=(\Psi_{1,\ell }^{0}, \cdots , \Psi_{n+1 ,\ell}^{0},\Psi_{1,\ell }^{1},\cdots , \Psi_{d_{1} ,\ell}^{1},
\cdots,\Psi_{1,\ell }^{k},\cdots , \Psi_{d_{k} ,\ell}^{k} ),
\end{equation}
$\ell =0,\cdots ,\mu -1$.
\item We have
\begin{equation}\label{eq:PsiEll}
\Psi_{i,\ell }^{j}=e^{2i\pi \ell f_{j}}\sum_{m=0}^{i-1}\frac{(-2i\pi \ell)^{m}}{m!}\Psi_{i-m,0}^{j}
\end{equation}
for $j=0,\cdots, k$, $i=1,\cdots ,d_{j}$ and $\ell =0,\cdots,\mu -1$. 
\item We have, for $j=1,\cdots ,k$,
\begin{eqnarray}\label{eq:classecarorbifinal}\nonumber
b_{j}\Gamma (rs+f_{j})\prod_{m=1}^{n}\Gamma (w_{m}(rs-(1-f_{j}))+a(\mu +n-p_{j}+1)_{m})\\
=\sum_{m=0}^{d_{j}-1}(rs)^{-m-1}\Psi_{d_{j}-m,0}^{j}+O(1)
\end{eqnarray}
as $s\rightarrow 0$, where the sequence $(a(k))$ is defined by formula (\ref{eq:a(i)}), the number $r$ by formula (\ref{eq:r}) and
\begin{equation}\label{eq:bj}
 b_{j}=\frac{w_{1}^{w_{1}+1}\cdots w_{n}^{w_{n}+1}}{w_{1}^{a(\mu +n-p_{j}+1)_{1}}\cdots w_{n}^{a(\mu +n-p_{j}+1)_{n}}}
\end{equation}
For $j=0$, we have 
\begin{equation}\label{eq:classecarorbifinal0}
w_{1}\cdots w_{n} \prod_{m=0}^{n}\Gamma (w_{m}rs)
=\sum_{m=0}^{n}(rs)^{-m-1}\Psi_{n+1-m,0}^{0}+O(1)
\end{equation}
as $s\rightarrow 0$.
\end{enumerate}
\end{theorem}
\begin{proof} 1. This is corollary \ref{coro:Bconst}.\\
2. Let $\epsilon$ be a $\mu$-th primitive root of $1$. From the homogeneity condition
$\epsilon^{-\ell}F(u,\zeta )=F(\epsilon^{-\ell}u,\zeta )$,
we first get, using formula (\ref{eq:defOmega}) and the fact that $\sum_{i=0}^{n}a(s)_{i}=s$, 
$$\int_{\epsilon^{\ell}\Gamma_{0}(\epsilon^{-\ell}\tau ,\zeta )}e^{\epsilon^{-\ell}\tau F (u,\zeta )}\tilde{\omega}_{s}=
\epsilon^{\ell s}\int_{\Gamma_{0}(\tau ,\zeta)}e^{\tau F(u, \zeta)}\tilde{\omega}_{s}$$
for $s=0,\cdots ,\mu -1$.
Let $\imath =p_{j-1}+i-1$, for $j\in \{0,\cdots, k\}$ and $i\in\{1,\cdots ,d_{j}\}$. We thus have, using moreover 
equations (\ref{eq:VectorPsi0}) and (\ref{eq:VectorPsi}), 
$$ \sum_{m=0}^{i-1}\Psi_{i-m,\ell}^{cl,j}(\epsilon^{-\ell}\tau)\frac{\ln^{m} \zeta}{m!}(\frac{\epsilon^{-\ell}\tau}{r})^{m}=
\epsilon^{-\ell n}\epsilon^{\ell\overline{\imath}}\sum_{m=0}^{i-1}\Psi_{i-m,0}^{cl,j}(\tau)\frac{\ln^{m} \zeta}{m!}(\frac{\tau}{r})^{m},$$
see equation (\ref{eq:overlinei}) for the definition of $\overline{\imath}$.
It follows that
$$\Psi_{i,\ell}^{cl,j}(\epsilon^{-\ell}\tau)=\epsilon^{-\ell n}\epsilon^{\ell\overline{\imath}}\Psi_{i,0}^{cl,j}(\tau)$$
Now, the eigenvalue $\alpha_{i}$ of $A_{\infty}$ satisfies 
$$\alpha_{\imath}+\overline{\imath}-n=\alpha_{\imath}-\imath =-\mu f_{j}$$
if $0\leq i\leq n$ and 
$$\alpha_{\imath}+\overline{\imath}-n=\alpha_{\imath}-\imath +\mu  =\mu (1-f_{j})$$
otherwise.
We deduce, using equation (\ref{eq:VectorPsicl}) and the fact that $\tilde{A}_{0}(0)=-\mu \tilde{M}_{1}(0)$ (see section \ref{ex:PetiteCohQuantOrbi}),
$$ \sum_{m=0}^{i-1}\Psi_{i-m,\ell}^{j}\frac{(\mu\ln (-\tau ) +2i\pi\ell )^{m}}{m!}=
e^{2i\pi \ell f_{j}}\sum_{m=0}^{i-1}\Psi_{i-m,0}^{j}\frac{\mu^{m}\ln^{m} (-\tau )}{m!}$$
3. We have, for fixed indices $i$ and $v$, using again formula (\ref{eq:defOmega}),
\begin{equation}\nonumber
(-\tau) ^{n}\int_{0}^{+\infty}\int_{\Gamma_{0}}e^{\tau F}\zeta^{-f_{v}/r}\tilde{\omega}_{i}\zeta^{s}\frac{d\zeta}{\zeta}
\end{equation}
\begin{equation}\nonumber
=\frac{1}{w_{1}^{a(i)_{1}}\cdots w_{n}^{a(i)_{n}}}
(-\tau)^{-\mu (rs-f_{v})-i+n}
r\Gamma (rs+1-f_{v})\prod_{m=1}^{n}\Gamma (w_{m}(rs-f_{v})+a(i)_{m}).
\end{equation}
and we thus get, taking into account formulas (\ref{eq:tildeI}), (\ref{eq:VectorPsi0}), (\ref{eq:VectorPsi}) and using regularization,
$$b_{j}(-\tau)^{-\mu (rs+1-f_{j})+p_{j}-1}
\Gamma (rs+f_{j})\prod_{m=1}^{n}\Gamma (w_{m}(rs-(1-f_{j}))+a(\mu +n-p_{j}+1)_{m})$$
$$=\sum_{m=0}^{d_{j}-1}(rs)^{-m-1}(-\tau)^{m} \Psi_{d_{j }-m,0}^{j}(\tau )+O(1)$$
The expected equality follows putting $\tau =-1$. Similar computations for the case $j=0$. Notice that we have used here the fact that the good differential forms 
to consider are the $\zeta^{-f_{v}/r}\tilde{\omega}_{i}$'s (and not only the $\tilde{\omega}_{i}$'s).
\end{proof}

\noindent  Using an expansion in power series\footnote{Notice that, by the very definition,
$-w_{m}(1-f_{j})+a(\mu +n-p_{j}+1)_{m}\geq 0$.},
we see that the numbers $\Psi_{m,0}^{j}$, $j=1,\cdots ,k$, $m=1,\cdots ,d_{j}$, are determined by equation (\ref{eq:classecarorbifinal})
while the numbers $\Psi_{m,0}^{0}$, $j=0,\cdots ,k$, $m=1,\cdots ,n+1$, are determined by equation (\ref{eq:classecarorbifinal0}).\\

We have the following closed formula for the $\Psi_{1,\ell}^{j}$'s: let
\begin{equation}\label{eq:Cj}
C_j =\{m\in [1,n],\ w_m (1-f_j ) = a(\mu +n-p_{j}+1)_{m}\}
\end{equation}
for $j=0,\cdots ,k$.

\begin{corollary}\label{coro:psiNonNuls}
We have 
\begin{equation}
\Psi_{1,0}^{j}=b_{j}\prod_{m\in C_j}\frac{1}{w_{m}}\prod_{m=0}^{n}\Gamma (1-\{w_{m}f_{j}\})
\end{equation}
for $j=1,\cdots ,k$ where the $b_{j}$'s are defined by formula (\ref{eq:bj}) and $\{r\}=\lceil r\rceil -r$.
We have also $\Psi_{1,0}^{0}=1$ . In particular $\Psi_{1,\ell}^{j}\neq 0$ for $j =0,\cdots ,k$ and $\ell =0,\cdots ,\mu -1$.
\end{corollary}
\begin{proof} The first formula follows from formula (\ref{eq:classecarorbifinal}): because the cardinal 
of $C_j$ is precisely equal to $d_{j}$ for $j=1,\cdots ,k$, by the very definition of $d_{j}$ and formula (\ref{eq:ck}), we first deduce that
\begin{equation}
\Psi_{1,0}^{j}=b_{j}\Gamma (f_{j})\prod_{m\notin C_j}\Gamma (w_{m}f_{j}-w_{m}+a(\mu +n-p_{j}+1)_{m})
\prod_{m\in C_j}\frac{1}{w_{m}}
\end{equation}
Now, and by definition, we have 
\begin{itemize}
 \item $\lceil f_{j}w_{m} \rceil= w_{m}+1- a(\mu +n-p_{j}+1)_{m}$ if $m\notin C_{j}$,
 \item $f_{j}w_{m} = w_{m}- a(\mu +n-p_{j}+1)_{m}$ if $m\in C_{j}$.
\end{itemize}
The assertion follows.
For $j=0$, use formula  (\ref{eq:classecarorbifinal0}). 
Last, $\Psi_{1,\ell}^{j}\neq 0$ because $\Psi_{1,\ell}^{j}=e^{2i\pi\ell f_{j}}\Psi_{1,0}^{j}$.
\end{proof}

\begin{example} Assume that $w_{1}=\cdots =w_{n}=1$. Equation (\ref{eq:classecarorbifinal0}) is
$$\Gamma (s+1)^{n+1}=\sum_{m=0}^{n}s^{m}\Psi_{m+1, 0}^{0}+O(s^{n+1})$$
and we get
$$\Psi_{m+1, 0}^{0}=\frac{1}{m!}[\frac{d^{m}}{ds^{m}}\Gamma (s+1)^{n+1}]_{|s=0}.$$
In particular $\Psi_{1, 0}^{0}=1$ and $\Psi_{2, 0}^{0}=-(n+1)\gamma$ where $\gamma$ is the Euler 
constant.
\end{example}

\begin{example}\label{ex:P123P122}
(1)  Let $n=2$ and $(w_0 , w_1 , w_2 )=(1,2,3)$. Then $\mu =6$, $f_{0}=0$ and $d_{0}=3$, $f_{1}=\frac{1}{3}$ and $d_{1}=1$,
$f_{2}=\frac{1}{2}$ and $d_{2}=1$, $f_{3}=\frac{2}{3}$ and $d_{3}=1$. We have
$$a(0)=(0,0,0),\ a(1)=(1,0,0),\ a(2)=(1,1,0),\ a(3)=(1,1,1),\ a(4)=(1,1,2),\ a(5)=(1,2,2)$$ 
$$C_{1}=C_{3}=\{2\},\ C_{2}=\{1\}$$
$$b_{1}=18,\ b_{2}=36,\ b_{3}=108$$
and
$$\Psi_{1, 0}^{1}=6\ \Gamma (\frac{1}{3})\Gamma (\frac{2}{3}),\ 
\Psi_{1, 0}^{2}=18\ \Gamma (\frac{1}{2})^{2},\ \Psi_{1, 0}^{3}=36\ \Gamma (\frac{1}{3})\Gamma (\frac{2}{3}).$$
(2) Let $n=2$ and $(w_0 , w_1 , w_2 )=(1,2,2)$. Then $\mu =5$, $f_{0}=0$ and $d_{0}=3$, $f_{1}=\frac{1}{2}$ and $d_{1}=2$. We have
$$a(0)=(0,0,0),\ a(1)=(1,0,0),\ a(2)=(1,1,0),\ a(3)=(1,1,1),\ a(4)=(1,2,1),\ a(5)=(1,2,2)$$ 
$$C_{1}=\{1,2\}$$
and
$$\Psi_{1, 0}^{1}=4\ \Gamma (\frac{1}{2}).$$
\end{example}

\subsubsection{Conjugation}
\label{subsec:ConjugationB}
We now describe the conjugation on ${\cal G}^{B, const}$ defined by the 
rational structure $\Sigma^{B,const}_{\qit}$. We will denote by $\overline{\eta}$ the conjugate of $\eta$.
Recall the set ${\cal F}$ defined by formula (\ref{eq:F}). Notice first that $1-f_{j}\in {\cal F}$ if $f_{j}\in {\cal F}$, $j\neq 0$. 
For $j=1,\cdots ,k$, let $c(j)$ be the index such that $1-f_{j}=f_{c(j)}$. For $j=0$,
we define $c(0)=0$. We have $d_{c(j)}=d_j$ for $j=0,\cdots ,k$.

\begin{corollary}
We have, for $j=0,\cdots ,k$ and $m=0,\cdots , d_{j}-1$, 
\begin{equation}\label{eq:conjugB}
\sum_{i-1+m\leq d_j -1}\Psi^{j}_{i,0}\overline{\omega}_{p_{j-1}+i-1+m}=(-1)^{m}\sum_{i-1+m\leq d_j -1}\Psi^{c(j)}_{i,0}\overline{\omega}_{p_{c(j)-1}+i-1+m}
\end{equation}
In particular, the Jordan blocks $B_{j}$ and $B_{c(j)}$ are conjugate.
\end{corollary}
\begin{proof}Use the relations $\overline{\Psi}_{\Gamma_{\ell}}^{const}=\Psi_{\Gamma_{\ell}}^{const}$ for $\ell =0,\cdots ,\mu -1$ together 
with  theorem \ref{theo:descPsiEll}.
\end{proof}

\begin{example}\label{ex:P123P122Conj} (Example \ref{ex:P123P122} (1) continued)\\
  Let $n=2$ and $(w_0 , w_1 , w_2 )=(1,2,3)$. Recall that $\mu =6$, $f_{0}=0$ and $d_{0}=3$, $f_{1}=\frac{1}{3}$ and $d_{1}=1$,
$f_{2}=\frac{1}{2}$ and $d_{2}=1$, $f_{3}=\frac{2}{3}$ and $d_{3}=1$. 
We have
\begin{itemize}
\item $\Psi_{1, 0}^{1}\overline{\omega}_{3}=\Psi_{1, 0}^{3}\omega_{5}$,
\item $\overline{\omega}_{4}=\omega_{4}$,
\item $\overline{\omega}_{2}=\omega_{2}$,
\item $\overline{\omega}_{1}=-\omega_{1}-2\Psi^{0}_{2,0}\omega_{2}$,
\item $\overline{\omega}_{0}=\omega_{0}+2\Psi^{0}_{2,0}\omega_{1}+2(\Psi^{0}_{2,0})^2 \omega_{2}$
\end{itemize}

\end{example}

\subsection{$A$-side}

\label{section:Aside}

We give here a description of the rational structure $\Sigma^{A, const}_{\qit}$ on $H^{*}_{orb}(\ppit (w), \cit)$ 
defined by the rational structure $\Sigma^{B, const}_{\qit}$ and the mirror theorem \ref{theo:miroir} and, as a by-product, 
a description 
of the rational structure $\Sigma^{A, quant}_{\qit}$ on ${\cal G}^{\nabla^{A}}$ given by corollary \ref{cor:RealStructureLT}.\\

\subsubsection{The rational structure}

For any subset $I=\{i_{1},\cdots ,i_{r}\}\subset \{0,\cdots ,n\}$, we put 
$$\ppit (w_{I}):=\ppit (w_{i_{1}},\cdots , w_{i_{r}}).$$
 Recall that we have the decomposition (as vector spaces)
$$H^{*}_{orb}(\ppit (w),\cit )=\oplus_{j =0}^{k}H^{*}(|\ppit (w_{I_{j}})|,\cit )$$
where $I_{j}$ is defined by formula (\ref{eq:Il}). Each $H^{*}(|\ppit (w_{I_{j}})|,\cit )$
has a basis of the form
$$1_{f_{j}},1_{f_{j}}p,\cdots , 1_{f_{j}}p^{d_{j}-1}$$
where 
$p\in H^{2}(|\ppit (w)|, \cit)\subset H^{2}_{orb}(\ppit (w),\cit )$ is the Chern class of $O(1)$
and 
$1_{f_{j}}\in H^{0}(\ppit (|w_{I_{j}})|,\cit )\subset H^{*}_{orb}(\ppit (w),\cit )$.
As usual, we will denote by $1, p,\cdots ,p^{n}$ the corresponding basis of $H^{*}(|\ppit (w_{I_{0}})|,\cit )$.\\

According to the discussion in section \ref{sec:RatAB}, we define:

\begin{definition}
The rational structure $\Sigma^{A, const}_{\qit}$ on the orbifold cohomology of weighted projective spaces is the image of $\Sigma^{B, const}_{\qit}$ under
 the mirror isomorphism of theorem \ref{theo:miroir}. 
\end{definition}

We then  have the following explicit description of the rational structure $\Sigma^{A, const}_{\qit}$:
define the rational numbers
\begin{equation}\label{def:si}
s_{j}=\prod_{r=0}^{n}w_{r}^{-\lceil c_{j}w_{r}\rceil} 
\end{equation}
for $j=0,\cdots ,k$.

\begin{theorem}\label{theo:Rat}
The rational structure $\Sigma^{A, const}_{\qit}$ on the orbifold cohomology is the $\qit$-vector space generated by the vectors 
\begin{equation}\label{eq:Phi}
\Psi_{\Gamma_{\ell}}^{const}=\sum_{j=0}^{k}s_{j}\sum_{i=1}^{d_{j}}\Psi_{i,\ell}^{j} 1_{f_{j}}p^{i-1}
\end{equation}
$\ell =0,\cdots ,\mu -1$, where the numbers $\Psi_{i,\ell}^{j}$ are determined by equations (\ref{eq:PsiEll}), (\ref{eq:classecarorbifinal})
and (\ref{eq:classecarorbifinal0}) and the numbers $s_{i}$ are defined in (\ref{def:si}).
\end{theorem}
\begin{proof}
We use the following mirror correspondence, see \cite[Theorem 5.1.1 and Remark 5.1.3]{DoMa}:
under the mirror theorem \ref{theo:miroir}, the basis $(1_{f_{j}}p^{i})$ of orbifold cohomology of $\ppit (w)$ corresponds to the basis 
$$[\![\omega ]\!]=([\![\omega_{0}]\!],\cdots ,[\![\omega_{\mu -1}]\!])$$
of ${\cal G}^{B, const}$ 
induced by 
$(\omega_{0}x^{-c_{0}},\cdots , \omega_{\mu -1}x^{-c_{\mu -1}})$ as follows: 
the image of $1_{f_{j}}p^{i}$, $j=0,\cdots ,k$, $i=0,\cdots ,d_{j}-1$ under this correspondence is 
$s_{j}^{-1}[\![\omega_{p_{j-1}+i}]\!]$.
Now, the theorem follows from theorem \ref{theo:descPsiEll} (1).
\end{proof}

\subsubsection{A description via characteristic classes}

Inspired by \cite{KKP} and \cite{Ir}, we now rewrite theorem \ref{theo:Rat} with the help of some ´´characteristic´´ classes. Among other things we will see that 
the constants $s_{j}$ and $b_{j}$ (see equation (\ref{eq:classecarorbifinal})) in formula (\ref{eq:Phi}) miraculously disappear.
We use the notations of section \ref{subsubsec:Bside}.\\

Let us define, after formula (\ref{eq:classecarorbifinal}) and (\ref{eq:classecarorbifinal0}),

\begin{itemize}
\item for $j =1,\cdots ,k$, the cohomology classes
\begin{equation}
\widehat{\Gamma}_{j}=\prod_{m=0}^{n}\Gamma (rw_{m}1_{f_{j}}p+1-\{w_{m}f_{j}\})
\end{equation}
where $\{x\}=\lceil x\rceil -x$ and $r$ is defined by formula (\ref{eq:r}),

\item for $j=0$, the cohomology class
$$\widehat{\Gamma}_{0}:=\prod_{m=0}^{n}\Gamma (1+rw_{m}p).$$
\end{itemize}

\noindent These definitions have to be understood in the following way: in order to calculate $\Gamma (a1_{f_{j}}p+b)$ ($b>0$) we expand in power series the function 
$$s\mapsto \Gamma (as+b)$$
and we replace in this expansion $s^{k}$ by  
$1_{f_{j}}p^{k}$ keeping in mind that $1_{f_{j}}p^{d_{j}}=0$.

\begin{corollary} 
\label{coro:descriptionPsiconst}
The rational structure $\Sigma^{A, const}_{\qit}$ is the $\qit$-vector space generated in the orbifold cohomology by the vectors 
\begin{equation}\label{eq:descriptionPsiconst}
\Psi_{\Gamma_{\ell}}^{const}=\sum_{j=0}^{k}e^{2i\pi \ell f_j }\exp (-2i\pi\ell p)\cup\widehat{\Gamma}_{j}
\end{equation}
for $\ell =0,\cdots ,\mu -1$. Here $\cup$ denotes the cup-product on $H^{*}_{orb}(\ppit (w),\cit)$.
\end{corollary}
\begin{proof}
From theorem \ref{theo:descPsiEll} and theorem \ref{theo:Rat} it first follows that 
$$\Psi_{\Gamma_{\ell}}^{const}=\sum_{j=0}^{k}s_{j}e^{2i\pi \ell f_j }\exp (-2i\pi\ell p)\cup\Pi_{j}$$
where
\begin{eqnarray}\nonumber
\Pi_{j} &= & b_{j}\Gamma (r1_{f_{j}}p+f_{j})\prod_{m \in C_j}\frac{1}{w_{m}}\Gamma (rw_{m}1_{f_{j}}p+1)\\ \nonumber
&&\times\prod_{m \notin C_j}\Gamma (rw_{m}1_{f_{j}}p+w_{m}f_{j}-w_{m}+a(\mu +n-p_{j}+1)_{m}).
\end{eqnarray}
As already noticed, we have
$$\lceil f_{j}w_{m} \rceil= w_{m}+1- a(\mu +n-p_{j}+1)_{m}\ \mbox{if}\ m\notin C_{j},$$
$$f_{j}w_{m} = w_{m}- a(\mu +n-p_{j}+1)_{m}\ \mbox{if}\ m\in C_{j}.$$
It follows that $s_{j}b_{j}\prod_{m\in C_{j}}\frac{1}{w_{m}}=1$ and we get first 
\begin{eqnarray}\nonumber
\widehat{\Gamma}_{j} & = &
\Gamma (r1_{f_{j}}p+f_{j})\prod_{m \in C_j}\Gamma (rw_{m}1_{f_{j}}p+1)\\ \nonumber
&&\times\prod_{m \notin C_j}\Gamma (rw_{m}1_{f_{j}}p+w_{m}f_{j}-w_{m}+a(\mu +n-p_{j}+1)_{m})
\end{eqnarray}
The assertion follows using again the formula for $\lceil f_{j}w_{m} \rceil$ above.
\end{proof}

\noindent Up to the factor $r$, formula (\ref{eq:descriptionPsiconst}) agrees with \cite[Theorem 4.11]{Ir}.

\begin{remark}
In the case of $\ppit^n$, that is if $w_{0}=w_{1}=\cdots =w_{n}=1$, we have 
\begin{equation}\label{eq:descriptionSigmaConstPn}
\Sigma^{A, const}_{\qit}=\widehat{\Gamma}_{0}\cup\delta (H^{*}(\ppit^{n},\qit ))
\end{equation}
where $\delta (p^{m})=(2i\pi )^{m}p^{m}$ and $\widehat{\Gamma}_{0}=\Gamma (1+p)^{n+1}$.
Indeed, by corollary \ref{coro:descriptionPsiconst}, the vectors 
$$\Psi_{\Gamma_{\ell}}^{const}=\widehat{\Gamma}_{0}\cup\delta (1)-\ell \widehat{\Gamma}_{0}\cup\delta (p)+\cdots +
\frac{(-\ell )^n}{n!}\widehat{\Gamma}_{0}\cup\delta (p^n),$$
$\ell =0,\cdots ,n$ generate $\Sigma^{A, const}_{\qit}$ over $\qit$. It follows that the vectors $\widehat{\Gamma}_{0}\cup\delta (p^{i})$, $i=0,\cdots ,n$, 
also generate $\Sigma^{A, const}_{\qit}$ over $\qit$.
This result can already be found in \cite[Proposition 3.1]{KKP}.
\end{remark}

\subsubsection{Conjugation}
We now describe the conjugation on $H^{*}_{orb}(\ppit (w),\cit )$ defined by the 
rational structure $\Sigma^{A,const}_{\qit}$. We will denote by $\overline{\eta}$ the conjugate of $\eta\in H^{*}_{orb}(\ppit (w),\cit )$. 
From corollary \ref{coro:descriptionPsiconst} we get, keeping the notations of section \ref{subsec:ConjugationB},

\begin{corollary}
We have, for $j=0,\cdots ,k$ and $m=0,\cdots , d_{j}-1$, 
\begin{equation}\label{eq:conjug}
\overline{\widehat{\Gamma}_{j}\cup 1_{f_{j}}p^{m}}=(-1)^{m}\widehat{\Gamma}_{c(j)}\cup 1_{f_{c(j)}}p^{m} 
\end{equation}
In particular, the Jordan blocks $B_{j}$ and $B_{c(j)}$ are conjugate.
\end{corollary}




\section{Correlators of a logarithmic quantum differential system}

\label{sec:correlateurs}

The aim of this section is to consider, in the light of quantum differential systems, the following question: how to compute the gravitational correlators from Picard-Fuchs equations? The correlators alluded to are defined on the $A$-side for instance in \cite[Definition 10.1.1]{CK}. 
In general, we define the gravitational correlators (two points, genus $0$) of a flat logarithmic quantum differential system to be the coefficients 
of the matrix $H$ defined in lemma \ref{lemma:nonres} (the fact that we can define only two points correlators from a flat logarithmic quantum differential 
system is not so surprising because such systems correspond on the A-side to the small quantum cohomology).\\

Three remarks are in order:\\

$\bullet$ the correlators defined in this way by the quantum differential system of example \ref{exemplebasiqueA} ($A$-side) are precisely the ones of algebraic geometry,\\

$\bullet$ given a flat quantum differential system, the matrix of correlators $H$ is calculated solving the recursion equations (\ref{eq:DiffxHd}),\\

$\bullet$ the quantum differential system associated on the $B$-side with a regular tame function gives directly 
({\em i.e} without any reference to correlator) the matrix $H$ we are looking for.\\

\noindent In practise, a mirror theorem will thus give a way to compute the correlators of the mirror partner of a regular tame function. 
We apply the recipe in this section and we illustrate this by some simple examples.

\subsection{Gravitational two-points correlators of a flat logarithmic quantum differential system}

Let ${\cal Q}$ be a flat logarithmic quantum differential system on $M=\cit$ and
\begin{equation}
M(x,\tau )\frac{dx}{x}+N(x,\tau )\frac{d\tau}{\tau}
\end{equation}
be the matrix of the connection $\nabla$ in the basis $\omega =(\omega_{0},\cdots ,\omega_{\mu -1})$. 
By lemma \ref{lemma:nonres}, there exists a unique matrix $H(x,\tau )=I+\sum_{d\geq 1}H^{d}(\tau )x^{d}$ of holomorphic functions  
such that, after the base change of matrix $H(x,\tau )$, the matrix of the connection $\nabla$ takes the form
\begin{equation}
M(0,\tau )\frac{dx}{x}+N(0, \tau )\frac{d\tau}{\tau}
\end{equation}
Recall that the matrices  $H^{d}$ (and thus the matrix $H$) are defined by the equations (\ref{eq:DiffxHd}), that is
$$dH^{d}(\tau )=H^{d}(\tau )M(0,\tau )-M(0,\tau )H^{d}(\tau )-\sum_{i=1}^{d}M^{i}(\tau)H^{d-i}(\tau)$$
for $d\geq 1$ and $M(x,\tau)=M(0,\tau)+\sum_{i\geq 1}M^{i}(\tau)x^{i}$.\\
 
We will 
$H^{d}(\tau )=\sum_{r\geq 0}H^{d,r}\tau^{r}$
and $H^{d,r}=(H^{d,r}_{i,j})_{ij}$.

\begin{definition}\label{def:correl}
We will call the numbers
$$<\tau_{r}\omega_{a},\omega_{j}>_{0,2,d}:=g(\omega_{j},\omega_{\overline{j}})H^{d,r+1}_{\overline{j}+1,a+1},$$
$r\geq 0$, $d\geq 1$ and $a,j=0,\cdots ,\mu -1$
(the integers $\overline{j}$ are defined in remark \ref{rem:metglob}), {\em gravitational two-points correlators in genus 0} of the quantum differential system ${\cal Q}$. The matrix $H(x,\tau)$
is called the
{\em correlator matrix} of ${\cal Q}$.
\end{definition}

\noindent The previous definition can be extended to the case $d=0$: keeping in mind that $H^{(0)}(\tau )=I$, 
we define $<\tau_{r}\omega_{a},\omega_{j}>_{0,2,0}:=0$ for all $r$.\\

The link with the usual correlators $<\tau_{r}\phi_{a},\phi_{j}>_{0,2,d}$ of algebraic geometry as 
defined for instance in \cite[Definition 10.1.1]{CK} is given by the following lemma, which explains the terminology:

\begin{lemma} 
Let ${\cal Q}$ be the quantum differential system associated with the small quantum cohomology of a projective manifold $X$ by example \ref{exemplebasiqueA}. Then we have
\begin{equation}\label{eq:correlateurs}
<\tau_{r}\phi_{a},\phi_{j}>_{0,2,d}=g(\phi_{j},\phi_{\overline{j}})H^{d,r+1}_{\overline{j}+1,a+1}
\end{equation}
for all $a,j=0,\cdots ,\mu -1$, $r\geq 0$ and $d\geq 1$. 
The Gromov-Witten invariants $<\tau_{0}\phi_{i},\phi_{j}>_{0,2,d}$, $d\geq 1$, are described by the coefficients of $\tau$ in the matrix $H^{d}(\tau)$.
\end{lemma}
\begin{proof}
We have 
$$H(\phi_{a})=\phi_{a}+\sum_{d\geq 1}\sum_{j=0}^{\mu -1}\sum_{r\geq 0}H_{j+1,a+1}^{d,r+1}\phi_{j}\tau^{r+1}x^{d}
=\phi_{a}+\sum_{d\geq 1}\sum_{j=0}^{\mu -1}\sum_{r\geq 0}H_{\overline{j}+1,a+1}^{d,r+1}\phi_{\overline{j}}\tau^{r+1}x^{d}$$
$$
=\phi_{a}+\sum_{d\geq 1}\sum_{j=0}^{\mu -1}\sum_{r\geq 0}H_{\overline{j}+1,a+1}^{d,r+1}g(\phi_{j},\phi_{\overline{j}})\phi^{j}\tau^{r+1}x^{d}
$$
because $\phi_{\overline{j}}=g(\phi_{j},\phi_{\overline{j}})\phi^{j}$ by definition and because $H(0,\tau)=I$ and $H^{d,0}_{j+1,a+1}=0$ for $d\geq 1$
(this follows from (\ref{eq:DiffxHd}) because ${\cal Q}$ is flat).
Define now, as in \cite[section 10.2]{CK},
$$\tilde{s}(\phi_{a})=\phi_{a}+\sum_{d\geq 1}\sum_{j=0}^{\mu -1}\sum_{r\geq 0}\tau^{r+1}<\tau_{r}\phi_{a},\phi_{j}>_{0,2,d}\phi^{j}q^{d}.$$
Under the correspondence $x\leftrightarrow q$, we have $H=\tilde{s}$: this follows from the unicity, once given the initial condition $H(0,\tau )=I$, because, up to the factor $e^{-\tau M_{1}(0)\ln x}$, $H$ and $\tilde{s}$ yield fundamental solutions of the Dubrovin connection.
\end{proof}

\noindent Let us emphasize once again that these correlators can be computed using the recursion relations (\ref{eq:DiffxHd}).

\subsection{Examples}

We discuss here some very simple examples.

\subsubsection{Projective space}
\label{ex:cohquantique}

Let us consider the quantum differential system of example \ref{exemplebasiqueB}, with $w_1 =\cdots =w_n =1$, the mirror \footnote{Strictly speaking, one should take into account the coordinate $x_{0}$ of $H^{0}(X,\cit )$: the relative part of the quantum differential system is
$$-\tau I\frac{dx_{0}}{x_{0}}-\tau M_{1}(x_{1})\frac{dx_{1}}{x_{1}}$$
and one has to twist the following results by $x_{0}^{-\tau I}$.} of the small quantum cohomology of 
$X=\ppit^{n}$. We have $M(x,\tau )=\tau M_{1}(x)$ where
$$M_{1}(x)=-\left ( \begin{array}{cccccc}
0  & 0  & 0 & 0  & 0 & x\\
1  & 0  & 0 & 0  & 0 & 0\\
0  & 1  & 0 & 0  & 0 & 0\\
0  & 0  & 1 & 0  & 0 & 0\\
0  & 0  & 0 & .. & 0 & 0\\
0  & 0  & 0 & 0  & 1 & 0
\end{array}
\right )$$
\noindent which is a $(n+1)\times (n+1)$ matrix. According to example \ref{ex:PetiteCohQuant}, 
the correlator matrix is 
$$H(x,\tau )=I+\sum_{d\geq 1}H^{d}(\tau)x^{d}$$
where the matrices $H^{d}(\tau)$ are defined by the relations
\begin{equation}\label{eq:recuManifold}
dH^{d}(\tau )=\tau [H^{d}(\tau ) M_{1}(0)-M_{1}(0)H^{d}(\tau )+NH^{(d-1)}(\tau )]
\end{equation}
for all $d\geq 1$ where $N_{1,n+1}=1$ and $N_{i,j}=0$ otherwise. Using definition \ref{def:correl}, this gives some very well known results (see {\em f.i} \cite[Example 10.1.3.1]{CK} and the references therein): 

\begin{example}
Let us assume that $n=1$. We have, for $d\geq 1$,
\begin{itemize}
\item $<\tau_{2d-1}\omega_{0},\omega_{1}>_{0,2,d}=\frac{1}{(d!)^{2}}-\frac{2d}{(d!)^{2}}(1+\cdots +\frac{1}{d})$ and $<\tau_{r}\omega_{0},\omega_{1}>_{0,2,d}=0$ otherwise,
\item $<\tau_{2d-2}\omega_{1},\omega_{1}>_{0,2,d}=\frac{d}{(d!)^{2}}$ and $<\tau_{r}\omega_{1},\omega_{1}>_{0,2,d}=0$ otherwise,
\item $<\tau_{2d}\omega_{0},\omega_{0}>_{0,2,d}=\frac{-2}{(d!)^{2}}(1+\cdots +\frac{1}{d})$ and $<\tau_{r}\omega_{0},\omega_{0}>_{0,2,d}=0$ otherwise,
\item $<\tau_{2d-1}\omega_{1},\omega_{0}>_{0,2,d}=\frac{1}{(d!)^{2}}$ and $<\tau_{r}\omega_{1},\omega_{0}>_{0,2,d}=0$ otherwise.
\end{itemize}

\noindent Indeed, the recursion relation (\ref{eq:recuManifold}) gives
\begin{itemize}
\item $H_{11}^{d}(\tau)=\frac{\tau^{2d}}{(d!)^{2}}-\frac{2d\tau^{2d}}{(d!)^{2}}(1+\cdots +\frac{1}{d})$ for $d\geq1$, $1$ if $d=0$,
\item $H_{12}^{d}(\tau)=\frac{d\tau^{2d-1}}{(d!)^{2}}$ for $d\geq1$, $0$ if $d=0$,
\item $H_{21}^{d}(\tau)=\frac{-2}{(d!)^{2}}\tau^{2d+1}(1+\cdots +\frac{1}{d})$ for $d\geq1$, $0$ if $d=0$,
\item $H_{22}^{d}(\tau)=\frac{\tau^{2d}}{(d!)^{2}}$ for $d\geq1$, $1$ if $d=0$
\end{itemize}
\end{example}

\begin{remark}
Let us put $\deg \omega_{0}=0$ and $\deg\omega_{1}=2$. Then $<\tau_{r}\omega_{i},\omega_{j}>_{0,2,d}=0$ if
$$2r+\deg\omega_{i}+\deg\omega_{j}\neq 4d$$
as predicted by the ``degree axiom'' \cite[page 192]{CK}.
\end{remark}

\subsubsection{Weighted projective space}

\label{sec:cohomologieorbifolde}
Let us continue with example \ref{exemplebasiqueB}, but now for general integers $w_{0},\cdots ,w_{n}$, the mirror of the weighted projective spaces $\ppit (w_{0},\cdots ,w_{n})$. 
Recall the canonical fundamental solution 
$$P(\zeta ,\tau )=H(\zeta ,\tau )e^{-\tau \tilde{M}_{1}(0)\ln \zeta }$$
defined in example \ref{ex:PetiteCohQuantOrbi1}. We write
$$H(\zeta ,\tau )=I+\sum_{d\geq 1}H^{d}(\tau)\zeta^{d}$$
and $H^{d}(\tau)=\sum_{r\geq 0} H^{d,r}\tau^{r}$. We then get the (orbifold) correlators
$$<\tau_{r}\widetilde{\omega}_{a},\widetilde{\omega}_{j}>_{0,2,d}:= g(\widetilde{\omega}_{j}, \widetilde{\omega}_{\overline{j}})H_{\overline{j}+1, a+1}^{d,r+1}$$
where the sections $\widetilde{\omega}_{i}$ are defined in example \ref{ex:PetiteCohQuantOrbi1}. Once again, these correlators can be computed in practise using formula (\ref{eq:DiffxHd}).

\begin{example}\label{petitexemple}
Let 
$n=1$, $w_{0}=1$ and $w_{1}=2$. We have, with the notations of section \ref{ex:PetiteCohQuantOrbi}, $\mu =3$, $c_{0}=c_{1}=0$, $c_{2}=\frac{1}{2}$ and
$$\tilde{M}_{1}(\zeta )=-2\left ( \begin{array}{ccc}
0  & 0 & \zeta /4\\
1  & 0 & 0\\
0  & \zeta & 0
\end{array}
\right )$$
We have $\overline{0}=1$ and $g(\widetilde{\omega}_{0},\widetilde{\omega}_{1})=\frac{1}{2}$, $\overline{2}=2$ and $g(\widetilde{\omega}_{2},\widetilde{\omega}_{2})=\frac{1}{8}$.
The matrix $H(\zeta ,\tau )$ is determined by the relations
$$dH^{d}(\tau )=
2\tau [
\left ( \begin{array}{ccccc}
0  & 0  & 0\\
1  & 0  & 0\\
0  & 0  & 0
\end{array}
\right )
H^{d}(\tau)
-
H^{d}(\tau)
\left ( \begin{array}{ccccc}
0  & 0  & 0 \\
1  & 0  & 0 \\
0  & 0  & 0
\end{array}
\right )]$$
$$+2\tau
\left ( \begin{array}{ccccc}
0  & 0  & 1/4\\
0  & 0  & 0\\
0  & 1  & 0
\end{array}
\right )
H^{(d-1)}(\tau)$$
for $d\geq 1$ and $H^{(0)}(\tau )=I$. We have for instance
$$H^{(1)}(\tau )=
\left ( \begin{array}{ccccc}
0  & 0  & \frac{\tau}{2}\\
0  & 0  & \tau^{2}\\
-4\tau^{2}  & 2\tau  & 0
\end{array}
\right )\ \mbox{and}\
H^{(2)}(\tau )=
\left ( \begin{array}{ccccc}
-\frac{3}{2}\tau^{3}  & \frac{\tau^{2}}{2}  & 0\\
-2\tau^{4}  & \frac{\tau^{3}}{2}  & 0\\
0  & 0  & \tau^{3}
\end{array}
\right).$$
\noindent This gives\\

$\bullet$ $<\tau_{2}\widetilde{\omega}_{0},\widetilde{\omega}_{1}>_{0,2,2}=-3/4$,
$<\tau_{1}\widetilde{\omega}_{1},\widetilde{\omega}_{1}>_{0,2,2}=1/4,$
$<\tau_{3}\widetilde{\omega}_{0},\widetilde{\omega}_{0}>_{0,2,2}=-1,$
$<\tau_{2}\widetilde{\omega}_{1},\widetilde{\omega}_{0}>_{0,2,2}=1/4$,
$<\tau_{2}\widetilde{\omega}_{2},\widetilde{\omega}_{2}>_{0,2,2}=1/8$
and $<\tau_{r}\widetilde{\omega}_{i},\widetilde{\omega}_{j}>_{0,2}=0$ otherwise,\\

$\bullet$
$<\tau_{0}\widetilde{\omega}_{2},\widetilde{\omega}_{1}>_{0,2,1}=1/4$,
$<\tau_{1}\widetilde{\omega}_{0},\widetilde{\omega}_{2}>_{0,2,1}=-1/2$,
$<\tau_{0}\widetilde{\omega}_{1},\widetilde{\omega}_{2}>_{0,2,1}=1/4$,
$<\tau_{1}\widetilde{\omega}_{2},\widetilde{\omega}_{0}>_{0,2,1}=1/2$
and $<\tau_{r}\widetilde{\omega}_{i},\widetilde{\omega}_{j}>_{0,2,1}=0$ otherwise.
\end{example}

\begin{remark} In the previous example we have $<\tau_{r}\widetilde{\omega}_{i},\widetilde{\omega}_{j}>_{0,2,d}=0$ if
$$2r+\deg\widetilde{\omega}_{i}+\deg\widetilde{\omega}_{j}\neq 3d$$
where $\deg\widetilde{\omega}_{0}=0$, $\deg\widetilde{\omega}_{1}=2$ and $\deg\widetilde{\omega}_{2}=1$. More generally set 
now ${\cal X}=\ppit (w_{0},\cdots ,w_{n})$. If 
$<\tau_{k_{0}}\phi_{0},\cdots , \tau_{k_{\ell -1}}\phi_{\ell -1}>_{0,\ell ,d}$ is not equal to zero then
\begin{equation}\label{eq:degre}
\sum_{i=0}^{\ell -1}(\deg^{orb}\phi_{i}+2k_{i})=2n+2<c_{1}(T{\cal X}),d>+2\ell -6.
\end{equation}
We have $c_{1}(T{\cal X})=\mu p$ where 
$p:=\frac{1}{pgcd(w_{0},\cdots ,w_{n})}c_{1}({\cal O}_{{\cal X}}(pgcd(w_{0},\cdots ,w_{n})))\in H^{2}(|{\cal X}|,\cit )$.   
If moreover the numbers $\mu :=w_{0}+\cdots +w_{n}$ and $\frac{1}{r}:=lcm (w_{0},\cdots ,w_{n})$ are prime there exists 
a generator $D_{w}$ of $H_{2}(|{\cal X}|,\zit )$ such that $\int_{D_{w}}p=r$. In these conditions, equation (\ref{eq:degre}) becomes
\begin{equation}\label{eq:degrebis}
\sum_{i=0}^{\ell -1}(\alpha_{i}+k_{i})=n+\mu rd+\ell -3
\end{equation}
if $R_{\infty}(\tilde{\omega}_{i})=\alpha_{i}\tilde{\omega}_{i}$.
This justifies the twist by $r$ above.
\end{remark}

\section{A mirror partner of the Hirzebruch surface $\fit_{2}$ {\em via } quantum differential systems and its classical limit}

\label{sec:hirzebruch}

In this section we compute a mirror partner of the small quantum cohomology of the Hirzebruch surface $\fit_{2}$ using quantum differential systems, 
as explained in section \ref{subsec:Mir}. 
This provides a concrete mirror theorem for a non Fano variety {\em via} Gauss-Manin systems and Brieskorn lattices. 
The explicit construction of the quantum differential system ${\cal Q}^B$ associated with the Landau-Ginzburg model of $\fit_2$ is interesting for several reasons. First, it brings to light some new phenomena on the $B$-side, in comparison with the Fano situations considered until now: for instance, and for tameness reasons, the base space (which is two dimensional in this situation because the Picard group of $\fit_2$ is so)
$M^B$ is not the whole torus $(\cit^{*})^{2}$. Second, 
the description of the mirror map $\nu$ (see definition \ref{def:IsoSDQ})  in terms of flat coordinates
is very transparent in this setting. Flatness has to be understood with respect 
to a residual connection which is explicitely produced by the quantum differential system ${\cal Q}^B$. Notice that several normalizations of such flat coordinates are possible, and this is essentially due to the fact that the rank
of the Picard group is greater than one: keeping in mind mirror symmetry, this ambiguity is at the end set by the metric.
Last, and independently of the mirror theorem, we construct a logarithmic Frobenius manifold starting from a restricted set of data, using the reconstruction results of \cite{R} and \cite{HeMa}.

We describe briefly the setting in section \ref{sec:Setting} and we calculate, on the $B$-side, the mirror flat quantum differential 
system in section \ref{sec:MirSDQF2}.  
The mirror theorem is stated in section \ref{sec:MirF2} and we check 
in section \ref{sec:ResP112} that specialization of the previous results at suitable values of the parameters $q_1$ et $q_2$ gives 
the small quantum orbifold cohomology of $\ppit (1,1,2)$, an aspect of Ruan's conjecture 
(this has been done first in \cite{CIT}, in a slightly different setting). 
On the way, 
we use our computations in order to construct a logarithmic Frobenius manifold in section \ref{sec:LogFrob}.

\subsection{A Landau-Ginzburg model for the Hirzebruch surface $\fit_{2}$}
We will denote by $X$ the Hirzebruch surface $\fit_{2}$: it is a compact and smooth toric variety such that
$Pic (X)=\zit f\oplus \zit H$ where $f$ is the class of a fiber and $H$ is  
the zero section (the section at infinity is $H-2f$). The cohomology algebra of $X$ is $\cit [f, H]/ <f^{2}, H^{2}-2Hf>$ and we have intersection numbers
\begin{equation}\label{eq:NombreIntersection}
f^2=0,\  H^2=2\ \mbox{and}\ (H-2f)^2 =-2
\end{equation}

\noindent We consider the fan $\Sigma$ of $X$ with one dimensional cones 
$$\Sigma (1)=\{(1,0), (0,1), (-1,2), (0,-1)\}$$
where we identify elements of $\Sigma (1)$
with their primitive generators $v_{1}=(1,0)$, $v_{2}=(0,1)$, $v_{3}=(-1,2)$ and $v_{4}=(0,-1)$. We denote by
$D_{1}$, $D_{2}$, $D_{3}$ and $D_{4}$ the corresponding divisors,
with intersection numbers 
$$D_{1}^2=0,\  D_{2}^2=-2,\ D_{3}^2 =0,\ D_{4}^2=2$$  
\noindent We have

\begin{equation}\label{eq:Divisors}
D_{1}=f,\ D_{2}=H-2f,\ D_{3}=f, \ D_{4}=H
\end{equation}

\noindent and thus $D_{1}+D_{2}+D_{3}+D_{4}=2H$. In particular $X$ is not Fano.

Recall also the exact sequence 

\begin{equation}\label{eq:SuiteExacteH2}
0\longrightarrow Pic(X)\stackrel{\psi}{\longrightarrow} \zit^{4}\stackrel{\varphi}{\longrightarrow}\zit^{2}\longrightarrow 0 
\end{equation}

\noindent where $\varphi (e_{i})=v_{i}$ for $i=1,\cdots ,4$, 
$(e_{i})$ denoting the canonical basis of $\zit^{4}$, and 
$$\psi (af+bH)=\sum_{i=1}^{4}<D_{i},af+bH>e_{i}.$$
The matrix of $\varphi$ is 
$$\left ( \begin{array}{cccc}
1 & 0  & -1 & 0  \\
0  & 1  & 2 & -1  
\end{array}
\right )$$

\noindent while the matrix of $\psi$ in the basis $(f,H)$ of $Pic(X)$ is

$$\left ( \begin{array}{cc}
 1  & 0  \\
-2 & 1    \\
1   & 0    \\
0   & 1    
\end{array}
\right )$$
\noindent (the two column vectors of the latter matrix generate the linear relations between the $v_{i}$'s).\\

 Applying the functor $Hom_{\zit}(--, \cit^{*})$ to the exact sequence (\ref{eq:SuiteExacteH2}), we get the Landau-Ginzburg model for $X$: 
it is the function $F$ defined by
$$u_{1}+u_{2}+u_{3}+u_{4}$$
restricted to 
$$U=\{ (u_{1},u_{2},u_{3},u_{4})\in (\cit^{*})^{4}|\ u_{1}u_2^{-2}u_{3}=q_{1}\ \mbox{and}\ u_{2}u_4 =q_{2}\}$$
Throughout this paper, we will consider the following presentation of $F$:

\begin{definition}
The Landau-Ginzburg model of the Hirzebruch surface $\fit_2$ is the function $F$ 
defined  by 
\begin{equation}\label{eq:LGModel}
F(u_1 ,u_2, q_1 ,q_2)=u_1+u_2+q_{1}\frac{u_{2}^{2}}{u_1}+q_{2}\frac{1}{u_2}
\end{equation}
on $(\cit^{*})^2\times\cit^2$.
\end{definition}

\subsection{Tameness properties of the Landau-Ginzburg model and its Brieskorn lattice}
\label{sec:Setting}

\subsubsection{Tameness}

The function $F$ has some tameness properties, depending on the position of the parameters $(q_1 ,q_2)$. In order to see this, let $\Gamma$ be the 
convex hull of $(1,0)$, $(0,1)$, $(-1,2)$, $(0,-1)$ in $\rit^{2}$.

\begin{lemma}\label{lemma:NDC}
The Laurent polynomial function $f:(u_1,u_2)\mapsto F(u_1,u_2, q_1 ,q_2)$ is convenient and non-degenerate with respect to $\Gamma$
in the sense of \cite{K} for all
$$(q_{1},q_{2})\in M:=\{(q_{1},q_{2})\in (\cit^{*})^{2}| q_{1}\neq \frac{1}{4}\}.$$ 
For $(q_1 ,q_2 )\in M$, the function $f$ has four non-degenerate critical points (and four distinct critical values)
and its global Milnor number is equal to $4$.
\end{lemma}
\begin{proof}
The function $f$ is convenient for $q_{1}q_{2}\neq 0$ because $0$ belongs to the interior of $\Gamma$. 
 Let us denote
\begin{itemize}
\item $\Gamma_{0}$ the face of $\Gamma$ whose equation is  $x-y=1$,
\item $\Gamma_{1}$ the face whose equation is $-3x-y=1$,
\item $\Gamma_{2}$ the face whose equation is $x+y=1$
\end{itemize}
We define 
\begin{equation}\nonumber
f_{|\Gamma _{0}}=u_1+q_{2}\frac{1}{u_2},\  f_{|\Gamma _{1}}=q_{1}\frac{u_{2}^{2}}{u_1}+q_{2}\frac{1}{u_2},\ 
f_{|\Gamma _{2}}=u_1+u_2+q_{1}\frac{u_{2}^{2}}{u_1},
\end{equation}
the restrictions of $f$ to the boundary of $\Gamma$. It is easily seen that 
\begin{equation}\nonumber
 u_{1}\frac{\partial f_{|\Gamma _{j}}}{\partial u_{1}}= u_{2}\frac{\partial f_{|\Gamma _{j}}}{\partial u_{2}}=0 \Longrightarrow u_{1}u_{2}=0
\end{equation}
(this condition means precisely that $f$ is non-degenerate) if and only if moreover $q_{1}\neq \frac{1}{4}$.  The assertion about the global Milnor number
then follows for instance from \cite{K}, but it can also be directly checked.
\end{proof}


\noindent Notice that the restriction of $F$ at $q_{1}=\frac{1}{4}$ has two non-degenerate critical points and it follows that its (global) Milnor number is equal to two. This ``jump'' of Milnor numbers is not so surprising:   
the restriction of $f$ at $q_{1}=\frac{1}{4}$ is indeed degenerate.
 Notice also that there are no critical points in $M$ disappearing at infinity in the sense of 
\cite[section 2]{DoSa1}.

\begin{remark}\label{rem:discriminant}
The discriminant $D$ 
(image of the singular locus)
of $(F(u, q_1 , q_2 ), q_1 , q_2 )=(t,q_1 , q_2 )$ has the equation
$$t^4 -8q_2 t^2 -16 q_2^2 (4q_1 -1)=0$$
and is therefore smooth on $M$. The projection $\pi :D\rightarrow M$ is finite and the cardinal of any fiber is equal to four.
\end{remark}

\subsubsection{The Brieskorn lattice}

We refer to \cite{DoSa1} for the definition of the (Fourier-Laplace tranform of the) Brieskorn lattice $G_0$ and the Gauss-Manin system $G$ of $F$.
Recall the following facts: put 
$$\cit [M,\theta ]:=\cit [q_{1}, q_{2}, q_{1}^{-1}, q_{2}^{-1}, (4q_1 -1)^{-1},\theta ]$$ 
where $\theta$ is a new variable and $U=(\cit^{*})^2$. Then
\begin{itemize}
\item $G_{0}=\Omega^{n}(U)[M,\theta ]/ 
(\theta d-dF\wedge)\Omega^{n-1}(U)[M,\theta ]$,
\item $G_{0}/\theta G_{0}=\Omega^{n}(U)[M,\theta ]/ 
dF\wedge\Omega^{n-1}(U)[M,\theta ]$.
\end{itemize}
\noindent where $d$ is a relative differential: the derivation is taken with respect to $(u_1 ,u_2)$ only. $G_{0}$ is naturally a $\cit [M,\theta ]$-module and is equipped with a action of $\theta^{2} \nabla_{\partial_{\theta}}$ which is induced by the multiplication by $F$. We will also write $\tau$ for $\theta^{-1}$; in particular $\nabla_{\partial_{\tau}}=-\theta^{2} \nabla_{\partial_{\theta}}$.\\

\begin{proposition}
\label{prop:basisG0}
The classes 
\begin{equation}\label{eq:Triangle}
\triangle :=(\vartriangle_0 ,\vartriangle_1 ,\vartriangle_2 ,\vartriangle_3 ):=
([\frac{du_1}{u_1}\wedge\frac{du_2}{u_2}], [\frac{1}{u_2}\frac{du_1}{u_1}\wedge\frac{du_2}{u_2}], 
[\frac{u_2}{u_1}\frac{du_1}{u_1}\wedge\frac{du_2}{u_2}], [\frac{u_2^2}{u_1}\frac{du_1}{u_1}\wedge\frac{du_2}{u_2}])
\end{equation}
yield a basis of $G_{0}$ over $\cit [M,\theta ]$. 
\end{proposition}
\begin{proof}
Notice first that the classes of 
 $$\frac{du_1}{u_1}\wedge\frac{du_2}{u_2}, \frac{1}{u_2}\frac{du_1}{u_1}\wedge\frac{du_2}{u_2}, 
\frac{u_2}{u_1}\frac{du_1}{u_1}\wedge\frac{du_2}{u_2}, \frac{u_2^2}{u_1}\frac{du_1}{u_1}\wedge\frac{du_2}{u_2}$$
are linearly independant in $G_{0}/\theta G_{0}$ because we have
 \begin{itemize}
\item $[F du_{1}\wedge du_2 ]= q_2 [\frac{1}{u_2}du_1\wedge du_2]$
\item $[F^{2}du_{1}\wedge du_2 ] = q_{2}[du_1\wedge du_2 ]+2q_1 q_2[\frac{u_2}{u_1}du_1\wedge du_2 ]$
\item $[F^{3}du_1 \wedge du_2] = 2 q_1(1-4q_{1})[\frac{u_2^{2}}{u_1}du_1\wedge du_2 ]+q_2 (1+4q_1 ) [\frac{1}{u_2}du_1\wedge du_2 ]$
\end{itemize}
in $G_{0} /\theta G_{0}$ and the classes 
 $$[du_1 \wedge du_2 ], [F du_{1}\wedge du_2 ], [F^{2}du_{1}\wedge du_2 ], [F^{3}du_1 \wedge du_2 ]$$
 are linearly independant in $G_{0}/\theta G_{0}$ (see lemma \ref{lemma:NDC} and remark \ref{rem:discriminant}).
 Now, if $\omega_{1},\cdots ,\omega_{\mu}\in G_{0}$ are such that there is no non-trivial relation between their classes 
in $G_{0}/\theta G_{0}$ then there are no non-trivial relations between $\omega_{1},\cdots ,\omega_{\mu}$ in $G_{0}$.
Indeed, assume that
$$\sum_{i=1}^{\mu}a_{i}(\theta , q_{1}, q_{2})\omega_{i}=0$$
in $G_{0}$. Using the assumption, we first get $a_{i}(0, q_{1}, q_{2})=0$ for all $i$ and, because $G_{0}$ has no $\theta$-torsion (see for instance \cite{D1}), 
we get by induction that the 
coefficients of the monomials $\theta^{k}$ in the $a_{i}(\theta , q_{1}, q_{2})$'s are all equal to $0$.
We conclude in particular that the classes $\vartriangle_0 ,\vartriangle_1 ,\vartriangle_2 ,\vartriangle_3 $ are linearly independant in $G_{0}$. 
In particular they generate a free module $H_{0}$ of rank $4$, contained in $G_{0}$. 
The module $H:=\cit [M, \theta ,\theta^{-1}]\otimes H_{0}$ is free over $\cit [M, \theta ,\theta^{-1}]$, equipped with a connection 
(see the second part of theorem \ref{theo:baseomega} below, which is independent of the first one): it follows that $G/H$ is also free
over $\cit [M, \theta ,\theta^{-1}]$, 
because of finite type and equipped with a connection. 
Now, and because of lemma \ref{lemma:NDC}, $G$ is free of rank $4$ over $\cit [M, \theta ,\theta^{-1}]$ (see for instance \cite{DoSa1}). 
We thus have $G=H$. In particular, $H_{0}$ is a lattice in $G$ and we finally get $G_{0}=H_{0}$ because $G_{0}$ is also a lattice in $G$, see {\em loc. cit.} 
\end{proof}

\subsection{A quantum differential system for the Landau-Ginzburg model}
\label{sec:MirSDQF2}

We solve here the Birkhoff problem for the Brieskorn lattice of $F$ (see the Appendix). In other words,
we describe a basis of $G_{0}$ yielding a differential system on $\ppit^{1}\times M$, with logarithmic poles along  $\{ \tau =0\}\times M$ and with poles of Poincare rank less or equal to $1$ along 
$\{ \theta =0\}\times M$. We ask moreover that this basis provides 
 (canonical) logarithmic extensions of the Brieskorn lattice along  $q_2 =0$ and $q_1 =0$. Here, the adjective canonical refers to canonical Deligne's extensions: we require that 
the eigenvalues of the residue matrices along $q_2 =0$ and $q_1 =0$ do not differ from non-zero integers. This explains why we work with a modified 
version of the basis $\triangle$.

\subsubsection{A (non-resonant, logarithmic) differential system}

 Define
\begin{equation}
\omega =(\omega_{0},\omega_{1},\omega_{2}, \omega_{3}):=
(\omega_{0}, -\theta q_{2}\nabla_{\partial_{q_{2}}}\omega_{0},  \theta q_{2}\nabla_{\partial_{q_{2}}}\theta q_{1}\nabla_{\partial_{q_{1}}}\omega_{0}, 
-\theta q_{1}\nabla_{\partial_{q_{1}}}\omega_{0})
\end{equation}
\noindent We will make a constant use 
of the following result which describes the matrix of the connection $\nabla$:

\begin{theorem}\label{theo:baseomega}
\begin{enumerate}
\item $\omega$ is a basis of $G_{0}$ over $\cit [M,\theta ]$.
\item The matrix of $\tau\nabla_{\partial_{\tau}}$ in the basis $\omega$ is
$$-\tau \left (
\begin{array}{cccc}
0     & 2q_2     & 0         & 0\\
2 & 0     & 4q_1 q_2     & 0\\
0     & 4 & 0         & 2\\
0     & 0     & 2q_2 (1-4q_1 ) & 0
\end{array}
\right )
-
\left ( \begin{array}{cccc}
0 & 0  & 0  & 0\\
0 & 1  & 0   & 0\\
0 & 0  & 2  & 0\\
0 & 0  & 0  & 1
\end{array}
\right ),$$
the one of $q_{2}\nabla_{q_{2}}$ is
$$-\tau \left (
\begin{array}{cccc}
0     & q_2       &  0                                   & 0\\
1     & 0           & 2 q_1 q_2                      & 0\\
0     & 2           & 0                                     & 1\\
0     &      0      & q_2 (1-4q_1 )                  & 0
\end{array}
\right )$$

\noindent and the one of $q_1 \nabla_{q_{1}}$ is

$$-\tau \left (
\begin{array}{cccc}
0                   & 0          &  0                     & \frac{q_2  q_1}{1- 4q_1 }\\
0                   & 0          & q_1 q_2            & 0\\
0                   & 1        & 0                     & \frac{2q_1 }{4q_1 -1}\\
1                 & 0          & -2q_1 q_2         & 0
\end{array}
\right )
+
\left ( \begin{array}{cccc}
0 & 0 & 0 & 0\\
0 & 0 & 0           & \frac{q_1}{4q_1 -1}\\
0 & 0 & 0  & 0\\
0 & 0 & 0           & -\frac{2q_1}{4q_1 -1}
\end{array}
\right )$$
\end{enumerate}
\end{theorem}

\begin{proof} 
We have 
$$\triangle =(\omega_{0}, -\theta \nabla_{\partial_{q_{2}}}\omega_{0},  \theta\nabla_{\partial_{q_{2}}}\theta \nabla_{\partial_{q_{1}}}\omega_{0}, 
-\theta\nabla_{\partial_{q_{1}}}\omega_{0})$$
where $\omega_{0}:=[\frac{du_{1}}{u_{1}}\wedge\frac{du_{2}}{u_{2}}]$, by the very definition of the relative connection.
We thus have
$$\omega =(\vartriangle_0 ,q_{2}\vartriangle_1 , q_{1}q_{2}\vartriangle_2  ,
q_{1}\vartriangle_3 )$$ where the basis $\triangle$ is defined in proposition \ref{prop:basisG0} and this shows the first point.
We first describe the matrix of the connection in the basis $\triangle$. In order to do so, we need
the following data:
let us define, for $j=0,1,2$, 
\begin{itemize}
\item $h_{\Gamma_j} =a^j_1 u_1\frac{\partial F}{\partial u_1}+a^j_2 u_2\frac{\partial F}{\partial u_2}-F$ if $\Gamma_j$ has equation $a^j_1 x+a^j_1 y=1$,
\item $\phi_{\Gamma_j}(u_1^{\alpha}u_{2}^{\beta})=
a^j_1 \alpha +a^j_1 \beta$ and $\phi (u_1^{\alpha}u_2^{\beta})= max_j \phi_{\Gamma_j}(u_1^{\alpha}u_2^{\beta})$.
\end{itemize}

\noindent where the faces $\Gamma_{j}$ of $\Gamma$ are defined in the proof of lemma \ref{lemma:NDC}.
One has
$$h_{\Gamma_{0}}=-2u_2 -4q_1\frac{u_{2}^2}{u_1},\ h_{\Gamma_{1}}=-4u_1 -2u_2 ,\ h_{\Gamma_{2}}=-2q_{2}\frac{1}{u_2}$$
and, for instance,
$$\phi (1)=\phi_{\Gamma_2}(1)=0,\ \phi (\frac{1}{u_2})=\phi_{\Gamma_0}(\frac{1}{u_2})=1, \ \phi (\frac{u_2}{u_1})=\phi_{\Gamma_1}(\frac{u_2}{u_1})=2,\ \phi (\frac{u_2^2}{u_1})=\phi_{\Gamma_2}(\frac{u_2^2}{u_1})=1.$$
The map $\phi$ is the ¨Newton degree¨, giving rise to the Newton filtration, closely related with the $V$-filtration and the spectrum at infinity of 
the function $f$ for $(q_{1},q_{2})\in M$, see \cite[Section 4]{DoSa1}.
We then have, in the Gauss-Manin system $G$ of $F$ (keeping in mind that the action of $\nabla_{\partial_{\tau}}$ is induced by the multiplication by $-F$),
\begin{equation}\label{eq:formulebase}
(\tau\nabla_{\partial_{\tau}}+\phi_{\Gamma_{j}}(g))[g\frac{du_1}{u_1}\wedge \frac{du_2}{u_2}]=\tau [h_{\Gamma_{j}}g\frac{du_1}{u_1}\wedge \frac{du_2}{u_2}]
\end{equation}
for any monomial $g=u_1^{a_1}u_2^{a_2}$, $(a_1 ,a_2)\in\zit^2$, where $[\ ]$ denotes the class in $G$.
This formula easily follows from the computation rules in the Gauss-Manin system.
We also have
\begin{itemize}
\item $u_1du_1\wedge du_2=q_1 \frac{u_2^2}{u_1}du_1\wedge du_2+dF\wedge u_1du_2$
\item $u_2du_1\wedge du_2=q_2 \frac{1}{u_2}du_1\wedge du_2-2q_1 \frac{u_2^{2}}{u_1}du_1\wedge du_2+dF\wedge -u_2du_1$
\item $\frac{1}{u_2^{2}}du_1\wedge du_2=\frac{1}{q_{2}}du_1\wedge du_2+2\frac{q_1}{q_2}\frac{u_2}{u_1}du_1\wedge du_2+dF\wedge\frac{1}{q_{2}}du_1$
\item $q_1 \frac{1}{u_1}du_1\wedge du_2=\frac{q_1}{q_{2}}(1-4q_{1})\frac{u_2^{2}}{u_1}du_1\wedge du_2+2q_1 \frac{1}{u_2}du_1\wedge du_2+
dF\wedge [\frac{u_1}{q_2}-\frac{u_1}{u_2^{2}}]du_2
+dF\wedge [\frac{u_1}{q_2}-\frac{2q_1}{q_2}u_2]du_1$
\item $q_1 (4q_1 -1)\frac{u_2^4}{u_1^2}du_1\wedge du_2=2q_1 q_2\frac{u_2}{u_1}du_1\wedge du_2 -q_2 du_1\wedge du_2 +
dF\wedge [-2q_1\frac{u_2^3}{u_1}+u_2^{2}]du_1
+dF\wedge u_2^2 du_2$
\end{itemize}
where $d$ denotes the relative differential $d_{(\cit^* )^2 \times  / M}$,
because  $dF =(1-q_1 \frac{u_2^2}{u_1^2})du_1 + (1+2q_1 \frac{u_2}{u_1}-q_2 \frac{1}{u_2^2})du_2$.
From this we get, using formula (\ref{eq:formulebase}), 
that the matrix of $\tau\nabla_{\partial_{\tau}}$ in the basis $\triangle$ is
$$\tau \left (
\begin{array}{cccc}
0     & -2    & 0         & 0\\
-2q_2 & 0     & -4q_2     & 0\\
0     & -4q_1 & 0         & -2q_2\\
0     & 0     & 2(4q_1-1) & 0
\end{array}
\right )
+
\left ( \begin{array}{cccc}
0 & 0 & 0 & 0\\
0 & -1 & 0 & 0\\
0 & 0 & -2 & 0\\
0 & 0 & 0 & -1
\end{array}
\right )$$
\noindent Similarly, and using the fact that $\theta\nabla_{q_{2}}\omega$  ({\em resp.} $\nabla_{q_{1}}\omega$) is 
$-\frac{\partial F}{\partial q_{2}}\omega$ ({\em resp.} $-\frac{\partial F}{\partial q_{1}}\omega$) for a section $\omega\in\Omega^{n}(U)$, we find that
the matrix of $\nabla_{q_{2}}$ is
$$\tau \left (
\begin{array}{cccc}
0     & -\frac{1}{q_2}      &  0                    & 0\\
-1    & 0                   & -2                    & 0\\
0     & -2\frac{q_1}{q_{2}} & 0                     & -1\\
0     &      0              & \frac{(4q_1-1)}{q_{2}} & 0
\end{array}
\right )
+
\left ( \begin{array}{cccc}
0 & 0                & 0                 & 0\\
0 & -\frac{1}{q_{2}} & 0                 & 0\\
0 & 0                & -\frac{1}{q_{2}}  & 0\\
0 & 0                & 0                 & 0
\end{array}
\right )$$

\noindent (notice that the residue matrix of $\nabla_{q_{2}}$ along $q_{2}=0$ is  resonant that is the difference of two of its eigenvalues is a non-zero integer) 
and that the one of $\nabla_{q_{1}}$ is

$$\tau \left (
\begin{array}{cccc}
0     & 0      &  0                     & \frac{q_2 }{q_1 (4q_1 -1)}\\
0     & 0      & -\frac{q_2}{q_1}         & 0\\
0     & -1     & 0                      & -2\frac{q_2}{4q_1 -1}\\
-1    & 0      & 2                      & 0
\end{array}
\right )
+
\left ( \begin{array}{cccc}
0 & 0 & 0                & 0\\
0 & 0 & 0                & \frac{q_2}{q_1 (4q_1 -1)}\\
0 & 0 & -\frac{1}{q_{1}} & 0\\
0 & 0 & 0                & -\frac{(6q_1 -1)}{q_1 (4q_1 -1)}
\end{array}
\right )$$

\noindent It follows that in the basis $\omega$ the connection has the expected form\footnote{ 
Put $v_{1}=q_{2}^{-1/2}u_1$ and $v_{2}=q_{2}^{-1/2}u_2$. Then
$$F(v_1 ,v_2 )=q_2^{1/2}L(v_1 ,v_2 ,q_1 )$$
where $L(v_1 ,v_2 ,q_1 )= v_1 +v_2 +q_1 \frac{v_2^2}{v_1}+\frac{1}{v_2}$. 
The function $F$ is thus a ``rescaling'' of the function $L$, and these kind of functions yield naturally logarithmic degenerations along $q_2 =0$.
A connected result is that $\tau\nabla_{\partial_{\tau}}=2q_{2}\nabla_{\partial_{q_{2}}}$
(see also remark \ref{rem:EulerVectorField} below).}. 
\end{proof}

\begin{corollary}\label{coro:ConnectionNabla}
Put $\theta :=\tau^{-1}$ and $E:=G_{0}/\theta G_{0}$.
The connection $\nabla$ takes the form
$$\nabla =\bigtriangledown +\frac{\Phi}{\theta} +(\frac{V_{0}}{\theta}+V_{\infty})\frac{d\theta}{\theta}$$
where
\begin{itemize}
\item $\bigtriangledown$  is a flat connection on $E$,
\item $\Phi$ is an ${\cal O}_{M}$-linear map $\Phi: E\rightarrow  E\otimes\Omega_{M}^{1}$, such that $\Phi\wedge\Phi=0$,
\item $V_{0}$ and $V_{\infty}$ are two ${\cal O}_{M}$-linear endorphisms of $E$
\end{itemize} 
\end{corollary}\qed\\

We will call $\bigtriangledown$ the {\em residual} connection. It will play a central role in our perception of mirror symmetry.  
By theorem \ref{theo:baseomega},
the matrix of $\bigtriangledown$ is 
\begin{equation}\label{eq:MatTriangledown}
\left ( \begin{array}{cccc}
0 & 0 & 0 & 0\\
0 & 0 & 0           & 1\\
0 & 0 & 0  & 0\\
0 & 0 & 0           & -2
\end{array}
\right )\frac{dq_1}{4q_1 -1}
\end{equation}
\noindent in the basis of $E$ induced by $\omega$.

\begin{remark} 
\label{rem:RealStructureF2}
The monodromy matrices of the connection $\nabla$ around $\tau =0$, $q_{2} =0$, $q_{1} =0$  in the basis $\omega$ are 
$$M_{\tau }:=\exp (2i\pi  \left (
\begin{array}{cccc}
0     & 0     & 0    & 0\\
-2   & 0     & 0    & 0\\
0     & -4    & 0    & -2\\
0     & 0     & 0    & 0
\end{array}
\right )),\  M_{q_2 }:=\exp (2i\pi \tau  \left (
\begin{array}{cccc}
0     & 0     & 0    & 0\\
-1   & 0     & 0    & 0\\
0     & -2    & 0    & -1\\
0     & 0     & 0    & 0
\end{array}
\right ))$$
and
$$\ M_{q_1}:=\exp (2i\pi \tau  \left (
\begin{array}{cccc}
0     & 0     & 0    & 0\\
0     & 0     & 0    & 0\\
0     & -1    & 0    & 0\\
-1    & 0     & 0    & 0
\end{array}
\right ))$$
respectively\footnote{Notice that $M_{\{q_{2}=0\}}^{2}=M_{\{\tau =0\}}$.}.
Indeed, theorem \ref{theo:baseomega} gives explicit residue matrices: for the two last assertions we can use the non-resonance condition and for the first one 
the fact that $[B_{\infty},B_{0}(0)]=-B_{0}(0)$ if we write the matrix of $\tau\partial_{\tau}$ as $\tau B_{0}(q) +B_{\infty}$. 
Notice that these monodromies are not cyclic.
\end{remark}

\subsubsection{Flattening: the $\bigtriangledown$-flat basis $\varepsilon$}
\label{sec:flat}

In order to get a quantum differential system, we are still looking for a $\nabla$-flat bilinear form $S$, the {\em metric}.  
We first define flat bases with respect to the flat residual connection $\bigtriangledown$: 
indeed, the bilinear we are looking for should be constant 
in such bases, and therefore easier to describe. 
It turns out that these flat bases depend on some choices. A key point is that these choices will be set by the metric. More precisely, using equation (\ref{eq:MatTriangledown}), we get a $\triangledown$-flat basis from $\omega$ {\em via} a multivalued base change, whose matrix is
$${\cal P}(c)=\left (
\begin{array}{cccc}
1  &  0    &  0     & 0\\
0 &   1  &  0 & -\frac{1}{2}(1- 4q_1 )^{1/2}+c\\
0                   & 0    & 1                    & 0\\
0  &  0   &  0 & (1-4q_1 )^{1/2}
\end{array}
\right )
$$

\noindent where\footnote{The parameter $c$ is after all natural: it corresponds to the choice of a basis of $H^{2}(\fit_{2})$, 
see section \ref{sec:MirF2} below.} $c\in\cit$.
Let us define 
$$
\begin{array}{lll}
\varepsilon    & =(\varepsilon_{0},\varepsilon_{1}, \varepsilon_{2}, \varepsilon_{3})         & :=(\omega_{0},\omega_{1},\omega_{2},\omega_{3}){\cal P}(-\frac{1}{2})
\end{array},
$$
see remark \ref{rem:whyc} for an explanation of this choice.

\begin{lemma}\label{lemma:ConnexionBasePlate}

 In the basis $\varepsilon$, the matrix $\tau\nabla_{\partial_{\tau}}$ is
$$-\tau \left (
\begin{array}{cccc}
0       &   2 q_2  & 0                                     & -q_{2}-q_2(1-4q_1)^{1/2}\\
2      &   0       & q_2 +q_2(1-4q_1)^{1/2}               & 0\\
0       &   4      & 0                                     & -2\\
0       &   0       & 2q_2 (1-4q_1)^{1/2}                  & 0
\end{array}
\right )
-
\left ( \begin{array}{cccc}
0 & 0  & 0  & 0\\
0 & 1 & 0  & 0\\
0 & 0  & 2 & 0\\
0 & 0  & 0  & 1
\end{array}
\right ),$$
\noindent the one of $q_2 \nabla_{q_{2}}$ is
$$-\tau \left (
\begin{array}{cccc}
0                 & q_2               &  0                                        & -\frac{1}{2}q_2 (1 -(1-4q_1)^{1/2})\\
1    & 0                & \frac{1}{2}q_2 (1 +(1-4q_1)^{1/2})           & 0\\
0                 & 2   & 0                                         & -1\\
0                 & 0                & q_2 (1-4q_1)^{1/2}                           & 0
\end{array}
\right )$$

\noindent and the one of $q_1 \nabla_{q_{1}}$ is

$$\tau \left (
\begin{array}{cccc}
0                                                                            & 0                          &  0                                             & -q_1 q_2 (1-4q_1 )^{-1/2}\\
-\frac{1}{2}-\frac{1}{2}(1-4q_1)^{-1/2}    & 0                          & q_1  q_2 (1-4q_1 )^{-1/2}             & 0\\
0                                                                            & -1     &  0                                            & \frac{1}{2}( (1-4q_1 )^{-1/2}+1)\\
-(1-4q_{1})^{-1/2}                           & 0                           &  2q_1 q_2 (1-4q_1 )^{-1/2}            & 0
\end{array}
\right )
$$

\end{lemma}
\begin{proof}
Follows from theorem \ref{theo:baseomega} and the definition of the matrix ${\cal P}(-\frac{1}{2})$.  
\end{proof}

\begin{remark}\label{rem:FlatDegrees} 
We define naturally the degree of $\varepsilon_0$ (resp. $\varepsilon_1$, $\varepsilon_2$, 
$\varepsilon_3$) to be $0$ 
(resp. $1$, $2$, $1$).
\end{remark}

\subsubsection{Flat metric} 
\label{sec:DefMetrique}
Let us define, for $a\in\cit^{*}$,
\begin{equation}
\label{eq:MetriquePlateCl}
S(\varepsilon_{0},\varepsilon_{2}):=a,\ S(\varepsilon_{1},\varepsilon_{1})=2a,\ S(\varepsilon_{3},\varepsilon_{1})=-a
\end{equation}
and $S(\varepsilon_{i},\varepsilon_{j})=0$ otherwise.

\begin{lemma} Formulas (\ref{eq:MetriquePlateCl}) provide a bilinear form $S$ on $G_{0}$ by
\begin{equation}
 \label{eq:Metrique}
S(\omega_{0},\omega_{2}):=a,\ S(\omega_{1},\omega_{1})=2a,\ S(\omega_{3},\omega_{3})=\frac{2q_1}{4q_1 -1}a,\ S(\omega_{1},\omega_{3})=a,
\ S(\omega_{2},\omega_{2})=2aq_1 q_2
\end{equation}
\noindent and $S(\omega_{i},\omega_{j})=0$ otherwise, these formulas being extended by 
$\cit [M,\theta ]$-
sequilinearity keeping in mind the involution $j$ alluded to in the introduction. The form $S$ is non-degenerate and $\nabla$-flat.
\end{lemma}
\begin{proof} This result is directly checked: flatness
follows from the symmetry properties of the matrices involved in lemma \ref{lemma:ConnexionBasePlate}.
\end{proof}

\begin{remark}\label{rem:whyc} 
Formula (\ref{eq:MetriquePlateCl}) explains why the normalization $c=-\frac{1}{2}$ in the base change ${\cal P}(c)$ is 
the good one to consider in the mirror symmetry framework, see theorem \ref{theo:miroirF2}.
 For instance, one would have
\begin{equation}
S(\eta_{0},\eta_{2}):=a,\ S(\eta_{1},\eta_{1})=2a,\ S(\eta_{3},\eta_{3})=-\frac{a}{2}
\end{equation}
if $(\eta_{0},\eta_{1}, \eta_{2}, \eta_{3})$ denotes the basis obtained from $\omega$ using the matrix ${\cal P}(0)$.
\end{remark}

\noindent Unless otherwise stated, we will choose $a=1$ in the sequel.

\subsubsection{Flat coordinates}
\label{subsub:FlatCoordinates}

In order to get a precise mirror theorem, we first search for flat coordinates on $M$. 
Define the period map 
\begin{equation}\label{eq:PeriodMap}
\varphi_{\omega_{0}}:\Theta_M \rightarrow \frac{G_{0}}{\theta G_{0}}
\end{equation}
 by 
\begin{equation}
\varphi_{\omega_{0}}(X)=-\Phi_{X}(\omega_{0})
\end{equation}
where $\Phi$ is the Higgs field defined in corollary \ref{coro:ConnectionNabla}.
Notice that $\varphi_{\omega_{0}}$ is injective, see theorem \ref{theo:baseomega}.
We use this map to shift the connection $\bigtriangledown$ to a flat connection 
$\triangledown^{\omega_{0}}$ on $\Theta_{M}$ putting
\begin{equation}
\varphi_{\omega_{0}}(\bigtriangledown^{\omega_{0}}(X))=\bigtriangledown \varphi_{\omega_{0}}(X)
\end{equation}
The flat coordinates $\varphi_{1}$ and $\varphi_{2}$ alluded to are coordinates such that the vector fields $\partial_{\varphi_{1}}$
and $\partial_{\varphi_{2}}$ defined by $\partial_{\varphi_{i}}(d\varphi_{j})=\delta_{ij}$
are $\triangledown^{\omega_{0}}$-flat ($\delta_{ij}$ denotes the Kronecker symbol)\footnote{One could also shift shift the $\triangledown$-flat bilinear form $S$ to a $\triangledown^{\omega_{0}}$-flat bilinear form $S^{\omega_{0}}$ on $\Theta_{M}$ putting
$S^{\omega_{0}}(X, Y)=S(\varphi_{\omega_{0}}(X), \varphi_{\omega_{0}}(Y))$.}.\\


Let us define the vector fields
\begin{itemize}
\item 
$\xi_1 =-(1-4q_{1})^{1/2}q_{1}\partial_{q_{1}}+(\frac{1}{2}(1-4q_{1})^{1/2}
+\frac{1}{2})q_{2}\partial_{q_{2}}$
\item $\xi_2 =q_2 \partial_{q_{2}}$
\end{itemize}


\begin{theorem} 
\label{theo:FlatCoordinate}
The vector fields $\xi_1$ and $\xi_2$ are $\triangledown^{\omega_{0}}$-flat and 
the functions $\varphi_{1}$ and $\varphi_{2}$ defined by 
$$q_{1}=\frac{e^{\varphi_{1}}}{(1+e^{\varphi_{1}})^{2}}\ \mbox{and}\ q_{2}=e^{\varphi_{2}}(e^{\varphi_{1}}+1)$$
are flat coordinates\footnote{Compare with \cite[formula (11.94) p. 394]{CK}}.
\end{theorem}
\begin{proof} 
By the very definition we have 
\begin{equation}\nonumber
\varphi_{\omega_{0}}(\triangledown^{\omega_{0}}q_{1}\partial_{q_{1}})=-\triangledown\omega_{3}\ \mbox{and}\  
\varphi_{\omega_{0}}(\triangledown^{\omega_{0}}q_{2}\partial_{q_{2}})=-\triangledown\omega_{1}
\end{equation}
It follows from theorem \ref{theo:baseomega} and the injectivity of the period map $\varphi_{\omega_{0}}$ that
$$\triangledown^{\omega_{0}}_{\partial_{q_{1}}} q_1\partial_{q_{1}} = (4q_1 -1)^{-1}q_2\partial_{q_{2}} -2(4q_1 -1)^{-1}q_1\partial_{q_{1}}
\ \mbox{and}\ \triangledown^{\omega_{0}}_{\partial_{q_{2}}} q_1\partial_{q_{1}} =0,$$
$$\triangledown^{\omega_{0}}_{\partial_{q_{1}}} q_2\partial_{q_{2}} =0\ \mbox{and}\ \triangledown^{\omega_{0}}_{\partial_{q_{2}}} q_2\partial_{q_{2}}=0.$$
The vector fields $\xi_1$ and $\xi_2$ are thus $\triangledown^{\omega_{0}}$-flat. 
We also have $\partial_{\varphi_{1}}=\xi_{1}$ and $\partial_{\varphi_{2}}=\xi_{2}$ and this gives the second assertion. 
\end{proof}

\begin{remark}
\label{rem:diversplat}
 After theorem \ref{theo:FlatCoordinate} , it is natural to define 
\begin{equation}\label{eq:ri}
r_{1}:=e^{\varphi_{1}}\ \mbox{and}\ r_{2}:=e^{\varphi_{2}}
\end{equation}
We thus have $r_i\partial_{r_i}=\xi_{i}$. We will also call $r_1$ et $r_2$ flat coordinates (a small misuse of language). Using these coordinates,
we can rewrite lemma \ref{lemma:ConnexionBasePlate} as follows :
in the basis $\varepsilon$, the matrix $\tau\nabla_{\partial_{\tau}}$ is
$$-\tau \left (
\begin{array}{cccc}
0       &   2 r_{2}(1+r_{1})  & 0                   & -2r_{1}r_{2}\\
2      &   0       & 2r_{1}r_{2}              & 0\\
0       &   4      & 0                   & -2\\
0       &   0       & 2(r_{1}-1)r_{2} & 0
\end{array}
\right )
-
\left ( \begin{array}{cccc}
0 & 0  & 0  & 0\\
0 & 1 & 0  & 0\\
0 & 0  & 2 & 0\\
0 & 0  & 0  & 1
\end{array}
\right ),$$
\noindent the one of $r_{2}\nabla_{r_{2}}$ is

$$-\tau \left (
\begin{array}{cccc}
0                   &    r_{2}(1+r_{1})                        &  0                            & -r_{2}r_{1}\\
1                   & 0                                             &  r_{1}r_{2}                & 0\\
0                   & 2                                             &  0                            & -1\\
0                   & 0                                             &  (r_{1}-1)r_{2}           & 0
\end{array}
\right )$$

\noindent and the one of $r_{1}\nabla_{r_{1}}$ is

$$-\tau \left (
\begin{array}{cccc}
0                   &    r_{1}r_{2}                    &  0                                     & -r_{2}r_{1}\\
0                   & 0                                                   &  r_{1}r_{2}         & 0\\
0                   & 1                                                 &  0                                     & 0\\
-1                   & 0                                                   &  r_{1}r_{2}         & 0
\end{array}
\right )$$
\end{remark}

\subsubsection{Summary: quantum differential systems}
\label{sec:ResumeSDQ}

The first part of this section yields a trivial bundle ${\cal G}$ on $\ppit^{1}\times M$, 
equipped with a meromorphic connection $\nabla$ with the expected poles and the second part yields a $\nabla$-flat metric $S$, 
where we choose the normalization $a=1$ in formulas (\ref{eq:MetriquePlateCl}). In other words,
the tuple 
$${\cal Q}^{B}_{F}=(M, {\cal G}, \nabla , S, 2)$$
is a quantum differential system on $M$.
  In the same way, using remark \ref{rem:diversplat} and the metric defined in section  \ref{sec:DefMetrique}, we also define 
a quantum differential system ${\cal Q}^{B, qc}_{F}$ on the universal covering of $M$.

\subsection{Applications: a logarithmic Frobenius manifold and a mirror theorem}
\subsubsection{First application: a logarithmic Frobenius manifold}
\label{sec:LogFrob}
We show here how the datum ${\cal Q}^{B}_{F}$ (the {\em initial condition}) provides a logarithmic Frobenius manifold in the sense of \cite{R}. Let
$${\cal L}:=\sum_{i=0}^{3}\cit [q_{1}, q_{2}, (4q_{1}-1)^{-1}, \theta]\omega_{i}$$ 
which is an extension 
of $G_{0}$ along $D:=\{(q_{1},q_{2})\in\cit^{2}|\ q_{1}q_{2}=0\}$ for which the eigenvalues
of the residue matrices are equal to zero, see theorem \ref{theo:baseomega}. 
${\cal L}$ is equipped with a bilinear form $S$ 
 defined by formula (\ref{eq:Metrique}) and naturally  extended by (sesqui-) linearity.
Define $E^{\log}:={\cal L}/\theta {\cal L}$ and the logarithmic version of the period map (\ref{eq:PeriodMap})
\begin{equation} 
\varphi_{\omega_{0}}^{\log}:Der(\log D) \rightarrow E^{\log}
\end{equation} 
by 
\begin{equation}
\varphi_{\omega_{0}}^{\log}(X)=-\Phi_{X}(\omega_{0})
\end{equation}
where $Der(\log D)$ denotes the module of the logarithmic vector fields along $D$.

\begin{lemma}\label{lemma:GCcondition}
Let us denote by $^{o}$ the fiber at $(q_{1},q_{2})=(0,0)$. Then:
\begin{enumerate}
\item the map $\varphi_{\omega_{0}}^{\log ,o}$ is injective,
\item the vector $\omega_{0}^{o}$ of $E^{\log, o}$ and its images under iterations of the maps $\Phi_{X}:E^{\log, o}\rightarrow E^{\log, o}$ generate 
$E^{\log, o}$,
\item the section $\omega_{0}$ is $\triangledown$-flat and homogeneous\footnote{With the notation of corollary \ref{coro:ConnectionNabla}, $\omega_{0}$ is an eigenvector of
$V_{\infty}$ for the eigenvalue $0$.}.
\end{enumerate}
 \end{lemma}
\begin{proof} Use theorem \ref{theo:baseomega}. 
For the two first assertions, notice that $\omega_{3}^{o}$ and $\omega_{1}^{o}$ are linearly independent in $E^{\log, o}$, 
$\Phi_{q_{1}\partial_{q_{1}}}\omega_{0}^{o}=\omega_{3}^{o}$, $\Phi_{q_{2}\partial_{q_{2}}}\omega_{0}^{o}=\omega_{1}^{o}$ and
$\Phi_{q_{1}\partial_{q_{1}}}\circ \Phi_{q_{2}\partial_{q_{2}}}\omega_{0}^{o}=\omega_{2}^{o}$.
The last one is clear.
\end{proof}


Denote by ${\cal G}^{\log}$ the extension of ${\cal L}$ at $\ppit^{1}\times N$ where $N:=(\cit^{2},0)$. Logarithmic quantum differential systems are naturally defined, see for instance \cite[Definition 1.8]{R} where they are called, after \cite{HeMa}, logD-trTLEP structures.
Unfoldings and universal unfoldings of such objects are defined in \cite[definition 1.9]{R}.

\begin{theorem}
The tuple ${\cal Q}^{B, \log}=({\cal G}^{\log}, \nabla , S, N, D)$ 
is a logarithmic quantum differential system. It has a universal unfolding which defines, together
with the $\bigtriangledown$-flat section $\omega_{0}$, a logarithmic Frobenius manifold at the origin of $\cit^{2}$. 
\end{theorem}

\begin{proof}
The first assertion follows from the definition because $S$ is nondegenerate on $N$ (because $S(\omega_1 ,\omega_3 )\in\cit^*$) and the second from 
\cite[Theorem 1.12]{R}, together with lemma \ref{lemma:GCcondition}
which gives the required generation condition in {\em loc. cit.} 
\end{proof}

\noindent To conclude, let us notice that
this construction of logarithmic Frobenius manifold gives in some sense an intermediate step between the one associated with projective space and the one associated with
weighted projective spaces as described in \cite{DoMa}: the monodromies are not cyclic, 
as it is the case for weighted
projective spaces, and the bilinear form $S$ is nondegenerate, as it is the case for projective spaces.

\subsubsection{Second application: a mirror theorem for the small quantum cohomology of $\fit_2$}
\label{sec:MirF2}

The goal of this section is to describe a mirror partner of the small quantum cohomology of $\fit_2$ using quantum differential systems. 
On the quantum cohomology side we keep the notations of  \cite[section 11.2]{CK}.
Recall the quantum differential system ${\cal Q}_{F}^{B, qc}$ (on $M$) defined in section \ref{sec:ResumeSDQ}, using the flat coordinates $(r_{1}, r_{2})$, 
and let ${\cal Q}^{A}_{\fit_{2}}$ (on $M^A$) be the one associated with the small quantum cohomology of $\fit_2$, see \cite{CK}. Let us define the map  
$$\gamma : {\cal Q}^{A}_{\fit_{2}}\rightarrow (id, \nu )^{*}{\cal Q}_{F}^{B, qc}$$
of quantum differential systems in the following way:
\begin{itemize}
 \item  the map $\nu : M\rightarrow M^A$ is the identity,
\item the map $\gamma$ is defined by 
\begin{equation}\nonumber
\gamma (1)=\varepsilon_0,\  \gamma (f)=\theta r_1 \nabla_{r_1}\varepsilon_0 ,
\ \gamma (H)=\theta r_2 \nabla_{r_2}\varepsilon_0 ,
\end{equation}
\begin{equation} \nonumber
\gamma (H\circ f)=(\theta r_2 \nabla_{r_2})(\theta r_1 \nabla_{r_1})\varepsilon_0
\end{equation}
\end{itemize}

\noindent A central point is that, in the original coordinates $(q_{1},q_{2})$, that is if we consider
${\cal Q}^{B}_{F}$ instead of ${\cal Q}_{F}^{B, qc}$,
the map $\nu$ is given by $\nu (q_1,q_2)=(r_1 , r_2 )$.

\begin{theorem}
\label{theo:miroirF2}
The map $\gamma$ is an isomorphism for which
\begin{enumerate}
\item the small quantum poduct is given by \footnote{We can check directly the Frobenius property 
$g_{\fit_{2}}(a\circ b, c)=g_{\fit_{2}}(b, a\circ c)$ 
for any cohomology classes $a$, $b$, $c$ in $H^{*}(\fit_{2})$, using the definition of $\gamma$ and the properties of $S$.}
$$f\circ c=\gamma^{-1}((\theta r_1 \nabla_{r_1 }\gamma (c))_{|\theta =0})\ 
\mbox{and}\ H\circ c=\gamma^{-1}((\theta r_2 \nabla_{r_2}\gamma (c))_{|\theta =0})$$
for any cohomology class $c$ in $H^{*}(\fit_{2})$,
\item the metric $g_{\fit_{2}}$ of ${\cal Q}^{A}_{\fit_{2}}$ is given by
$$g_{\fit_{2}}(a,b)=S(\gamma (a), \gamma (b))$$
for any cohomology classes $a$ and $b$ in $H^{*}(\fit_{2})$ where $S$ is defined by formula (\ref{eq:MetriquePlateCl}).
\end{enumerate}
\end{theorem}
\begin{proof} By remark \ref{rem:diversplat} we have
\begin{equation}\label{eq:GammaIso}
\gamma (1)=\varepsilon_0 ,\ \gamma (f)=\varepsilon_3 ,\ \gamma (H)=-\varepsilon_1 \ \mbox{and}\ 
\gamma (H\circ f)=\varepsilon_2 +r_{1}r_{2}\varepsilon_{0}.
\end{equation}
It follows that the map $\gamma$ is indeed an isomorphism.
The following facts can be found in \cite[section 11.2]{CK} for instance:
\begin{itemize}
\item The matrix of the quantum mutiplication by $f$ in the basis $(1, f, H, f\circ H)$ is 
$$\left (
\begin{array}{cccc}
0     & r_{1}r_{2}    & 0         & 0 \\
1     & 0                       & 0         & 0\\
0     & 0                       & 0         & r_{1}r_{2}\\
0     & 0                       & 1         & 0
\end{array}
\right )$$
\item  the one of the quantum multiplication by $H$ is
$$\left (
\begin{array}{cccc}
0     & 0                       & r_{2}(1-r_{1})         & 0\\
0     & 0                       & 0                                & r_{2}(1-r_{1})\\
1     & 0                       & 0                                & 2r_{1}r_{2}\\
0     & 1                       & 2                                & 0
\end{array}
\right )$$

\item last, the matrix of $g_{\fit_{2}}$ in the same  basis is
$$\left (
\begin{array}{cccc}
0     & 0                       & 0                                & 1\\
0     & 0                       & 1                                & 0\\
0     & 1                       & 2                                & 0\\
1     & 0                       & 0                                & 2r_{1}r_{2}
\end{array}
\right )$$
\end{itemize}
Thus, the assertions follow using the definition of $S$ and remark \ref{rem:diversplat}.

\end{proof}

\begin{remark} 
 1. It follows from theorem \ref{theo:baseomega} (2) that
\begin{equation}\nonumber
\theta q_2 \nabla_{q_2} (\theta q_2 \nabla_{q_2}-2\theta q_1 \nabla_{q_1})\omega_0 =q_2 \omega_0
\end{equation}
This is precisely the differential equation satisfyed by Givental's $I_{\fit_2}$-function, see f.i \cite[page 394, formula 11.96]{CK}.
This gives
\begin{equation}\nonumber
\theta r_2 \nabla_{r_2} ( \theta r_2 \nabla_{r_2} -2\theta r_1 \nabla_{r_1})\varepsilon_0 =
r_2 (1-r_1 )\varepsilon_0
\end{equation}
in flat coordinates and thus
$$H\circ (H -2f) = r_2 (1-r_1)1$$ 
using the isomorphism $\gamma$.\\
2. \label{rem:EulerVectorField}
The restriction of the Euler vector field $E$ at $H^2 (\fit_2 )$ is equal to $2H$. Together with the mirror correspondence, 
this is consistent with the fact that $\tau\nabla_{\partial_{\tau}}=2q_{2}\nabla_{\partial_{q_{2}}}$, see theorem 
\ref{theo:baseomega}.
\end{remark}

\subsubsection{By way of conclusion: rational structures}
\label{sec:Rat}
Let us emphasize that one of the interest of the mirror isomorphism is that it shifts the structures from the $B$-side to the $A$-side and this includes the 
rational structures, see section \ref{sec:RationalStructure}. A rational structure (on the fiber at $\theta =-1$ of the classical limit) 
is a $\qit$-vector space generated by a basis in which the monodromy matrices around $q_{1}=0$ and
$q_{2}=0$ have rational coefficients.
Let us come back for instance to the setting of this paper and
recall the monodromy matrices of remark \ref{rem:RealStructureF2}.
We have,
in the basis $W:=(W_{0},W_{1},W_{2},W_{3})=(\omega_{0}, 2i\pi \tau\omega_{1}, (2i\pi)^{2}\tau^{2}\omega_{2}, 2i\pi\tau \omega_{3})$, 
$$M_{q_{2} }=
\left (\begin{array}{cccc}
1           & 0     & 0     & 0\\
-1          & 1     & 0     & 0\\
1           & -2    & 1     & -1\\
0           & 0     & 0     & 1
\end{array}
\right )\ \mbox{and}\  M_{q_{1} }=
\left( \begin{array}{cccc}
1           & 0     & 0     & 0\\
0           & 1     & 0     & 0\\
0           & -1    & 1     & 0\\
-1          & 0     & 0     & 1
\end{array}
\right)$$
\noindent so that we can eventually define a rational structure\footnote{Probably not a good one; it has for the moment no geometric meaning.} as the 
$\qit$-vector space generated by the basis $W$ for which we have for instance the following conjugation relations:
 $\overline{\omega}_{0}=\omega_{0}$, $\overline{\omega}_{1}=-\omega_{1}$, $\overline{\omega}_{2}=\omega_{2}$ 
and $\overline{\omega}_{3}=-\omega_{3}$.
The mirror isomorphism shifts this rational structure on the cohomology of the Hirzebruch surface and provides the following conjugation relations:
$\overline{1}=1$, $\overline{H}=-H$, $\overline{f}=-f$, $\overline{H\cup f}=H\cup f$. 

\noindent {\em Problem}: describe the rational structure given on the $B$-side by the Lefschetz thimbles on the flat sections of the Gauss-Manin connection 
as in section \ref{sec:RationalStructure} and shift it on the $A$-side.

\subsubsection{By way of conclusion (bis): the quantum cohomology of the weighted projective space $\ppit (1,1,2)$ as a limit, after \cite{CIT}}
\label{sec:ResP112}
We check here, using our framework, that specialization of the previous results at suitable values of the parameters $q_1$ et $q_2$ gives 
the small quantum orbifold cohomology of $\ppit (1,1,2)$, an aspect of Ruan's conjecture \cite{Ru}. 
This has been first done in \cite{CIT}, in a slightly different setting.
 Let us make the following observation on the B-side: put $v_{1}=u_{1}$ and $v_{2}=\frac{q_{2}}{u_{2}}$; our Landau-Ginzburg model becomes
$$v_{1}+v_{2}+\frac{q_{2}}{v_{2}}+\frac{q_{1}q_{2}^{2}}{v_{1}v_{2}^{2}}$$
and thus, in the flat coordinates $(r_{1},r_{2})$
defined by
$$(q_1,q_2)=(\frac{r_1 }{(1+r_1 )^2}, r_2 (1+r_1))$$
(see section \ref{subsub:FlatCoordinates}),
the Landau-Ginzburg model is
$$v_{1}+v_{2}+\frac{r_{2}(1+r_{1})}{v_{2}}+\frac{r_{1}(r_{2})^{2}}{v_{1}v_{2}^{2}}$$
If we set $r_{1}=-1$ and $r_{1}(r_{2})^{2}=q$ we get the usual Landau-Ginzburg model for $\ppit (1,1,2)$, see for instance \cite{DoMa}.
Thus, using the mirror theorem \ref{theo:miroirF2}, we can think the small quantum cohomology of the weighted projective space $\ppit (1,1,2)$ as a limit
of the one of $\fit_2$.

Let us be now more precise. 
Let $g(v_{1},v_{2}, q)=v_{1}+v_{2}+\frac{q}{v_{1}v_{2}^{2}}$
be the Landau-Ginzburg model for $\ppit (1,1,2)$, $\omega^{orb}=(\omega^{orb}_{0},\omega^{orb}_{1},\omega^{orb}_{2},\omega^{orb}_{3})$ be the basis of
its (twisted) Brieskorn lattice defined in \cite[section 4.3.2]{DoMa} : under mirror isomorphism, $\omega^{orb}$ corresponds to the
standard basis $(1,p,p^2 ,1_{1/2})$ of the orbifold 
cohomology $H^{*}_{orb}(\ppit (1,1,2))$ of the weighted projective space $\ppit (1,1,2)$, see {\em loc. cit.}. 
The matrix of $-\theta q\nabla_{\partial_{q}}$ in this basis is
$$\left (
\begin{array}{cccc}
0     &   0    &  0 & \frac{1}{2}q^{1/2}\\
1     &   0    &  0 & 0\\
0     &   1    &  0 & 0\\
0     &   0    &   \frac{1}{2}q^{1/2} & 0
\end{array}
\right ),$$
(on the $A$-side, this is also the matrix of the small quantum multiplication by $p$).
Recall the flat basis $\varepsilon =(\varepsilon_{0},\varepsilon_{1},\varepsilon_{2},\varepsilon_{3})$ in section \ref{sec:flat}.

\begin{proposition}\label{prop:crepant}
The matrix of $-\theta q\nabla_{\partial_{q}}$ in the basis $\omega^{orb}$ is obtained from the matrix
 of $-\frac{1}{2}\theta r_{2}\nabla_{r_{2}}$ in the basis 
$$(\varepsilon_{0},\frac{1}{2}\varepsilon_{1},\frac{1}{2}\varepsilon_{2}, i\varepsilon_{3}+\frac{i}{2}\varepsilon_{1})$$ 
after the transformation
$r_{1}=-1$ and $r_{2}=-iq^{1/2}$. 
\end{proposition}
\begin{proof}By remark \ref{rem:diversplat},
the matrix of $-\frac{1}{2}\theta r_{2}\nabla_{r_{2}}$ 
in the basis 
$(\varepsilon_{0},\frac{1}{2}\varepsilon_{1},\frac{1}{2}\varepsilon_{2}, i\varepsilon_{3}+\frac{i}{2}\varepsilon_{1})$
takes the form

$$\frac{1}{2}\left (
\begin{array}{cccc}
0     &  \frac{1}{2}r_{2}(1+r_{1})                               &  0                                    & \frac{i}{2}r_{2}(1-r_{1})\\
2     &    0                                                               &  \frac{1}{2}r_{2}(r_{1}+1)  & 0\\
0     &   2                                                                & 0                                     & 0\\
0     &   0                                                                & \frac{i}{2}r_{2}(1-r_{1})   & 0
\end{array}
\right )$$

\noindent and the result follows.
\end{proof}

\begin{corollary}\label{coro:crepantlimite}
The matrix of 
the small quantum multiplication $p \circ$ in the basis $(1,p,p^2 ,1_{1/2})$ of $H^{*}_{orb}(\ppit (1,1,2))$ is obtained from the matrix
 of $-\frac{1}{2}H \circ$ in the basis 
$$1, -\frac{1}{2}H, \frac{1}{2}H\circ f -\frac{i}{2}q^{1/2} , if-\frac{i}{2}H$$ 
after the transformation
$r_{1}=-1$ and $r_{2}=-iq^{1/2}$. 
\end{corollary}

\appendix

\section{Appendix: construction of the quantum differential systems associated with regular tame functions (B-side)}
\label{sub:SDQcoteB}

The Laplace transform of the Gauss-Manin connection of a tame regular function on an affine manifold yields quantum differential systems, 
see \cite{DoSa1}, \cite{Sab3}, \cite{D1}, \cite{D2}. We outline here the construction. 

Let $f:U\rightarrow\cit$ be a regular fonction on an affine manifold $U$\footnote{Prototypes : $U=\cit^{n}$ and $U=(\cit^{*})^{n}$}, equipped with coordinates $ \underline{u}=(u_{1},\cdots ,u_{n})$. We consider the differential system (rather than its solutions) satisfied by the Laplace integrals 
$\int_{\Gamma}e^{-f/\theta}\omega$
where $\omega\in\Omega^{n}(U)$ and $\Gamma$ is a Lefschetz thimble \cite{Ph}. This differential system is a meromorphic connection on $\ppit^{1}$ with poles at $\theta =0$ and $\theta =\infty$, that is a free $\cit [\theta , \theta^{-1}]$-module $G$ of finite rank  $\mu$, equipped with a flat connection $\nabla$, the Gauss-Manin connection. We have
$$G=\frac{\Omega^{n}(U)[\theta ,\theta^{-1}]}{(d-\theta^{-1}df)\wedge\Omega^{n-1}(U)[\theta ,\theta^{-1}]}$$
(in other words, we work modulo the exacts forms $d(e^{-\tau f}\omega)$) and the connection  $\nabla$ is defined by 
$$\theta^{2}\nabla_{\partial_{\theta}}(\sum_{i}\omega_{i}\theta^{i})=\sum_{i}f\omega_{i}\theta^{i}+\sum_{i}i\omega_{i}\theta^{i+1},$$
taking into account the kernel $e^{-f/\theta }$.\\

\noindent {\em Step $1$ : construction of a trivial (algebraic) bundle on $\ppit^{1}$.} We need
a free $\cit [\theta ]$-submodule $G_{0}$ in $G$ of maximal rank (in other words, a {\em lattice} in $G$, which gives an extension of $G$ at $\theta =0$) 
and a module {\em opposite} to $G_0$, that is a free $\cit [\tau]$-submodule $G_{\infty}$ (an extension of $G$ at $\theta =\infty$) such that
\begin{equation}\label{eq:oppose}
G_{0}=G_{0}\cap G_{\infty}\oplus \theta G_{0}
\end{equation}
Indeed, we have $G=G_{0}[\tau ]=G_{\infty}[\theta]$: the pair $(G_{0}, G_{\infty})$ defines a bundle ${\cal G}$ on $\ppit^{1}$ 
and the decomposition (\ref{eq:oppose}) shows that this bundle is trivial, see \cite[Chapitre IV, paragraphe 5]{Sab1}.  
It follows from equation (\ref{eq:oppose}) that the restrictions of ${\cal G}$ at $\theta =0$ and $\theta =\infty$ are 
isomorphic {\em via} the global sections $G_{0}\cap G_{\infty}$. 

A natural candidate for $G_{0}$ is
$$G_{0}:=\frac{\Omega^{n}(U)[\theta]}{(\theta d-df)\wedge \Omega^{n-1}(U)[\theta]}.$$
the {\em Brieskorn lattice} of $f$, which is the image of $\Omega^{n}(U)[\theta]$ in $G$. Notice the following important two points:
by definition we have 
$G_{0}/\theta G_{0}=\Omega^{n}(U)/df\wedge\Omega^{n-1}(U)$ and $\theta^{2}\nabla_{\partial_{\theta}}G_{0}\subset G_{0}$. However,
$G_{0}$ is not always {\em free} over $\cit [\theta]$: it will be the case if $f$ is assumed to be {\em tame} \cite{Sab3}, \cite{DoSa1}. 
A basic example of such tame functions are the ( Laurent)
polynomials which are convenient and non-degenerate with respect to their Newton polygons at infinity, 
for which the freeness follows from a division theorem (essentially due to Kouchnirenko \cite{K}).\\

\noindent {\em Step $2$: adding a connection with prescribed poles.} We still need a connection on the trivial bundle ${\cal G}$  
with poles of order less or equal to $2$ at $\theta =0$ and  logarithmic poles at $\theta =\infty$. 
In other words, the matrix of this connection in a basis of global sections should take the form 
\begin{equation}\label{eq:Bir}
(\frac{A_{0}}{\theta}+A_{\infty})\frac{d\theta}{\theta}
\end{equation}
This is the so-called {\em Birkhoff problem} for $G_{0}$. A canonical solution is provided by Hodge theory as follows: first, 
the general statement is {\em The solutions of the Birkhoff problem are in one-to-one correspondence with the opposite filtrations, stable under the
 action of the monodromy, to the Hodge filtration defined on the nearby cycles}, see \cite{DoSa1}, \cite{Sai}. 
In brief, the oppositness gives decomposition (\ref{eq:oppose}) and the stability with respect to the monodromy gives formula (\ref{eq:Bir}). 
Here we use also the classical correspondence between logarithmic lattices and decreasing filtrations, see {\em f.i} \cite[Theorem III.1.1]{Sab1}.
To be precise, let $V_{\bullet}$ be the Kashiwara-Malgrange filtration of $G$ at $\tau =0$ and $H_{\alpha}:=V_{\alpha}G/ V_{<\alpha}G$. 
For $\alpha\in\qit\cap [0,1[$, we define the (Hodge) filtration $F_{\bullet}$ by 
$$F_{p}H_{\alpha}:=(V_{\alpha}G\cap\tau^{p}G_{0}+V_{<\alpha}G)/V_{<\alpha}G$$
where $p\in\zit$.
Because $F_{\bullet}$ is the Hodge filtration of a mixed Hodge structure (see \cite{Sab2}), there exists a decreasing filtration $U^{\bullet}$ of $H_{\alpha}$ 
such that:
\begin{itemize}
 \item for all $p\in\zit$, $N(U^{p}H_{\alpha})\subset U^{p}H_{\alpha}$ where $N$ denotes the nilpotent endomorphism induced by $\tau\partial_{\tau}+\alpha$ on 
$H_{\alpha}$,
\item the filtration $U^{\bullet}$ is a filtration opposite to the filtration $F_{\bullet}$, {\em i.e} $H_{\alpha}=\oplus_{q}F_{q}H_{\alpha}\cap U^{q}H_{\alpha}$.
\end{itemize}
As observed in \cite[Lemma 2.8]{Sai} (a game with Deligne's $I^{pq}$), we can even choose the filtration $U^{\bullet}$ such that $N(U^{p}H_{\alpha})\subset U^{p+1}H_{\alpha}$. 
In this case,the matrix $A_{\infty}$ in equation (\ref{eq:Bir}) is semi-simple, with the expected eigenvalues.
This opposite filtration, built using M. Saito's method, provides a canonical solution of the 
Birkhoff problem, see \cite[Appendix B]{DoSa1}.\\

\noindent {\em Step $3$: the metric.} The Gauss-Manin system $G$ of a {\em tame}, regular function, is self-dual (microlocal Poincar\'e duality, see \cite{Sab3}):
if
$$ G^*=Hom_{\cit [\theta , \theta ^{-1}]} (G,  \cit [\theta ,\theta^ {-1}])\ \mbox{and}\  G_0^*=Hom_{\cit [\theta]}(G_0, \cit [\theta])$$
we have an isomorphism of connections
$ G^*\rightarrow j^{*}G$ which sends $G_0^*$ onto $\theta^{n} j^{*}G_0$. We thus get a non-degenerate bilinear form 
$$ S : G\times j^{*}G\rightarrow\cit [\theta ,\theta^{-1}]$$
such that $S : G_0\times j^{*}G_0\rightarrow \theta^n\cit [\theta ]$. Let us write
$S=\sum_{k\geq n}S_k\theta^k$
on $G_0$: the pairings $S_k$ are called {\em higher residue pairings} (after K. Saito) and $S_n$ is precisely the Grothendieck residue defined on $G_0/\theta G_0$. 
The form
$S$ extends to ${\cal G}$ if there exists a basis $\omega$ of global sections which is {\em adapted} to $S$, {\em i.e} 
\begin{equation}\label{eq:Sadaptee}
S(\omega_{i},\omega_{j})\in\cit \theta^{n}
\end{equation} for all $i,j$. This will be the case if the lattice $G_{\infty}$ alluded to in step 1 and constructed in step 2 is choosen such that
$S : G_{\infty}\times j^{*}G_{\infty}\rightarrow \tau^{-n}\cit [\tau ]$. But this is again provided by the canonical opposite filtration.\\

\noindent {\bf R\'esum\'e of steps 1-3}: we attach a quantum differential system (on a point) $({\cal G}, \nabla ,S , n)$ to any regular, tame
function on $U$.\\

\noindent {\em Step $4$ : adding parameters.} In order to get a bundle on  $\ppit^{1}\times M$, we have to extend the previous situation to a situation ``with  parameters''. We will denote by $\underline{x}=(x_{1},\cdots ,x_{r})$ the coordinates on $M$.\\

\noindent {\em Method 1:} one can
repeat the previous construction, starting with the Gauss-Manin system of an unfolding $F$ of $f$ (see for instance \cite{DoSa1}) and 
taking into account (and in addition) the covariant derivative of the Gauss-Manin connection with respect to the parameters. 
Due to the ``critical points vanishing at infinity'' (see \cite[Examples 2.5]{DoSa1}), this method is in general transcendental, in the parameter axis (always) but also in the $\theta$-axis. The coherence of $G_0^F$ (which is a central point), follows in this setting from standard results in analytic geometry, as in the local ({\em i.e} germ) case. Notice that $\theta\nabla_{\partial_{x_{i}}}G_{0}^{F}\subset G_{0}^{F}$: in a basis of $G_0^F$ is {\em a priori}
\begin{equation}\label{MatCoRel}
\frac{C^{(i)}(x)}{\theta}+ D^{(i)}(x)+\sum_{r=1}^{p}D_{r}^{(i)}(x)\theta^{r}
\end{equation}
and we want the formula $D^{(i)}(x)+\frac{C^{(i)}(x)}{\theta}$ in order to get a quantum differential system. As before, we have $\theta^{2}\partial_{\theta}G_{0}^{F}\subset G_{0}^{F}$.\\ 

\noindent {\em Method 2:} one can use, as in \cite{D1}, \cite{D2} for instance, the Dubrovin-Malgrange-Hertling-Manin reconstruction theorem, see theorem \ref{theo:reconstruction}. The idea is to start with a deformation of $f$ that doesn't produce vanishing critical points at infinity: this is actually what is done for "subdiagram deformations" of a convenient and non degenerate polynomial in \cite{D1}, \cite{D2}.
In some cases these deformations (the ``initial data'') suffice in order to understand universal ones, thanks to the reconstruction theorem quoted above. The advantage now is that we can work algebrically in the variable $\theta$.\\

\noindent {\bf R\'esum\'e of steps 1-4} : summarizing, one associates a 
quantum differential system on $M$ to a regular, tame, function on the affine manifold $U$.

\end{document}